\documentclass[11pt, reqno]{amsart}
\usepackage{blindtext}
\usepackage[left=3cm,right=3cm,top=3.5cm,bottom=2.5cm,letterpaper]{geometry}
\usepackage[english]{babel}
\usepackage{amsthm, amsfonts, amssymb, mathtools}
\usepackage{amsmath}
\usepackage{mathrsfs}
\usepackage{hyperref}
\usepackage{float}
\usepackage{kotex}
\usepackage{ragged2e}
\usepackage{graphicx}
\usepackage{bbm}
\usepackage{enumitem}
\usepackage{centernot}
\usepackage{fancyhdr}
\usepackage{xfrac}
\usepackage{tikz}
\usetikzlibrary{shapes.misc}

\tikzset{cross/.style={cross out, draw=black, minimum size=2.5*(#1-\pgflinewidth), inner sep=2pt, outer sep=0.5pt},
	cross/.default={1pt}}
\usepackage{xcolor}
\usetikzlibrary{calc,arrows}

\newtheorem{theorem}{Theorem}[section]
\newtheorem{lemma}[theorem]{Lemma}
\newtheorem{claim}[theorem]{Claim}
\newtheorem{corollary}[theorem]{Corollary}
\newtheorem{proposition}[theorem]{Proposition}
\newtheorem*{prop-non}{Proposition}
\newtheorem{thm}{Theorem}[section]

\graphicspath{ {images/} }
\numberwithin{equation}{section}
\theoremstyle{definition}
\newtheorem{definition}[theorem]{Definition}
\newtheorem{remark}[theorem]{Remark}
\newtheorem*{remark-non}{Remark}

\begin{document}

\author{Danny Nam and Allan Sly}
\address[Danny Nam]{Department of Mathematics, Princeton University, Princeton, NJ 08540, USA.}
\email{dhnam@math.princeton.edu}
\address[Allan Sly]{Department of Mathematics, Princeton University, Princeton, NJ 08540, USA.}
\email{asly@math.princeton.edu}

\title{Cutoff for the Swendsen-Wang dynamics\\ on the lattice}

\setlength{\parindent}{0pt}

\begin{abstract}
	We study the Swendsen-Wang dynamics for the $q$-state Potts model on the lattice.
	Introduced as an alternative algorithm of the classical single-site Glauber dynamics, the Swendsen-Wang dynamics is a non-local Markov chain that recolors many vertices at once based on the random-cluster representation of the Potts model.
	In this work we derive strong enough bounds on the mixing time, proving that 
	the Swendsen-Wang dynamics on the lattice at sufficiently high temperatures exhibits a sharp transition from ``unmixed" to ``well-mixed," which is called the cutoff phenomenon.
	In particular, we establish that at high enough temperatures the Swendsen-Wang dynamics on the torus $(\mathbb{Z}/n\mathbb{Z})^d$  has cutoff at time $\frac{d}{2 } \left( -\log (1-\gamma) \right)^{-1} \log n $, where $\gamma(\beta)$ is the spectral gap of the infinite-volume dynamics.
\end{abstract}
\maketitle

\vspace{-0.5cm}

\section{Introduction}

A finite ergodic Markov chain is said to exhibit \textit{cutoff phenomenon} if its total-variation distance from stationarity decreases sharply from near its maximum to near 0 over a negligible period of time. In other words, the cutoff phenomenon of a Markov chain describes the \textit{abrupt} transition of the chain from ``unmixed" to ``well-mixed." Here by \textit{abrupt} we mean that this transition takes place during a time period of a strictly smaller order than the mixing time.
Although the cutoff phenomenon is believed to be ubiquitous among a broad class of Markov chains, verifying the existence of cutoff requires a much finer control of the Markov chain than just establishing the order of the mixing time.  Such an analysis proving cutoff typically yields more information as well as it must identify the main obstruction or bottleneck to mixing.

\quad Markov chains, such as the Glauber dynamics give a dynamical version of spin systems that both models the evolution of a physical system and provides an MCMC algorithm to sample from its stationary distribution.  As such their rate of convergence has always been studied in mathematics, physics and computer science communities. While cutoff is expected for many such models and dynamics, establishing it for such  chains is challenging due to their complicated equilibrium distributions. Only recently has which cutoff been verified for a few such chains, particularly the single-site Glauber dynamics on the Ising model at high temperatures (\cite{llp} on the complete graph, and \cite{ls13, ls14, ls15, ls16} on the lattice).

\quad The Swendsen-Wang dynamics is a non-local Markov chain in which a constant fraction of spins change in each step.  For the ferromagnetic Ising model and its $q$-state generalization, the Potts model, we show cutoff on the torus at high enough temperatures, which is the first proof of cutoff for a non-local chain sampling spin-systems.

\begin{thm}\label{thm1}
	Let $d\geq1$ and $q \geq 2$ be fixed integers and consider the Swendsen-Wang dynamics for the ferromagnetic $q$-state Potts model on the torus ${(\mathbb{Z}/n\mathbb{Z})}^d$ at inverse temperature $\beta$. Then there exists $\beta_0 = \beta_0 (d,q) >0$ such that for any $0<\beta < \beta_0$, the dynamics exhibits cutoff at 
	$$t_{\textnormal{mix}}= \frac{d}{2 } \left( \log \left(\frac{1}{1-\gamma}\right) \right)^{-1} \log n $$ with a window of $O(\log \log n)$, where $\gamma$ is the spectral gap of the infinite-volume dynamics.
\end{thm}


\quad Introduced by Swendsen and Wang \cite{swd} in the 1980s, the Swendsen-Wang dynamics was proposed as an alternative of the classical single-site Glauber dynamics. In each step it maps the current spin configuration to a bond percolation configuration on the graph and then colours each component of the configuration uniformly from the $q$ colours (for the formal definitions of the Potts model and the Swendsen-Wang dynamics, see \S\ref{secprelim}).  It has been the preferred algorithm to simulate ferromagnetic spin systems in practice as it is highly efficient in most regimes (see, e.g., \cite{dzw, os, swd}). In the classical case of the $d$-dimensional lattice cube of side-length $n$, it is believed that the Swendsen-Wang dynamics mixes in time of order $\log n$ for all temperatures except at criticality. However, in contrast to the single-site Glauber dynamics, less progress has been made analysing the rate of convergence of the Swendsen-Wang dynamics due to its non-local behavior. 

\subsection{Related works}
There has been a huge effort from mathematics, physics, and computer science communities to understand the rate of convergence of the Swendsen-Wang dynamics. The work started from dealing with special kinds of graphs, proving polynomial bounds on the mixing time at all temperatures. These are on trees, cycles \cite{cooper2}, narrow grids and complete graph \cite{cooper1}, where the latter studied the case of the Ising model (i.e., $q=2$).

\quad In particular, a substantial improvement has been made in the case of the complete graph. Long, Nachmias and Peres \cite{swdtree} established the precise order of the mixing time of the Swendsen-Wang dynamics for the Ising model: $\Theta(1)$ in high temperature, $\Theta(\log n)$ in low temperature, and $\Theta(n^{1/4})$ at criticality. The Potts case ($q\geq 3$) is known to present a different behavior compared to the case of the Ising model. For $q\geq 3$, Galanis, $\check{\textnormal{S}}$tefankovi$\check{\textnormal{c}}$ and Vigoda \cite{gsv} verified that the mixing time exhibits an $\Theta{(n^{1/3})}$-power law at the threshold which is strictly below the ordered/disordered transition point. They also showed the quasi-exponential slowdown of the dynamics in the critical region, extending the result of Gore and Jerrum \cite{gorej}. Recently, Gheissari, Lubetzky and Peres \cite{glp} improved the quasi-exponential lower bound into a pure exponential lower bound of $\exp(cn)$.

\quad On the other hand, if our underlying graph is the lattice cube, less results are known compared to the case of the complete graph. In \cite{bct, bfk}, Borgs et al. showed the exponential lower bound on the mixing time when the dynamics is defined on the $d$-dimensional torus at critical temperature with sufficiently large $q$. Later on, Ullrich \cite{ullrich1, ullrichphd} developed a nice comparison technique that enables to compare the spectral gaps of the Swendsen-Wang dynamics and the heat-bath Glauber dynamics. He proved that the mixing time of the Swendsen-Wang dynamics in $\mathbb{Z}^2$ with $q=2$ is at most polynomial at all temperature, and that if $q\geq3$ then the same estimate holds except at criticality. Guo and Jerrum \cite{gj} extended this result to general dimensions $d$ when $q=2$. At critical temperature in two dimensions, recent work of Gheissari and Lubetzky \cite{gl1, gl2} established a polynomial upper bound for $q=3$, a quasi-polynomial upper bound for $q=4$, and an exponential lower bound for $q>4$, implementing the results of Duminil-Copin et al. \cite{dc1, dc2} who settled the continuity of the phase transition of the two-dimensional Potts model. We finally mention another recent work of Blanca et al. \cite{bcsv} where they demonstrated that the spectral gap of the Swendsen-Wang dynamics on $\mathbb{Z}^2$ has a $\Omega{(1)}$ lower bound in the high temperature regime.

\quad Although the cutoff phenomenon for the Swendsen-Wang dynamics is established in this paper for the first time, there are still a plenty of interesting problems left unanswered. For instance, we conjecture that the dynamics defined on the lattice exhibits cutoff at all temperature except at criticality, in contrast to the current work which only deals with sufficiently high temperatures. In fact, demonstrating an $O(\log n)$ upper bound on the mixing time at general non-critical temperature remains to be an unsolved problem.

\subsection{Main Techniques} In contrast to the monotonicity property present in the Glauber dynamics for the Ising model, the Swendsen-Wang dynamics is non-monotone.  Two previous approaches have been developed for cutoff for single-site Markov chains on spin systems, a spatial decomposition method~\cite{ls13} and information percolation~\cite{ls15}.  The former proof crucially relied on log-Sobolev estimates to bound $L^2$ distances which are not available for the Swendsen-Wang dynamics.  On the other hand, monotonicity is essential to prove cutoff using information percolation.  To overcome this, our approach blends the two techniques, using a new application of information percolation to control the  $L^2$ distance locally and combines this with the coupling of the system to a product chain following the ideas of~\cite{ls13}.  The proof consists of the following major steps which we detail.
\vspace{2mm}


\vspace{1mm}
\noindent $\bullet$ \textit{Coupling the Swendsen-Wang dynamics.} ~As  in the Glauber dynamics for the Ising model, we construct a grand coupling for the Swendsen-Wang dynamics in terms of an ``update sequence" which consists of random variables that describe the evolution of the dynamics. The Swendsen-Wang dynamics is then a deterministic function of its starting configuration and the (random) update sequence.    


\vspace{1mm}
\noindent $\bullet$ \textit{Information percolation and the $L^2$ bound.} 
%
%
~To apply the information percolation framework, we construct the ``history diagram" of the Swendsen-Wang dynamics, by revealing the update sequence backwards in time in the $(d+1)$-dimensional space-time slab, and by connecting each pair of the nearest-neighboring vertices if there exists a possible dependency between the two. Thus it describes the propagation of information backwards in time; formal definitions of the framework will be given in \S4.  At high enough temperatures, the process of the history diagram is stochastically dominated by a sub-critical branching process. We compare the distance to stationarity using Propp and Wilson's coupling from the past method~\cite{cftp} and derive a bound on the $L^2$-distance from stationarity of the chain which turns out to be a crucial quantity in proving cutoff.

\vspace{1mm}
\noindent $\bullet$ \textit{Breaking the dependencies and the $L^1$ to $L^2$ reduction.} ~To establish cutoff in Theorem \ref{thm1}, we implement the technique of \cite{ls13} to relate the $L^1$-mixing time of the dynamics on $\mathbb{Z}_n^d$  in terms of the $L^2$-mixing of the dynamics on a product chain of smaller lattices $\mathbb{Z}_r^d$. Our starting point is observing that the subcriticality of our coupling implies that although the chain is non-local, information does not travel too fast, so that any two distant vertices do not affect each other unless the time period is long enough. Therefore, one might hope to understand the original dynamics on $\mathbb{Z}_n^d$ approximately as a product chain on smaller blocks $\mathbb{Z}_r^d$ by breaking the dependencies between remote sites.  Following the ideas of \cite{ls13}, this is achieved by showing that for most realisations of the update sequence, one only needs the configuration of $X_t$ on a ``sparse'' set of vertices to determine the configuration at time $t+s$ if $s\asymp \log\log n$.  Projected onto this sparse set, called the ``update support,'' the dynamics can be coupled to a product chain of much smaller lattices, and the $L^1$-mixing of the primary chain is controlled by the $L^2$-distance from stationarity of the dynamics defined on $\mathbb{Z}_r^d$ with $r \asymp \log^5 n.$  Finally, we utilize the estimate on the $L^2$-distance derived using information percolation to bound this distance and establish cutoff.


\quad This approach establishes cutoff in terms of the spectral gap on the small chain $\mathbb{Z}_r^d$.  In~\cite{ls13} it was shown that for the Glauber dynamics the cutoff location could in fact be given in terms of the spectral gap of the infinite volume dynamics.  This proof used monotonicity in a crucial way and a follow up work~\cite{ls14} extending the techniques to Potts and other spin systems established the existence of cutoff but could not relate the location to the infinite volume dynamics.  The same issue arises for the Swendsen-Wang dynamics but we are able to relate the spectral gap of $\mathbb{Z}_r^d$ with the infinite volume dynamics by controlling the $L^2$-distance and using the representation of the spectral gap as the exponential rate of convergence to stationarity.


\begin{thm}\label{thm2}
	For $d \geq1$ and $q \geq 2$,  let $\gamma_n$ (resp. $\gamma$) denote the spectral gap of the Swendsen-Wang dynamics on $(\mathbb{Z}/n\mathbb{Z})^d$ (resp. $\mathbb{Z}^d$). Then there exists $\beta_0 = \beta_0(d,q) >0$ such that for any $0<\beta < \beta_0$, we have
	\begin{equation*}
	\lim_{n\rightarrow \infty}  \gamma_n = \gamma ~~~\textnormal{and} ~~~ 0<\gamma<1.
	\end{equation*}
\end{thm}

Our approach works in great generality for dynamics at high temperature including the Potts Glauber dynamics, allowing one to give the cutoff locations in~\cite{ls14}  in terms of the infinite volume spectral gap. We discuss this generalization in Remark \ref{rmkspecgapgen}.

\subsection{Organization}
The rest of this article is organized as follows. \S\ref{secprelim} consists of an introduction on the background. In \S\ref{secglobalcoup}, we introduce a coupling of the Swendsen-Wang dynamics and deduce estimates on the spectral gap and the mixing time. The information percolation framework is explained in \S\ref{secinfoperc}, and the bound on the $L^2$-distance from equilibrium is derived in this section. In \S\ref{secl1l2}, we describe the reduction argument from $L^1$-mixing on $\mathbb{Z}_n^d$ to $L^2$-mixing on a smaller lattice $\mathbb{Z}_r^d$. The final section, \S\ref{secfinal}, is devoted to proving Theorem \ref{thm1} by implementing the results from the previous sections.

\section{Preliminaries}\label{secprelim}

\subsection{The $q$-state Potts model and the random-cluster model}\label{secprelim1} Let $G=(V,E)$ be a finite graph. Let $q\geq 2$ be an integer and $\beta$ be a nonnegative number. Then the $q$\textit{-state Potts model} on $G$ with inverse temperature $\beta$ is the probability distribution on the configuration space $\Omega_V := {\left\lbrace 1,2,\ldots,q \right\rbrace}^V$, where its formula given by 
\begin{equation*}
\pi(\sigma) := \frac{1}{Z_P(\beta,q)} \exp (\beta |E_m(\sigma)|),
\end{equation*}
where $E_m(\sigma) := \left\lbrace (uv) \in E : \sigma(u)=\sigma(v) \right\rbrace$ and $Z_P(\beta,q)$ is the normalizing constant. Each configuration denotes an assignment of colors to the sites in $V$. For $\beta \geq 0$ we say that the model is \textit{ferromagnetic}, otherwise it is \textit{anti-ferromagnetic}. In particular if $q=2$, this is equivalent to the Ising model. Throughout this paper, we will focus only on the ferromagnetic case.

\quad Another model that shows a rich connection with the Potts model is the random-cluster model. Also called as FK-Ising model, this model is introduced by Fortuin and Kasteleyn in \cite{fk1, fk2}. Here, the  configuration space is $\Sigma_E := \left\lbrace 0,1 \right\rbrace^E $, and 
for each configuration $\omega \in \Sigma_E$, we say that an edge $e$ is \textit{open} in $\omega$ if $\omega(e)=1$ and is \textit{closed} otherwise. The  random-cluster model with parameters $p\in [0,1]$ and $q> 0$ is the probability distrubition on $\Sigma_E$ defined by
\begin{equation*}
\phi(\omega) := \frac{1}{Z_{RC}(p,q)}  p^{\vert E(\omega)\vert} (1-p)^{\vert E\setminus E(\omega) \vert}q^{k(\omega)},
\end{equation*}
where $E(\omega)$ is the set of the open edges of $\omega$, $k(\omega)$ denotes the number of connected components in the subgraph $(V, E(\omega))$  (note that we also count each isolated vertex as a component) and $Z_{RC}(p,q)$ is the normalizing constant.

\quad One can observe the relations between the $q$-state Potts model and the random-cluster model by considering the Edwards-Sokal measure \cite{es}, which is the joint distribution of the two defined as
\begin{equation*}
\nu (\sigma,\omega) := p^{\left\vert E(\omega) \right\vert} (1-p)^{\left\vert E \setminus E(\omega)\right\vert } \mathbbm{1}_{\left\lbrace E(\omega) \subset E_m(\sigma) \right\rbrace}.
\end{equation*}

\quad Indeed, one can check that when $p,\beta$ satisfies the relation $p=1-e^{-\beta}$, the marginal distribution of $\nu$ on $\Omega_V$ (resp. $\Sigma_E$) is equal to $\pi$ (resp. $\phi$). A detailed illustration on this fact can be found in \cite{grimmettrc}. Throughout the rest of the paper, we always assume that $p$ and $\beta$ satisfy
\begin{equation*}
p=1-e^{-\beta}.
\end{equation*}

\subsection{The Swendsen-Wang dynamics}\label{secprelim2}
One interesting feature about the Edwards-Sokal measure is that it provides an insight to sample a random cluster configuration from a Potts configuration, and vice versa. This is closely related with the formulation of the Swendsen-Wang dynamics which is first introduced in \cite{swd}. Given a Potts configuration $X_t \in \Omega$, a step of the \textit{Swendsen-Wang dynamics} results in a new configuration $X_{t+1}$ as follows:	
\vspace{2mm}

\begin{enumerate}
	\item Sample $\omega_t \in \Sigma_E$ by setting $\omega_t(e) = 1 $ with probability $p=1-e^{-\beta}$ and $\omega_t(e) = 0$ with probability $1-p$ for each $e\in E_m(X_t)$, independently of $e$. For $e \notin E_m(X_t)$, set $\omega_t (e)=0$. Hence we obtain a joint configuration $(X_t, \omega_t)$;
	\vspace{2mm}
	
	\item Assign to each connected component of $(V,E(\omega_t))$ independently a new color from $Q= \left\lbrace 1,2,\ldots,q \right\rbrace$ uniformly at random and obtain the new Potts configuration $X_{t+1}$.
	
\end{enumerate}
\vspace{1mm}

\quad Not only is it well-known, but also is a simple fact that the Markov chain defined as above is reversible and stationary with respect to the Potts measure. Similarly, if we run the dynamics by 2$\rightarrow$1 starting from an edge configuration $\omega_t$, then it defines a reversible Markov chain with respect to the random-cluster measure. 



\subsection{Mixing time, cutoff and spectral gap}\label{secprelim3}
The total-variation ($L^1$) distance is arguably the most fundamental notion of convergence in the theory of Markov chains. For two probability measures $\mu_1, \mu_2$ on a finite state space $S$ the \textit{total-variation distance} is defined as
\begin{equation*}
\| \mu_1 - \mu_2 \|_{{ \textnormal{\tiny TV}}} := \max_{A \subset S} \vert \mu_1 (A) - \mu_2 (A)\vert = \frac{1}{2} \sum_{x\in S} \vert \mu_1 (x) - \mu_2(x)\vert.
\end{equation*}
For an ergodic Markov chain $(Y_t)$ with stationary distribution $\mu$, we define the worst case total-variation distance from equilibrium as
\begin{equation*}
d(t) := \max_{y_0 \in S} \| \mathbb{P}_{y_0} ( Y_t \in \cdot) - \mu \|_{{ \textnormal{\tiny TV}}},
\end{equation*}
where $\mathbb{P}_{y_0} ( Y_t \in \cdot)$ denotes the probability distribution of $Y_t$ starting from $y_0$. Then, \textit{the mixing time} of the chain $(Y_t)$ is defined as the minimal time when $d(t)$ gets below some given threshold, i.e., for each $\epsilon \in (0,1)$,
\begin{equation*}
t_{\textnormal{mix}} (\epsilon) := \min \left\lbrace t \geq 0 : d(t) \leq \epsilon \right\rbrace.
\end{equation*}

\quad A family of chains $\{(Y_t^{(n)}) \}_n$ is said to exhibit \textit{cutoff} if for every fixed $\epsilon \in (0,1)$ we have 
$$\lim_{n\rightarrow \infty} \frac{t^{(n)}_{\textnormal{mix}}(\epsilon)}{  t^{(n)}_{\textnormal{mix}}(1-\epsilon)} = 1.$$
Here, we typically consider the family $\{(Y_t^{(n)}) \}_n$ that consists of the same type of Markov chains whose system size grows in $n$ (e.g., the Swendsen-Wang dynamics on $\mathbb{Z}_n^d$).
A sequence $(w_n)$ is said to be a \textit{cutoff window} if $t^{(n)}_{\textnormal{mix}}(\epsilon) - t^{(n)}_{\textnormal{mix}}(1-\epsilon) = O(w_n)$ for every $\epsilon \in (0, 1)$.
Then, the existence of cutoff is equivalent to the existence of such sequence $(w_n)$ that satisfies $w_n = o(t^{(n)}_{\textnormal{mix}}(\epsilon))$.

\vspace{2mm}
\quad For a discrete-time reversible Markov chain $(Y_t)$, the transition matrix $P$ of the chain  has real eigenvalues which we denote by $1=\lambda_1  \geq \ldots \geq \lambda_{|S|}\geq -1$. Then the \textit{spectral gap} of the Markov chain $(Y_t)$ is defined as
\begin{equation*}
\gamma := 1- \max \left\lbrace \lambda_2, |\lambda_{|S|}| \right\rbrace.
\end{equation*}
Spectral gap of a Markov chain provides some fundamental results on the mixing time. We point out a well-known property of it as follows. For a proof, see, e.g., \cite{lpwmcmt}.

\begin{proposition}{\label{specgap}}
	Let $(Y_t)$ be a discrete-time, ergodic and reversible Markov chain with stationary distribution $\mu$ and spectral gap $\gamma$. Define $d(t) = \max_{y_0 \in S} \| \mathbb{P}_{y_0} ( Y_t \in \cdot) - \mu \|_{{ \textnormal{\tiny TV}}}$. Then the following holds true for all $t>0$:
	\begin{equation*}
	(1-\gamma)^t \leq 2d(t) \leq \mu_{\textnormal{min}}^{-1} (1-\gamma)^t,
	\end{equation*}
	where $\mu_{\textnormal{min}} = \min_{x\in S} \mu(x)$.
\end{proposition}

\quad Moreover, if  $P$ is non-negative definite, then the spectral gap $\gamma = 1- \lambda_2$ can be written by the following formula using variational approach: 
\begin{equation}\label{specgapvar}
\gamma : = \inf_{\substack{ f \in L^2(\pi) \\ f \neq 0} } \frac{\mathcal{E}_\pi (f,f)}{\textnormal{Var}_\pi (f) },
\end{equation}
where $\mathcal{E}_\pi: L^2(\pi) \times L^2(\pi) \rightarrow \mathbb{R}$ denotes the Dirichlet form defined as
\begin{equation*}
\mathcal{E}_\pi (f,g) : = \frac{1}{2} \int_{\Omega \times \Omega} (f(x)-f(y))^2 P(x,dy) \pi(dx).
\end{equation*}


\section{Global Coupling of the Swendsen-Wang Dynamics}\label{secglobalcoup}

Throughout this section, $d\geq 2$ will be any fixed integer, and $G=(V,E)$ will be a finite graph of maximal degree $d$. 

\quad The purpose of the following subsection is to define a global coupling for the Swendsen-Wang dynamics. This coupling method gives a simple proof of the constant lower bound of the spectral gap  in \S\ref{secglobalcoup2}.

\subsection{A global coupling for the Swendsen-Wang dynamics}\label{secglobalcoup1} We introduce the update sequence of the Swendsen-Wang dynamics which consists of three types of elements that determine the updates.
For an edge configuration $\omega \in \Sigma_E$ on $G$, let $(V, \omega)$ denote the subgraph of $G$ induced by the edge set $\{e\in E : \omega(e) =1 \}$.

\begin{definition}[Update sequence]\label{updseq}
	Let $(X_t)_{0\leq t\leq t_{\star}}$ be the Swendsen-Wang dynamics for the $q$-state Potts model on $G=(V,E)$. The \textbf{update sequence} of $(X_t)_{0\leq t\leq t_{\star}}$ is defined by  $\mathfrak{H}_{t_{\star}} = \left\lbrace (\bar{\omega}_t, (c_{v,t})_{v \in V} , \mathcal{A}_{\bar{\omega}_t} ) \right\rbrace_{t=0}^{t_{\star}-1}$, where the elements of $\mathfrak{H}_{t_{\star}}$ are given as follows:
	
	\begin{enumerate}
		\item Let $(\bar{\omega}_t)_{t\geq 0}$ be the collection of i.i.d. Bernoulli (bond-)percolation configurations on $G=(V,E)$ with probability $p$. In other words, for each edge $e$, $\bar{\omega}_t(e)$ is set to $1$ with probability $p$ and to 0 with probability $1-p$ independently of $e$, and $\bar{\omega}_t$'s are independent. \vspace{1mm}
		
		\item Let $(c_{v,t})_{v\in V, t\geq 0}$ be i.i.d. \textnormal{Unif}$\left\lbrace{1,\ldots,q}\right\rbrace$ random variables which are independent with $(\bar{\omega}_t)_{t\geq 0}$.\vspace{1mm}
		
		\item For each $t\geq 0$, let $k(\bar{\omega}_t)$ be the number of connected components in  $(V, \bar{\omega}_t)$ and let $C_1^{\bar{\omega}_t}, \ldots, C_{k(\bar{\omega}_t)}^{\bar{\omega}_t}$ denote its components. For each $j \in \left\lbrace 1,\ldots,k(\bar{\omega}_t) \right\rbrace$, define $\alpha_t(C_j^{\bar{\omega}_t})$ to be any bijective function that maps $C_j^{\bar{\omega}_t}$ onto $\left\lbrace 1,\ldots, |C_j^{\bar{\omega}_t}|\right\rbrace$. 
		Then, we combine all the information of $\alpha_t(C_j^{\bar{\omega}_t})$'s, by defining $\mathcal{A}_{\bar{\omega}_t} : V \rightarrow \mathbb{N}$ as
		\begin{equation*}
		\mathcal{A}_{\bar{\omega}_t} ( v) = \alpha_t(C_j^{\bar{\omega}_t})(v) ~~~~\textnormal{  if } v\in C_j^{\bar{\omega}_t} \textnormal{ for some } j.
		\end{equation*}
	\end{enumerate}
\end{definition}

\quad Note that in step 3, the specific choice of the function $\alpha_t(C_j^{\bar{\omega}_t})$ is unimportant. Any function that is bijective from $C_j^{\bar{\omega}_t}$ onto $\left\lbrace 1,\ldots, |C_j^{\bar{\omega}_t}|\right\rbrace$ leads to the desired coupling. However, the functions $\alpha_t(C_j^{\bar{\omega}_t})$'s should be deterministic.

\quad It turns out that the combination of the three types of random variables constructed above can actually govern the evolution of the Swendsen-Wang dynamics. We describe how it is done in the following definition.

\begin{definition}[Global coupling]\label{globalcoupling}
	Let $\mathfrak{H}_{t_{\star}} = \left\lbrace (\bar{\omega}_t, (c_{v,t})_{v \in V} , \mathcal{A}_{\bar{\omega}_t} )  \right\rbrace_{t=0}^{t_\star-1}$. The the Swendsen-Wang dynamics $(X_t)_{0\leq t\leq t_\star}$ for the $q$-state Potts model on $G$ is coupled with $\mathfrak{H}_{t_\star}$ as follows:
	
	\begin{enumerate}
		\item Given the Potts configuration $X_t$ at time $t$, the corresponding edge configuration $\omega_t$ in the first step of the dynamics follows $\bar{\omega}_t$ on monochromatic edges, i.e., $\omega_t(e) = \bar{\omega}_t (e)$ if $X_t(u) = X_t(v)$ with $e=(uv)$, and $\omega_t(e) = 0$ otherwise.
		\vspace{1mm}
		
		\item Let $C_1^{\omega_t},\ldots, C_{k(\omega_t)}^{\omega_t}$ denote the connected components of $(V,\omega_t)$. For each  $C_j^{\omega_t}$, pick a vertex $v_j\in C_j^{\omega_t}$ that satisfies $\mathcal{A}_{\bar{\omega}_t} (v_j) = \min \lbrace \mathcal{A}_{\bar{\omega}_t} (u) : u \in C_j^{\omega_t} \rbrace$, i.e., for each connected component of $\omega_t$, we pick the vertex having the smallest label with respect to $\mathcal{A}_{\bar{\omega}_t}$.
		\vspace{1mm}
		
		\item Obtain $X_{t+1}$ by assigning to each component $C_j^{\omega_t}$  a new color $c_{v_j,t}$.
	\end{enumerate}
\end{definition}

\quad Let us briefly check how this procedure is actually identical with the law of the Swendsen-Wang dynamics. Firstly, since $\bar{\omega}_t \sim$ i.i.d.$\,$Perc$(G,p)$, generating the edge configuration $\omega_t$ in Step 1 is indeed the same in law as the first step of the Swendsen-Wang dynamics. In the second and third steps, no matter how $\alpha_t $ and $\mathcal{A}_{\bar{\omega}_t}$ are defined, each connected component in $(V,\omega_t)$ receives a color $c_{v_j,t} \sim$ i.i.d.$\,$Unif$\left\lbrace 1,\ldots,q \right\rbrace$, matching the definition of the Swendsen-Wang dynamics.


\begin{remark}
	\textnormal{Unlike the monotone coupling of the heat-bath Glauber dynamics for the Ising model or for the random-cluster model, the coupling is not monotone in the following sense:  Consider the edge Swendsen-Wang dynamics on $\Sigma_E$ by proceeding $2 \rightarrow 3 \rightarrow 1$ instead of $1 \rightarrow 2 \rightarrow 3$ in the Definition \ref{globalcoupling}. If we define the order between the edge configurations by  $\omega, \, \omega' \in \Sigma_E$, $\omega \leq \omega'$ if and only if $\omega(e) \leq \omega'(e)$ for all $e\in E$, then one can observe that $\omega_0 \leq \omega_0' $ does not imply $\omega_1 \leq \omega_1'$.  	}
\end{remark}

\subsection{Lower bound on the spectral gap}\label{secglobalcoup2} In this subsection, we implement the global coupling given in the Definition \ref{globalcoupling} to prove a constant lower bound on the spectral gap of the Swendsen-Wang dynamics. The proof will be an application of the path coupling method which is first introduced by Bubley and Dyer \cite{pathcoupling}. In this procedure, an upper bound on the mixing time will natuarally be derived as well.


\begin{proposition}\label{prop1}
	Let $p=1-e^{-\beta}$ and $q\geq 2$. For any $p$ such that $edp \leq 1 - \frac{1}{\sqrt{2}}$, the spectral gap $\gamma$ of the Swendsen-Wang dynamics for the ferromagnetic $q$-Potts model on $G$ at inverse temperature $\beta$ satisfies $\gamma \geq 1- 2edp$. Moreover, the mixing time has an upper bound
	\begin{equation*}
	t_{\textnormal{mix}} \left( 1/2e \right) \leq \frac{\log(2en)}{\log(1/2edp)}.
	\end{equation*}
\end{proposition}

\begin{proof}
	With Proposition \ref{specgap} in mind, the following lemma directly implies Proposition \ref{prop1}.
\end{proof}

\begin{lemma}\label{gapofswd}
	Let $(X_t)$ be the Swendsen-Wang dynamics on $G=(V,E)$ with $|V|=n$ and maximal degree bounded by $d$. Then for any $p$ such that $edp \leq 1- \frac{1}{\sqrt{2}}$ and $t>0$,
	\begin{equation*}
	\max_{x_0} \| \mathbb{P}_{x_0} \left( X_t \in \cdot \right) - \pi \|_{\textnormal{\tiny TV}} \leq n (2edp)^t.
	\end{equation*}
\end{lemma}

\begin{proof}
	Let $(X_t)$ and $(X_t') $ be two copies of the Swendsen-Wang dynamics that are coupled according to the global coupling in Definition \ref{globalcoupling}. Let $\omega_t$ and $\omega_t'$ be the edge configuarations generated after the first step of the dynamics, corresponding to $X_t$ and $X_t'$, respectively. Also, let $d(X_t, X_t')$ be the Hamming distance between the two configurations, i.e., 
	\begin{equation*}
	d(X_t, X_t') = \vert \{ u \in V : X_t(u) \neq X_t'(u) \} \vert.
	\end{equation*}
	
\quad 	Suppose that at time $t$, the two configurations satisfy $d(X_t, X_t') =1$. Let $v$ be the vertex such that $X_t(v) \neq X_t'(v)$. Let $\bar{\omega}_t$ be the percolation configuration at time $t$ included in the update sequence and let $\bar{C}$ be the connected component of $(V, \bar{\omega}_t)$ containing $v$.
	
	\quad At the second step of the coupling, every connected component of $\omega_t$ that is not contained in $\bar{C}$ receives the same color as the corresponding component of $\omega_t'$, since $\omega_t$ and $\omega_t'$ are the same except at $N(v) = \{e\in E: \exists u,~ e=(uv) \}$.
	This implies that $X_{t+1}(V\setminus \bar{C}) = X_{t+1}'(V \setminus \bar{C})$. Moreover, if we let $u$ to be the vertex in $\bar{C}$ such that $\mathcal{A}_{\bar{\omega}_t} (u) =1 $, then $X_{t+1}(u) = X_{t+1}'(u)$ by the definition of the global coupling. As a result, we deduce the following inequality:
	\begin{equation}\label{pathcoupeq}
	\mathbb{E} \left[ d(X_{t+1}, X_{t+1}') \vert d(X_t, X_t') =1 \right] \leq \mathbb{E} [ \vert \bar{C} \vert -1].
	\end{equation}
	
	\quad We bound the right hand side of (\ref{pathcoupeq}) by dominating our graph $G$ by its cover tree. Let $(\mathbb{T}_d, \bullet)$ be the infinite $d$-regular tree rooted at $\bullet$. Consider a Bernoulli percolation with edge inclusion probability $p$ on $(\mathbb{T}_d, \bullet)$, and let $\Gamma$ be the connected component of the percolation containing $\bullet$. Since the maximal degree of our graph $G$ is $d$, we can consider a projection $\varphi$ from  $(\mathbb{T}_d, \bullet)$ onto $(G,v)$, i.e., $\varphi(\bullet) = v$ and $\varphi$ preserves the edge relations of $\mathbb{T}_d$ and $G$. Then there is a natural coupling between $\Gamma$ and $\bar{C}$ via this projection, and the size of $\Gamma$ stochastically dominates the size of $\bar{C}$.
	
	\quad Hence we have $\mathbb{E} [ \vert \bar{C} \vert -1] \leq \mathbb{E} [ \vert \Gamma \vert -1] = \mathbb{E} [ \vert E(\Gamma) \vert ] $, where $E(\Gamma)$ is the set of edges in $\Gamma$. For each $k$, we have $\mathbb{P}(\vert E(\Gamma) \vert = k) \leq (edp)^k$, since we should select $k$ edges to be open (probability $p^k$) and the number of choosing a $k$-subtree containing $\bullet$ is at most $ {dk \choose k} \leq (ed)^k $. Thus, our assumption $edp \leq 1- 1/\sqrt{2}$ implies
	$\mathbb{E}[\vert E(\Gamma)\vert] \leq \sum_{k\geq 1} k(edp)^k \leq 2edp$,
	and hence
	\begin{equation}\label{pathcoupineq}
	\mathbb{E} \left[ d(X_{t+1}, X_{t+1}') \:\vert \: d(X_t, X_t') =1 \right] \leq 2edp.
	\end{equation}
	
	\quad Then the path coupling argument (see, e.g., \cite{lpwmcmt}) implies that (\ref{pathcoupineq}) can be extended to the case when $d(X_t,X_t')$ is arbitrary. Hence we deduce that
	\begin{equation*}
	\mathbb{E} \left[ d(X_{t}, X_{t}')\: \vert \: d(X_0, X_0')  \right] \leq d(X_0, X_0')(2edp)^t.
	\end{equation*}
	
	\noindent Finally, the basic coupling inequality and Markov's inequality imply that
	\begin{equation*}
	\max_{x_0} \| \mathbb{P}_{x_0} \left( X_t \in \cdot \right) - \pi \|_{\textnormal{\tiny TV}} \leq \mathbb{E}[d(X_t, X_t')] \leq n (2edp)^t,
	\end{equation*}
	which concludes the proof.
\end{proof}

\section{Information Percolation and Estimating the $L^2$ Distance}\label{secinfoperc}
In this section, we analyze the distance from stationarity measured by the $L^2$-norm. Detailed understanding on the $L^2$ distance will turn out to be crucial in establishing cutoff for the Swendsen-Wang dynamics. Indeed in \S\ref{secl1l2}, we will see that the total-variation ($L^1$) distance from stationarity can essentially be controlled by the $L^2$-distance of the chain defined on a smaller lattice $\mathbb{Z}_r^d$ from its stationarity.

\quad In the previous work \cite{ls13} on the Glauber dynamics for the Ising model, a similar $L^1$ to $L^2$ reduction technique is used to prove cutoff, and estimating the aforementioned $L^2$-distance is an important issue there as well. In \cite{ls13},  the condition of \textit{strong spatial mixing} is assumed, which provides a strong enough control on the $L^2$-distance based on the results from \cite{dsc96}.


\quad However in the case of the Swendsen-Wang dynamics, following the same program seems to be difficult (see Remark \ref{rmklogsob}), and hence we need a different approach. We implement the concept of \textit{information percolation} which is first introduced in \cite{ls15, ls16}. To be specific, we reveal the update sequence backwards in time to develop the \textit{history diagram}, whose purpose is to describe the information flow on the space-time slab. In this section, we explain how this framework is applied to the Swendsen-Wang dynamics and deduce an exponential decay of the $L^2$-distance.


\quad Throughout this section, the underlying graph $G=(V,E)$ will be a degree-$d$ \textit{transitive} graph on $n$ vertices. In other words, we will work with a $d$-regular graph with a nice symmetry such that for any two vertices $u$ and $v$ of $G$, there exists a graph isomorphism $f:G\rightarrow G$ that maps $u$ to $v$ and preserves the graph structure.

\subsection{Information percolation for the Swendsen-Wang dynamics}\label{secinfoperc1}

\quad Recall that the update sequence for the Swendsen-Wang dynamcis from time $0$ to $t_*$ is defined by $\mathfrak{H}_{t_{\star}} = \left\lbrace (\bar{\omega}_t, (c_{v,t})_{v \in V} , \mathcal{A}_{\bar{\omega}_t} )  \right\rbrace_{t=0}^{t_\star-1}$. According to Definition \ref{globalcoupling}, $\mathfrak{H}_{t_\star}$ and the initial condition determines the dynamics at time $t \leq t_\star$. We first introduce the notion of \textit{oblivious vertex} as follows.

\begin{definition}[Oblivious vertices] Given the update sequence $\mathfrak{H}_{t_\star} $ of the Swendsen-Wang dynamics from time $0$ to time $t_\star$, we say that $v$ is an \textbf{oblivious vertex} at time $t$ if $v$ is an isolated vertex in $(V, \bar{\omega}_t)$. Otherwise, $v$ is said to be \textbf{non-oblivious} at time $t$.
\end{definition}

\begin{remark}
	We have two simple observations on (non-)oblivious vertices as follows.
	\begin{enumerate}
		\item [1.] The term \textit{oblivious} comes from the following observation:
		If $u$ is an isolated vertex in $V,\bar{\omega}_t$ at time $t$, then $X_{t+1}(u)$ becomes independent of the initial state $X_0$. Hence, at $u$,
		it forgets all the information from the past when it proceeds to time $t+1$ from $t$.
		
		\item[2.] On the other hand, for a non-oblivious vertex $u$ at time $t$, we should also look at $X_t$ in order to determine $X_{t+1}(u)$. To be precise, if $C$ is the connected component of $(V,\omega_t)$ containing $u$, then $X_{t+1}(u)$ is determined by $C$, $(c_{v,t})_{v \in C}$ and $\mathcal{A}_{\bar{\omega}_t}$. If  $\bar{C}$ is the connected component of $(V, \bar{\omega}_t)$ containing $u$, then $C \subset \bar{C}$ and $C$ is determined by $\bar{\omega}_t$ and $X_{t}(\bar{C})$. Therefore, we conclude that $X_{t+1}(u)$ is possibly dependent on the colors of sites in $\bar{C}$ at time $t$, but independent on the colors of the rest.
	\end{enumerate}
\end{remark}




\quad Based on the observation, we can develop the \textit{history diagram}  of the Swendsen-Wang dynamics, which describes the information flow backwards in time. We will consider a space-time diagram on the underlying domain $V \times \{\frac{k}{2} : k=0,1,\ldots,2t_\star \}$. A layer at an integer time, say $t$, will describe the history at time $t$, and a layer at a half-integer time, say ${t + \frac{1}{2}}$, will contain the information of the edge configuaration $\bar{\omega_t}$.

\quad Let $\mathfrak{G}=(\mathcal{V},\mathcal{E})$ denote the graph with the vertex set $\mathcal{V} = V \times \{\frac{k}{2} : k=0,1,\ldots,2t_\star \}$ and the edge set which is defined as follows: $(u,t), \, (v,s) \in \mathcal{V} $ are adjacent if and only if we either have $(uv)\in E$ and $s=t$, or $u=v$ and $|t-s| = \frac{1}{2}$. In other words, the edges of $\mathfrak{G}$ are the nearest neighbors of $\mathcal{V}$.

\begin{definition}[The history diagram]\label{historydiag}
	Let $(X_t)_{0\leq t \leq t_\star}$ be the Swendsen-Wang dynamics on $G$, and suppose that
	$\mathfrak{H}_{t_{\star}} = \left\lbrace (\bar{\omega}_t, (c_{v,t})_{v \in V} , \mathcal{A}_{\bar{\omega}_t} )  \right\rbrace_{t=0}^{t_\star-1}$, the update sequence for $(X_t)$, is given. For each $v\in V$, the \textbf{history diagram }of $v$ (in short, the history of $v$) is the connected subgraph $\mathscr{H}_v$ of $\mathfrak{G}$ defined by the following recursive procedure that starts at  the vertex $(v,t_\star) $: \vspace{1mm}
	
	\begin{enumerate}
		\item[0.] $(v,t_\star)$  is the unique vertex of $\mathscr{H}_v$ at time $t_\star$.
		
		\item[1.] At time $t \in \mathbb{N}$, connect $(u,t)$ with $(u,t-\frac{1}{2}) $ with an edge for all $u$ such that $(u,t  ) \in \mathscr{H}_v$, and include $(u, t-\frac{1}{2})$ as a vertex of $\mathscr{H}_v $.
		
		\item[2.] Let $u$ be such that $( u,t-\frac{1}{2}) \in \mathscr{H}_v$, and let $C$ be the connected component of $(V,\bar{\omega}_{t-1})$ that contains $u$. If $u$ is non-oblivious at time $t-1$, then include every vertex and edge of $C$ in $\mathscr{H}_v$ as a vertex and an edge at time $t-\frac{1}{2}$, respectively. Note that edges of $C$ are given according to $\bar{\omega}_{t-1}$.) 
		
		\item[2$'$.] If $u$ satisfies $( u,t-\frac{1}{2}) \in \mathscr{H}_v$ but is oblivious at time $t-1 $, we do not take any more action for this vertex, i.e., we stop branching from $(u, t-\frac{1}{2})$. 
		
		\item[3.] For each $u$ such that $( u,t-\frac{1}{2}) \in \mathscr{H}_v$, connect $( u,t-\frac{1}{2})$ and $(u,t-1)$ with an edge if and only if $u$ is non-oblivious at time $t-1$.
		
		\item[4.] Return to step 1 with time set to be $t-1$ and repeat the process until there are no more vertices nor edges to be added to $\mathscr{H}_v$.\vspace{2mm}
	\end{enumerate}

	
	
	
	
	
	
\end{definition}

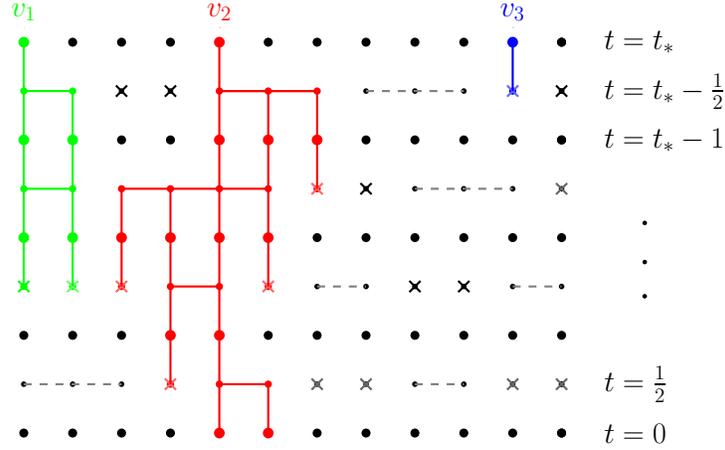
\begin{figure}
	\centering
	\begin{tikzpicture}[thick,scale=0.65, every node/.style={transform shape}]
	\foreach \x in {0,...,11}{
		\foreach \y in {0,...,3}{
			\filldraw[black] (\x,2*\y) circle (2pt);
			\filldraw[black] (\x,2*\y+1) circle (1pt);
		}
		\filldraw[black] (\x, 8) circle (2pt);
	}
	\foreach \x in {0,...,2}{
		
		\filldraw[green] (0,8-2*\x) circle (2.5pt);
		\filldraw[green] (0,7-2*\x) circle (1.5pt);
		
	}
	\filldraw[green] (1,6) circle (2.5pt);
	\filldraw[green] (1,4) circle (2.5pt);
	\filldraw[green] (1,7) circle (1.5pt);
	\filldraw[green] (1,5) circle (1.5pt);
	\filldraw[green] (1,3) circle (1.5pt);
	
	\filldraw[blue] (10,8) circle (2.5pt);
	\filldraw[black] (12.7,3.5) circle (0.7pt);
	\filldraw[black] (12.7,4.3) circle (0.7pt);
	\filldraw[black] (12.7,2.8) circle (0.7pt);
	\filldraw[black] (11,0) circle (2pt) node[black, anchor=west] {\LARGE{~~~~$t=0 $}};
	\filldraw[black] (11,1) circle (1pt) node[black, anchor=west] {\LARGE{~~~~$t=\frac{1}{2} $}};
	\filldraw[black] (11,6) circle (2pt) node[black, anchor=west] {\LARGE{~~~~$t=t_*-1 $}};
	\filldraw[black] (11,8) circle (2pt) node[black, anchor=west] {\LARGE{~~~~$t=t_*$}};
	\filldraw[blue] (10,7) circle (1.5pt);
	\filldraw[black] (11,7) circle (1pt) node[black, anchor=west] {\LARGE{~~~~$t=t_* - \frac{1}{2}$}};
	\foreach \x in {0,...,3}{
		\filldraw[red] (4,2*\x) circle (2.5pt);
		\filldraw[red] (4,2*\x+1) circle (1.5pt);
		
	}
	\filldraw[red] (4,8) circle (2.5pt);
	\filldraw[red] (5,7) circle (1.5pt);
	\filldraw[red] (6,7) circle (1.5pt);
	\filldraw[red] (6,6) circle (2.5pt);
	\filldraw[red] (6,5) circle (1.5pt);
	\filldraw[red] (5,5) circle (1.5pt);
	\filldraw[red] (5,4) circle (2.5pt);
	\filldraw[red] (5,3) circle (1.5pt);
	\filldraw[red] (3,3) circle (1.5pt);
	\filldraw[red] (3,2) circle (2.5pt);
	\filldraw[red] (3,4) circle (2.5pt);
	\filldraw[red] (3,1) circle (1.5pt);
	\filldraw[red] (3,5) circle (1.5pt);
	\filldraw[red] (5,0) circle (2.5pt);
	\filldraw[red] (5,1) circle (1.5pt);
	\filldraw[red] (2,3) circle (1.5pt);
	\filldraw[red] (2,4) circle (2.5pt);
	\filldraw[red] (2,5) circle (1.5pt);
	
	\draw (1,3) node[cross=3pt,green!60]{};
	\draw (3,1) node[cross=3pt,red!60]{};
	\draw (6,1) node[cross=3pt,black!60]{};
	\draw (7,1) node[cross=3pt,black!60]{};
	\draw (10,1) node[cross=3pt,black!60]{};
	\draw (11,5) node[cross=3pt,black!60]{};
	\draw (2,3) node[cross=3pt,red!60]{};
	\draw (5,3) node[cross=3pt,red!60]{};
	\draw (6,5) node[cross=3pt,red!60]{};
	\draw (0,3) node[cross=3pt,green]{};
	\draw (7,5) node[cross=3pt,black]{};
	\draw (8,3) node[cross=3pt,black]{};
	
	\draw (3,7) node[cross=3pt,black]{};
	\draw (2,7) node[cross=3pt,black]{};
	\draw (10,7) node[cross=3pt,blue!60]{};
	\draw (11,7) node[cross=3pt,black]{};
	\draw[thick, red] (5,3) -- (5,5);
	\draw[thick, blue] (10,7) -- (10,8);
	\draw[thick, red] (4,0) -- (4,8);
	\draw[thick, red] (6,5) -- (6,7);
	\draw[thick, red] (4,7) -- (6,7);
	\draw[thick, red] (4,7) -- (6,7);
	\draw[thick, red] (2,5) -- (5,5);
	\draw[thick, red] (2,5) -- (2,3);
	\draw[thick, red] (4,1) -- (5,1);
	\draw[thick, red] (5,0) -- (5,1);
	\draw[thick, red] (3,3) -- (4,3);
	
	\draw[dashed, gray] (7,7) -- (9,7);
	\draw[thick, green] (0,3) -- (0,8);
	\draw[thick, green] (1,3) -- (1,7);
	\draw[thick, green] (0,5) -- (1,5);
	\draw[thick, green] (0,7) -- (1,7);
	\draw[dashed, gray] (8,5) -- (10,5);
	
	\draw[dashed, gray] (0,1) -- (2,1);
	\draw[dashed, gray] (8,1) -- (9,1);
	\draw[dashed, gray] (6,3) -- (7,3);
	\draw[dashed, gray] (10,3) -- (11,3);
	\filldraw[green] (0,8.3) circle (0pt) node[green, anchor=south] {\LARGE{$v_1 $}};
	\filldraw[red] (4,8.3) circle (0pt) node[red, anchor=south] {\LARGE{$v_2 $}};
	\filldraw[blue] (10,8.3) circle (0pt) node[blue, anchor=south] {\LARGE{$v_3 $}};
	
	\draw[thick, red] (3,1) -- (3,5);
	\draw[thick, red] (5,5) -- (5,7);
	\filldraw[red] (5,6) circle (2.5pt);
	\draw (9,3) node[cross=3pt,black]{};
	\draw (11,1) node[cross=3pt,black!60]{};
	\end{tikzpicture}
	\caption{An example of the history diagram of the Swendsen-Wang dynamics on $\{1,2,\ldots,12\}$ until time $t_\star = 4.$ The figure illustrates the history diagrams of $v_1=1, \; v_2=5$ and $v_3=11$. The  crossed-out points at a half-integer time $t-\frac{1}{2}$ denote the oblivious vertices, and the dashed-lines describe the edges of $\bar{\omega}_{t-1}$ which are not included in the history. History diagrams of different vertices are distinguished by different colors.} \label{fig1}
\end{figure}


\quad We now introduce several notations on the history diagram as follows:

\begin{itemize}
	\item For any subset $A\subset V$, the history of $A$ is defined by $\mathscr{H}_A = \cup_{v \in A} \mathscr{H}_v $.
	
	\item For each $t\in \{0,\frac{1}{2},\,1,\ldots,t_\star - \frac{1}{2},\,t_\star \}$, we define $\mathscr{H}_v (t) = \mathscr{H}_v \cap (V \times \{t\})$. For convenience, we will regard $\mathscr{H}_v (t)$ as a subset of $V$, since they all have the same time element $t$.
	
	\item For a subset $A \subset V$ and time $t$, we define $\mathscr{H}_A (t) = \mathscr{H}_A \cap (V \times \{t\})$. As above, we will consider $\mathscr{H}_A (t)$ as a subset of $V$.
	
\end{itemize}

\quad Now we are ready to derive a new graph structure on $V$, which will lead us to the definition of the \textit{information percolation clusters}:

\begin{definition}\label{graphstruc}
	For any $u,v \in V$, we write $u \sim_i v$ if and only if $\mathscr{H}_u (s) \cap \mathscr{H}_v(s) \neq \emptyset $ for some half-integer $s \leq t_\star$.
\end{definition}
Note that we only check the intersections at half-integer times, since $\mathscr{H}_u (t) \cap \mathscr{H}_v(t) \neq \emptyset $ implies $\mathscr{H}_u (t+\frac{1}{2}) \cap \mathscr{H}_v(t+\frac{1}{2}) \neq \emptyset $ for an integer $t$.

\begin{definition}[Information percolation clusters]\label{ipcluster}
	Let $(V, \sim_i)$ be the graph with edges induced by the relation $\sim_i$. The connected components of this graph are called the \textbf{information percolation clusters}. Moreover, information percolation clusters are classified into three types. For each information percolation cluster $C \subset V$,
	
	\begin{enumerate}
		\item  $C$ is marked $\textsc{Red}$ if $\mathscr{H}_C(0) \neq \emptyset$.
		
		\item $C$ is marked $\textsc{Blue}$ if $\mathscr{H}_C (0) = \emptyset$ and $|C|=1$.
		
		\item $C$ is marked $\textsc{Green}$ if otherwise, i.e., $\mathscr{H}_C (0) = \emptyset$ and $|C| \geq 2$.
	\end{enumerate}
	
\end{definition}

\quad Let us introduce some notations for the information percolation clusters. For a given history diagram $\{ \mathscr{H}_v : v\in V \}$, let $\mathcal{C}_{\mathcal{R}}$ denote the collection of red clusters, and let $V_{\mathcal{R}}$ be the union of red clusters. We define $\mathcal{C}_{\mathcal{B}}$, $V_{\mathcal{B}}$, $\mathcal{C}_{\mathcal{G}}$ and $V_{\mathcal{G}}$ analogously for blue and green clusters, respectively. We also write $\mathscr{H}_{\mathcal{R}}$ for $\mathscr{H}_{V_{\mathcal{R}}}$ for convenience (and similarly for blue and green). 


\begin{remark}\label{bluermk}
	The only possible case of a vertex $v \in V$ being $\{v\} \in \mathcal{C}_{\mathcal{B}}$ is when it is an isolated vertex in the graph $(V, \bar{\omega}_{t_\star -1})$. If there was another vertex $u$ that belongs to the same connected component as $v$ in the graph $(V, \bar{\omega}_{t_\star -1})$, then one can observe that $v \sim_i u$, hence implying that the size of the information percolation cluster containing $v$ is at least $2$.
\end{remark}

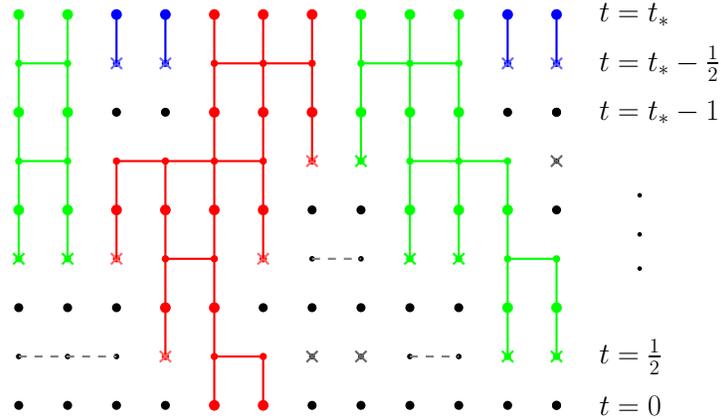
\begin{figure}
	\centering
	\begin{tikzpicture}[thick,scale=0.65, every node/.style={transform shape}]
	\foreach \x in {0,...,11}{
		\foreach \y in {0,...,3}{
			\filldraw[black] (\x,2*\y) circle (2pt);
			\filldraw[black] (\x,2*\y+1) circle (1pt);
		}
		\filldraw[black] (\x, 8) circle (2pt);
	}
	\foreach \x in {0,...,2}{
		
		\filldraw[green] (0,8-2*\x) circle (2.5pt);
		\filldraw[green] (0,7-2*\x) circle (1.5pt);
		\filldraw[green] (8,8-2*\x) circle (2.5pt);
		\filldraw[green] (8,7-2*\x) circle (1.5pt);
	}
	\filldraw[green] (1,8) circle (2.5pt);
	\filldraw[green] (1,7) circle (1.5pt);
	\filldraw[green] (1,5) circle (1.5pt);
	\filldraw[green] (7,8) circle (2.5pt);
	\filldraw[green] (7,7) circle (1.5pt);
	\filldraw[green] (7,6) circle (2.5pt);
	\filldraw[green] (7,5) circle (1.5pt);
	\filldraw[green] (9,8) circle (2.5pt);
	\filldraw[green] (1,4) circle (2.5pt);	
	\filldraw[green] (1,6) circle (2.5pt);	
	\filldraw[green] (9,7) circle (1.5pt);
	\filldraw[green] (1,3) circle (1.5pt);
	\filldraw[green] (9,5) circle (1.5pt);
	\filldraw[green] (10,3) circle (1.5pt);
	\filldraw[green] (10,4) circle (2.5pt);
	\filldraw[green] (10,5) circle (1.5pt);
	
	\filldraw[blue] (2,8) circle (2.5pt);
	\filldraw[blue] (3,8) circle (2.5pt);
	\filldraw[blue] (2,7) circle (1.5pt);
	\filldraw[blue] (3,7) circle (1.5pt);
	\filldraw[blue] (10,8) circle (2.5pt);
	\filldraw[black] (12.7,3.5) circle (0.7pt);
	\filldraw[black] (12.7,4.3) circle (0.7pt);
	\filldraw[black] (12.7,2.8) circle (0.7pt);
	\filldraw[black] (11,0) circle (2pt) node[black, anchor=west] {\LARGE{~~~~$t=0 $}};
	\filldraw[black] (11,1) circle (1pt) node[black, anchor=west] {\LARGE{~~~~$t=\frac{1}{2} $}};
	\filldraw[black] (11,6) circle (2pt) node[black, anchor=west] {\LARGE{~~~~$t=t_*-1 $}};
	\filldraw[blue] (11,8) circle (2.5pt) node[black, anchor=west] {\LARGE{~~~~$t=t_*$}};
	\filldraw[blue] (10,7) circle (1.5pt);
	\filldraw[blue] (11,7) circle (1.5pt) node[black, anchor=west] {\LARGE{~~~~$t=t_* - \frac{1}{2}$}};
	\foreach \x in {0,...,3}{
		\filldraw[red] (4,2*\x) circle (2.5pt);
		\filldraw[red] (4,2*\x+1) circle (1.5pt);
		
	}
	\filldraw[red] (4,8) circle (2.5pt);
	\filldraw[red] (5,8) circle (2.5pt);
	\filldraw[red] (6,8) circle (2.5pt);
	\filldraw[red] (5,7) circle (1.5pt);
	\filldraw[red] (6,7) circle (1.5pt);
	\filldraw[red] (6,6) circle (2.5pt);
	\filldraw[red] (6,5) circle (1.5pt);
	\filldraw[red] (5,5) circle (1.5pt);
	\filldraw[red] (5,4) circle (2.5pt);
	\filldraw[red] (5,3) circle (1.5pt);
	\filldraw[red] (3,3) circle (1.5pt);
	\filldraw[red] (3,1) circle (1.5pt);
	\filldraw[red] (3,5) circle (1.5pt);
	\filldraw[red] (5,0) circle (2.5pt);
	\filldraw[red] (3,2) circle (2.5pt);
	\filldraw[red] (3,4) circle (2.5pt);
	\filldraw[red] (5,1) circle (1.5pt);
	\filldraw[red] (2,3) circle (1.5pt);
	\filldraw[red] (2,4) circle (2.5pt);
	\filldraw[red] (2,5) circle (1.5pt);
	
	\draw (6,1) node[cross=3pt,black!60]{};
	\draw (7,1) node[cross=3pt,black!60]{};
	\draw (10,1) node[cross=3pt,green]{};
	\draw (11,5) node[cross=3pt,black!60]{};
	\draw (2,3) node[cross=3pt,red!60]{};
	\draw (5,3) node[cross=3pt,red!60]{};
	\draw (6,5) node[cross=3pt,red!60]{};
	\draw (0,3) node[cross=3pt,green]{};
	\draw (1,3) node[cross=3pt,green]{};
	\draw (7,5) node[cross=3pt,green]{};
	\draw (8,3) node[cross=3pt,green]{};
	
	\draw (3,7) node[cross=3pt,blue!60]{};
	\draw (2,7) node[cross=3pt,blue!60]{};
	\draw (10,7) node[cross=3pt,blue!60]{};
	\draw (11,7) node[cross=3pt,blue!60]{};
	\draw[thick, red] (5,3) -- (5,5);
	\draw[thick, blue] (2,8) -- (2,7);
	\draw[thick, blue] (3,8) -- (3,7);
	\draw[thick, blue] (10,7) -- (10,8);
	\draw[thick, blue] (11,7) -- (11,8);
	\draw[thick, red] (4,0) -- (4,8);
	\draw[thick, red] (5,7) -- (5,8);
	\draw[thick, red] (6,5) -- (6,8);
	\draw[thick, red] (4,7) -- (6,7);
	\draw[thick, red] (2,5) -- (5,5);
	\draw[thick, red] (2,5) -- (2,3);
	\draw[thick, red] (4,1) -- (5,1);
	\draw[thick, red] (5,0) -- (5,1);
	\draw[thick, red] (3,3) -- (4,3);
	
	\draw[thick, green] (10,1) -- (10,3);
	\draw[thick, green] (11,1) -- (11,3);
	\draw[thick, green] (7,7) -- (9,7);
	\draw[thick, green] (0,3) -- (0,8);
	\draw[thick, green] (0,5) -- (1,5);
	\draw[thick, green] (0,7) -- (1,7);
	\draw[thick, green] (7,8) -- (7,5);
	\draw[thick, green] (1,3) -- (1,7);
	\draw[thick, green] (8,8) -- (8,3);
	\draw[thick, green] (9,8) -- (9,7);
	\draw[thick, green] (8,5) -- (10,5);
	\draw[thick, green] (10,5) -- (10,3);
	\draw[thick, green] (1,8) -- (1,7);
	\draw[thick, green] (9,3) -- (9,7);
	
	\draw[dashed, gray] (0,1) -- (2,1);
	\draw[dashed, gray] (8,1) -- (9,1);
	\draw[dashed, gray] (6,3) -- (7,3);
	\draw[thick, green] (10,3) -- (11,3);
	
	\draw[thick, red] (3,1) -- (3,5);
	\draw[thick, red] (5,5) -- (5,7);
	\filldraw[red] (5,6) circle (2.5pt);
	\filldraw[green] (9,4) circle (2.5pt);
	\filldraw[green] (9,6) circle (2.5pt);	
	\filldraw[green] (9,3) circle (1.5pt);
	\filldraw[green] (10,2) circle (2.5pt);
	\filldraw[green] (11,2) circle (2.5pt);
	\filldraw[green] (10,1) circle (1.5pt);
	\filldraw[green] (11,1) circle (1.5pt);
	\filldraw[green] (11,3) circle (1.5pt);
	\draw (9,3) node[cross=3pt,green]{};
	\draw (11,1) node[cross=3pt,green]{};
	\draw (3,1) node[cross=3pt,red!60]{};
	
	\end{tikzpicture}
	\caption{An example of the history diagram of the Swendsen-Wang dynamics on $\{1,2,\ldots,12\}$ until time $t_* = 4.$ The crossed-out points at half-integer times denote the oblivious vertices. The dashed horizontal edges indicate the edges of $\bar{\omega}_t$ which are not included in the history. Colors are drawn with respect to the colors of the information percolation clusters that contain the vertices at the top.} \label{fig2}
\end{figure}

\quad Note that for a subset $A \subset V$, $\mathscr{H}_A (0) = \emptyset$ implies that the recursive procedure described in the Definition \ref{historydiag} terminates before reaching at time $0$. Therefore if an information percolation cluster $C$ satisfies $C\in \mathcal{C}_{\mathcal{B}} \cup \mathcal{C}_{\mathcal{G}}$ , then $X_{t_\star}(C)$ is independent of the starting configuration $X_0$. In particular, if $C\in \mathcal{C}_{\mathcal{B}}$, then the cluster itself is equal to a single vertex $v$ and $X_{t_\star}(v) $ is distributed according to Unif$\{1,\ldots,q\}$. 

\quad On the other hand, while $X_{t_\star}(C)$ is independent of $X_0$ when $C$ is a green cluster, configuration on $C$ can have a highly non-trivial distribution due to the dependencies between the intersecting update histories. It is these green clusters that contain the complicated structure of the Potts measure. 
In order to avoid this complication, we adopt the following strategy:
\vspace{2mm}

\noindent \textbf{Strategy 1}. Condition on the histories of the green clusters and study the remaining red and blue clusters. \vspace{1mm}

\quad By conditioning on the histories of the green clusters, each remaining vertex is either a blue singleton or a member of a red cluster. Since the law of the blue singletons are i.i.d. uniform distributions on $\{1,\ldots,q\}$, this approach turns our focus solely to the red clusters.

\quad Red clusters encode the information of possible dependency of $X_{t_\star}$ on the starting configuration. For instance, if $ \mathcal{C}_{\mathcal{R}} = \emptyset$, then $X_{t_\star} \sim \pi$; since $X_{t_\star}$ is independent of $X_0$, we would have the same configuration at time $t_\star$ even if we start with $X_0 \sim \pi$.

\quad Therefore at the point when $\mathbb{P}_{x_0} ( X_{t_\star} \in \cdot) $ is close to $\pi$, it is fair to expect the size of $V_{\mathcal{R}}$ to be small. In fact, it turns out that conditioned on $\mathscr{H}_{\mathcal{G}}$, the $L^2$-distance from $X_{t_\star}$ to stationarity can be controlled by the size of $V_{\mathcal{R}}$. \vspace{2mm}

\noindent \textbf{Strategy 2}. Estimate the size of the red clusters.\vspace{1mm}

\quad The key step in the analysis of the red clusters is controlling a conditional probability that $A \in \mathcal{C}_{\mathcal{R}}$, in which we condition not only on $\mathscr{H}_{\mathcal{G}}$, but on the entire histories outside of $A$, and that $A$ itself is either a full red cluster or a union of blue singletons. To write it formally, for any subset $A\subset V$ let $\mathscr{H}_A^- := \bigcup \{\mathscr{H}_v : v\in V \setminus A \}~ = \mathscr{H}_{V\setminus A}  $, and define
\begin{equation}\label{psidef}
\Psi_A := \sup_{\mathscr{H}_A^-: \: \mathscr{H}_A^- \in \mathscr{H}_{com}(A) } \mathbb{P} \left( A \in \mathcal{C}_{\mathcal{R}} \mid \mathscr{H}_A^-, \mspace{5mu} \{A \in \mathcal{C}_{\mathcal{R}  }\} \cup \{A \subset V_{\mathcal{B}} \} \right),
\end{equation}
where $\mathscr{H}_A^- \in \mathscr{H}_{com}(A)$ is the shorthand notation for $\mathscr{H}_A^- \cap \,(A \times \{t_\star-\frac{1}{2} \}) =\emptyset  $, which imposes a compatibility condition on $\mathscr{H}_A^-$. This is introduced to prevent $\{A \in \mathcal{C}_{\mathcal{R}}\} \cup \{A \subset V_{\mathcal{B}} \}  $ from being an empty event. Note that aiming to estimating the probability of $\{A \in \mathcal{C}_{\mathcal{R}}\} $, we may require that $\mathscr{H}_A$ must not intersect $\mathscr{H}_A^-$, since otherwise $\{A \in \mathcal{C}_{\mathcal{R}} \}$ is an empty event.

\begin{remark}\label{compatrmk}
	\textnormal{It is worth noting that when conditioning on the collective history of green clusters,  $\mathscr{H}_{\mathcal{G}}$ should always satisfy the condition $\mathscr{H}_{\mathcal{G}} \in \mathscr{H}_{com} (V \setminus V_{\mathcal{G}}) $. Indeed, the green histories must be disjoint from the histories of blue and red clusters. In what follows, therefore, we automatically impose $\mathscr{H}_{\mathcal{G}} \in \mathscr{H}_{com} (V \setminus V_{\mathcal{G}}) $ to hold whenever we condition on $\mathscr{H}_{\mathcal{G}}$. However, we refuse to write this notation explicitly for the sake of the simplicity of notation.}
\end{remark}
An important bound on $\Psi_A$ is described by the following lemma which will be proven in \S\ref{secinfoperc3}.

\begin{lemma}\label{psilem}
	Let $A\neq \emptyset$ be an arbitrary subset of $V$. For any $\theta >0$, there exist constants $M=M(\theta)$ and $p_0 = p_0 (\theta, d)$ such that for any $p < p_0$,
	\begin{equation*}
	\Psi_A \leq M (3edp)^{t_\star -\frac{1}{2}} e^{-\theta \mspace{3mu}\mathfrak{M}(A)},
	\end{equation*}
	where $\mathfrak{M}(A)$ is the size of the smallest connected subgraph containing $A$.
\end{lemma}

\quad For intuition, recall that if $A$ is a red cluster, then the histories $\{ \mathscr{H}_v : v\in A \}$ are all connected and at least one of them survives until time $0$. It turns out that the term $(3edp)^{t_\star -\frac{1}{2}} $ bounds the probability of $\mathscr{H}_v (0) \neq \emptyset$. The last term $e^{-\theta \mspace{4mu} \mathfrak{M}(A)}$ comes from the observation that the histories $\{ \mathscr{H}_v : v\in A \}$ must be spatially connected, thus the projection of the history diagram on $V$ is a connected subgraph containing $A$ (whose size is at least $\mathfrak{M}(A)$).


\subsection{Estimating the $L^2$ distance}\label{secinfoperc2} The goal of this section is to prove the following theorem that estimates the $L^2$ distance from stationarity of the Swendsen-Wang dynamics.

\begin{theorem}\label{l2thm}
	Let $d\geq 2$ and $q \geq2$ be fixed integers, $G = (V,E)$ be a degree-$d$ transitive graph on $n$ vertices, and let $X_t$ be the Swendsen-Wang dynamics for the $q$-state Potts model on $G$.  Then there exists a positive constant $p_0 = p_0 (d)$ and $C=C(d,p)$ such that for any $p <p_0$ the following inequality holds true for any large enough $n$ and $ t \geq C \log n$ :
	\begin{equation}\label{l2ineq}
	\max_{x_0} \| \mathbb{P}_{x_0} ( X_t \in \cdot) - \pi \|_{L^2(\pi)} \leq 2 \exp \left( - \log \left( \frac{1}{1-\gamma} \right) \left( t- C \log n \right)  \right).
	\end{equation}
\end{theorem}

\begin{remark}\label{rmklogsob}
	\textnormal{Diaconis and Saloff-Coste proves an inequality that is essentially the same as (\ref{l2ineq}) when a Markov chain has log-Sobolev constant that is bounded uniformly in the size of the system (\cite{dsc96}, Theorem 3.7). Therefore in \cite{ls13},  Theorem \ref{l2thm} for the single-site Glauber dynamics comes for free as the ``strong spatial mixing'' implies the log-Sobolev constant being bounded uniformly.  In our case, we prove Theorem \ref{l2thm} based on the information percolation framework rather than bounding the log-Sobolev constant of the Swendsen-Wang dynamics.}
\end{remark}

\begin{remark}\label{windowrep}
	\textnormal{It turns out in Proposition \ref{l2prop} that the constant $C=C(d,p)$ given by $C= 2({\log (\frac{1}{3edp})})^{-1}$ satisfies (\ref{l2ineq}). This fact will be used in sections \ref{secl1l2} and \ref{secfinal} when determining the cutoff window.}
\end{remark}

\begin{proof}
	Let $P_t: L^2(\pi) \rightarrow L^2(\pi)$ be the semigroup operator defined by $P_t f(x) = \mathbb{E}_x [f(X_t)]$, and let $p_{t,x} (y) = \frac{\mathbb{P}_x(X_t = y)}{\pi(y)}$  .  Since the Swendsen-Wang dynamics is reversible, a simple calculation yields that
	\begin{equation*}
	\begin{split}
	P_t (p_{s,x})(y) =  \sum_z \mathbb{P}_y(X_t = z) \frac{\mathbb{P}_x (X_s =z)}{\pi(z)} = \sum_z \frac{\mathbb{P}_z(X_t = y)}{\pi(y)} \mathbb{P}_x (X_s =z) = p_{t+s,x}(y).
	\end{split}
	\end{equation*}
	
	\quad Therefore, the $L^2$ distance from stationarity at time $t+s$ becomes
	\begin{equation}\label{scl2ineq}
	\begin{split}
	\| \mathbb{P}_{x_0} ( X_{t+s} \in \cdot) - \pi \|_{L^2(\pi)} &= 	\| P_t(p_{s,x_0}) - \pi(p_{s,x_0}) \|_{L^2(\pi)} \\&\leq \| P_t - \pi \|_{2 \rightarrow 2} \| p_{s,x_0} - \mathbbm{1} \|_{L^2(\pi)},
	\end{split}
	\end{equation}
	where $\pi(f) := \mathbb{E}_{\pi}f$ and $\mathbbm{1} $ denotes the identity map. Since the operator norm of $P_t - \pi$ satisfies $ \| P_t - \pi \|_{2 \rightarrow 2} \leq (1-\gamma)^t$ (see e.g., Chapter 2 of \cite{sc}), we readily obtain the conclusion from the following proposition.
\end{proof}

\begin{proposition}\label{l2prop}
	Under the same assumption as Theorem \ref{l2thm}, let $t_0 = \frac{ \log n}{\log (1/3edp)}$. Then there exists a positive constant $p_0 = p_0(d)$  such that for any $p<p_0$ the following inequality holds true for any large enough $n$:
	\begin{equation}\label{mainpropineq}
	\max_{x_0} \| \mathbb{P}_{x_0} ( X_{t_0} \in \cdot) - \pi \|_{L^2(\pi)} \leq 2.
	\end{equation}
\end{proposition}

\quad In the remaining section, we discuss the proof of Proposition \ref{l2prop}.

\quad Let $(X_t)$ and $(Y_t)$ be two copies Swendsen-Wang dynamics started from the initial configurations $X_0 = x_0$ and $Y_0 \sim \pi$, respectively. We couple their update sequence via thel coupling in Definition \ref{globalcoupling}, and hence they share the same history diagram. The following lemma is a simple variant of Jensen's inequality which enables us to work with $X_t$ and $Y_t$ conditioned on $\mathscr{H}_{\mathcal{G}}$.

\begin{lemma}\label{jensen}
	Let $\mu, \nu$ be two probability measures on the same finite sample space and $H$ be a random variable. Then the following inequality holds true:
	\begin{equation*}
	\| \mu - \nu \|_{L^2(\nu)}^2 \leq \int \| \mu(~\cdot~ \vert H) - \nu(~\cdot ~ \vert H) \|_{L^2(\nu(\cdot \vert H))}^2 d \mathbb{P}(H).
	\end{equation*}
\end{lemma}

\begin{proof}
	Proving this statement is a simple consequence of Jensen's inequality. Calculation yields that
	\begin{equation*}
	\begin{split}
	\| \mu - \nu \|_{L^2(\nu)}^2 + 1 &= \sum_x \left( \frac{\mu(x)}{\nu(x)} \right)^2 \nu(x)\\
	&= \sum_x \left[ \frac{\int \mu(x \vert H) d\mathbb{P}(H)}{\int \nu(x \vert H) d\mathbb{P}(H)} \right]^2 \int \nu(x \vert H) d\mathbb{P}(H) \\
	&= \sum_x \left[ \int \frac{\mu(x \vert H)}{\nu(x \vert H)} \frac{\nu(x \vert H)d\mathbb{P}(H) }{\int \nu(x \vert H') d\mathbb{P}(H')}  \right]^2 \int \nu(x \vert H) d\mathbb{P}(H)\\
	&\leq \sum_x \int{ \left( \frac{\mu(x \vert H)}{\nu(x \vert H)} \right)^2 \nu(x \vert H) d\mathbb{P}(H)} \\
	&= \int \sum_x \left( \frac{\mu(x \vert H)}{\nu(x \vert H)} \right)^2 \nu(x \vert H) d\mathbb{P}(H) \\
	&= \int \| \mu(~\cdot~ \vert H) - \nu(~\cdot ~ \vert H) \|_{L^2(\nu(\cdot \vert H))}^2 d\mathbb{P}(H) +1.
	\end{split}
	\end{equation*}
	
	\quad The inequality in the third line follows from Jensen's inequality, which is applied with respect to the probability measure $\frac{\nu(x \vert H) dH}{\int \nu(x \vert H') dH'} $.
\end{proof}

\quad Applying Lemma \ref{jensen} to our case, we get
\begin{equation*}
\begin{split}
\| \mathbb{P}_{x_0}( X_t \in \cdot) - \mathbb{P}(Y_t \in \cdot) \|_{L^2(\pi)}^2
&\leq \int	\| \mathbb{P}_{x_0}( X_t \in \cdot \mspace{3mu} \vert \mathscr{H}_{\mathcal{G}}) - \mathbb{P}(Y_t \in \cdot \mspace{3mu} \vert \mathscr{H}_{\mathcal{G}}) \|_{L^2(\pi(\: \cdot \: \vert \mathscr{H}_{\mathcal{G}}))}^2  d\mathscr{H}_{\mathcal{G}} \\
&\leq \sup_{\mathscr{H}_{\mathcal{G}}} 	\| \mathbb{P}_{x_0}( X_t \in \cdot \mspace{3mu} \vert \mathscr{H}_{\mathcal{G}}) - \mathbb{P}(Y_t \in \cdot \mspace{3mu} \vert \mathscr{H}_{\mathcal{G}}) \|_{L^2(\pi(\: \cdot \: \vert \mathscr{H}_{\mathcal{G}}))}^2 ,
\end{split}
\end{equation*}
where we take the supremum over all $\mathscr{H}_{\mathcal{G}} \in \mathscr{H}_{com} (V \setminus V_{\mathcal{G}})$, as discussed in Remark \ref{compatrmk}.

\quad Since $X_t (V_{\mathcal{G}})$ is independent of $X_0$, the coupling between $X_t$ and $Y_t$ implies that they are identical on $V_{\mathcal{G}}$. Moreover, conditioned on $\mathscr{H}_{\mathcal{G}}$, $X_t ( V \setminus V_{\mathcal{G}})$ and $X_t (V_{\mathcal{G}})$ are independent from each other, since the two histories $\mathscr{H}_{V_{\mathcal{G}}}$ and $\mathscr{H}_{V\setminus V_{\mathcal{G}}} $ are disjoint and $X_t (V_{\mathcal{G}})$ is independent of $X_0$. Therefore, the projection onto $V \setminus V_{\mathcal{G}}$ does not decrease the $L^2$ distance, and hence
\begin{equation}\label{l2distwithoutunif}
\begin{split}
\max_{x_0}\| \mathbb{P}_{x_0}( X_t \in \cdot) -\pi \|_{L^2(\pi)}&= \max_{x_0}\| \mathbb{P}_{x_0}( X_t \in \cdot) - \mathbb{P}(Y_t \in \cdot) \|_{L^2(\pi)} \\
&\leq \max_{x_0} \mspace{6mu} \sup_{\mathscr{H}_{\mathcal{G}}} 	\| \widetilde{\mu}_t - \widetilde{\pi} \|_{L^2(\widetilde{\pi}) },
\end{split}
\end{equation}
where we define $\widetilde{\mu}_t$ and $\widetilde{\pi}$ by
\begin{equation*}
\begin{split}
\widetilde{\mu}_t &:=  \mathbb{P}_{x_0}( X_t(V\setminus V_{\mathcal{G}}) \in \cdot \mspace{3mu} \vert \mathscr{H}_{{\mathcal{G}}});\\ \widetilde{\pi}&:=\mathbb{P}(Y_t (V \setminus V_{\mathcal{G}})\in \cdot \mspace{3mu} \vert \mathscr{H}_{{\mathcal{G}}}). 
\end{split}
\end{equation*}

\quad Assume that we have $V_{\mathcal{R}} = \emptyset$. Then $V \setminus V_{\mathcal{G}}$ consists only of blue vertices, and hence the law of $X_t(V \setminus V_{\mathcal{G}}) $ conditioned on $\mathscr{H}_{\mathcal{G}}$ is identical to $\nu_{V\setminus V_{\mathcal{G}}}$, where $\nu_{U}$ denotes the uniform distribution on $\{1,\ldots,q\}^U$. Therefore, if the effect of the red clusters is negligible, it is reasonable to predict that the law of $X_t(V \setminus V_{\mathcal{G}}) $ conditioned on $\mathscr{H}_G$ is close to $\nu_{V\setminus V_{\mathcal{G}}}$. Based on this intuition, we will modify (\ref{l2distwithoutunif}) using the following lemma, switching our attention to the distance between $X_t(V \setminus V_{\mathcal{G}}) $ and $\nu_{V\setminus V_{\mathcal{G}}}$.

\begin{lemma}\label{uniflem}
	Consider the Swendsen-Wang dynamics $X_t$ on $G$ starting from the initial configuration $x_0$. Let $Y_t$ be the coupled chain with initial distribution $Y_0 \sim \pi$. Then there exists $p_0 = p_0(d)$ such that the following holds true: for all $p\in(0,p_0)$, there exists a constant $C = C(d,p)$ such that conditioned on the history of green clusters, we have
	\begin{equation}\label{l2distwithunif}
	\begin{split}
	\max_{x_0} \mspace{6mu} \sup_{\mathscr{H}_{\mathcal{G}}} \| \widetilde{\mu}_t - \widetilde{\pi} \|^2_{L^2(\widetilde{\pi}) }
	\leq 2 \max_{x_0} \mspace{6mu} \sup_{\mathscr{H}_{\mathcal{G}}} 	\| \widetilde{\mu}_t -\nu_{V \setminus V_{\mathcal{G}}} \|^2_{L^2(\nu_{V \setminus V_{\mathcal{G}}})} ~+1,
	\end{split}
	\end{equation}
	for all $ t>C \log n$ and any large enough $n$.
\end{lemma}

\begin{proof}
	By expanding the $L^2$-distance, we have
	\begin{equation}\label{tildemu}
	\| \widetilde{\mu} - \widetilde{\pi} \|^2_{L^2 ( \widetilde{\pi})}=	\sum_{x \in \Omega_{V\setminus V_{\mathcal{G}}}} \frac{\widetilde{\mu}(x)^2}{\widetilde{\pi}(x)} -1 .
	\end{equation}
	
	\quad Observe that for any $x \in \Omega_{V\setminus V_{\mathcal{G}}}$, $\widetilde{\pi}(x)$ can be sampled as follows: Sample $V_{\mathcal{R}} \subset V\setminus V_{\mathcal{G}}$ via the law $\eta $ that generates the red clusters,  then sample the configuration $x_{V_{\mathcal{R}}} $ on the red clusters. Finally, generate $x_{V_{\mathcal{B}}}$ according to $\nu_{V_{\mathcal{B}}}$, where $V_{\mathcal{B}} = V \setminus (V_{\mathcal{G}} \cup V_{\mathcal{R}})$. Therefore, $\widetilde{\pi}(x)$ can be written as
	\begin{equation*}
	\widetilde{\pi}(x) = \sum_{R \subset V\setminus V_{\mathcal{G}}} \eta(R) \varphi_R (x_{R}) \nu_B(x_B),
	\end{equation*}
	where $\varphi_R$ represents the law of $x_R$ given that $R$ is the union of red clusters, and $B$ being the shorthand for $V\setminus (V_{\mathcal{G}} \cup R) $. Hence, $\widetilde{\pi}(x)$ satisfies the following trivial inequality when we just take $R= \emptyset$:
	\begin{equation}\label{boundtildepi}
	\widetilde{\pi}(x) \geq \eta(\emptyset) ~\nu_{V \setminus V_{\mathcal{G}}} (x).
	\end{equation}
	
	\begin{claim}\label{boundeta}
		There exists a constant $C=C(d,p)$ such that $\eta(\emptyset) \geq \frac{1}{2}$ for all $t > C \log n$.
	\end{claim}
	
	\quad Assume Claim \ref{boundeta} for the moment. Then (\ref{tildemu}) can be rewritten as
	\begin{equation*}
	\| \widetilde{\mu} - \widetilde{\pi} \|^2_{L^2 ( \widetilde{\pi})} \leq 2 \sum_{x \in \Omega_{V\setminus V_{\mathcal{G}}}} \frac{\widetilde{\mu}(x)^2}{\nu_{V\setminus V_{\mathcal{G}}} (x)} -1 = 2 \| \widetilde{\mu} - \nu_{V\setminus V_{\mathcal{G}}} \|^2_{L^2 (\nu_{V\setminus V_{\mathcal{G}}})} +1,
	\end{equation*}
	when $t > C \log n$. By taking supremum over $\mathscr{H}_{\mathcal{G}}$ and $x_0$, we deduce (\ref{l2distwithunif}).
\end{proof}
\vspace{2mm}
\textit{Proof of Claim \ref{boundeta}}. ~We start by observing that
\begin{equation}\label{etaeq1}
\eta(\{\emptyset\}^c) 
\leq
\sum_{A \neq \emptyset} \mathbb{P}(A =  V_\mathcal{R} \,\vert\, \mathscr{H}_\mathcal{G} )
\leq
\sum_{A \neq \emptyset} \mathbb{P}(A\in \mathcal{C}_\mathcal{R} \,\vert\, \mathscr{H}_\mathcal{G} )
\leq
\sum_{A \neq \emptyset} \Psi_A,
\end{equation}
by a union bound and the definition of $\eta$. The last inequality is clear by the definition of $\Psi_A$ in (\ref{psidef}). Using Lemma \ref{psilem} with the choice of $\theta = \log (4ed)$, we obtain that
\begin{equation}\label{sumpsi}
\begin{split}
\sum_{A\neq \emptyset} \Psi_A 
&\leq
\sum_{v\in V} \sum_{A \ni v} \Psi_A
\leq
n \sum_{k\geq 1} \sum_{\substack{A\ni v \\ \mathfrak{M}(A)=k }} M(3edp)^{t-\frac{1}{2}} e^{-\theta k}\\
&\leq
nM(3edp)^{t-\frac{1}{2}} \sum_{k\geq 1} (2e^{-\theta +1}d)^k 
\leq 
nM(3edp)^{t-\frac{1}{2}},
\end{split}
\end{equation}
where we used the fact that $|\{A: A\ni v, \,\mathfrak{M}(A)=k \}| \leq (2ed)^k $;  the number of connected subsets of size $k$ containing a given vertex $v$ is at most ${dk \choose k} \leq (ed)^k$ and each such subset includes at most $2^k$ subsets satisfying $\mathfrak{M}(A)=k$. 

\quad Thus, we can choose a positive constant $C=C(d,p)$ such that the r.h.s. of (\ref{sumpsi}) is smaller than $\frac{1}{2}$ for all $t > C\log n$. Together with (\ref{etaeq1}), we readily deduce the conclusion. \qed \vspace{2mm}

\quad There is a simple but beautiful lemma due to Miller and Peres \cite{mp} to  bound the $L^2$-distance of a measure from the uniform measure. Here, we use the version in \cite{ls15, ls16} of this lemma. Although  they deal with the case of $q=2$, generalizing it to arbitrary $q$ is straightforward, and hence we omit its proof. We use this lemma as a key ingredient in proving Proposition \ref{l2prop}.

\begin{lemma}[\cite{mp, ls15, ls16}]\label{mplem}
	Let $\Omega = \{1,\ldots,q\}^V$ for a finite set $V$. For each $R \subset V$, let $\varphi_R$ be a measure on $\{1,\ldots,q\}^R$. Let $\nu$ be the uniform measure on $\Omega$, and let $\mu $ be the measure on $\Omega$ obtained by sampling a subset $R \subset V$ via some measure $\eta$, generating the colors on $R$ via $\varphi_R$, and finally sampling the colors on $V \subset R$ uniformly. Then
	\begin{equation}\label{mpineq}
	\| \mu - \nu \|_{L^2(\nu)}^2 \leq \mathbb{E} \left[q^{| R \cap R' |} \right] -1 ,
	\end{equation}
	where $R$ and $ R'$ are i.i.d. with law $\widetilde{\mu}$.
\end{lemma}

\quad By applying Lemma \ref{mplem} to (\ref{l2distwithunif}) and combining with (\ref{l2distwithoutunif}), we obtain
\begin{equation}\label{l2distwithmp}
\max_{x_0}\| \mathbb{P}_{x_0}( X_t \in \cdot) -\pi \|^2_{L^2(\pi)} \leq 2 ~\sup_{\mathscr{H}_{\mathcal{G}}} \mathbb{E} \left[ q^{|V_R \cap V_{R'}|} ~\vert~ \mathscr{H}_{\mathcal{G}} \right] - 1,
\end{equation}
where $V_R$, $V_{R'}$ are i.i.d. copies of the variable $\cup_{A \in \mathcal{C}_{\textsc{Red} }} A $ conditioned on $\mathscr{H}_{\mathcal{G}}$.

\quad We will reduce the quantity $|V_{R}  \cap V_{R'} | $ to one that involves the $\Psi_A$ variables defined in Definition \ref{psidef}. This is done by the following lemma and its corollary due to  Lemma 2.3 and Corollary 2.4 of \cite{ls15}. 
Even though our definitions of update histories, history diagram and information percolation clusters differ in details from  \cite{ls15, ls16},   the two lemmas below can be proven exactly in the same way. 

\begin{lemma}[\cite{ls15}, Lemma 2.3]\label{couplinglem}
	Let $\{J_A : A\in V \} $ be a family of independent indicators satisfying $\mathbb{P}(J_A = 1) = \Psi_A$. The conditional distribution of red clusters given $\mathscr{H}_{{\mathcal{G}}}$ can be coupled such that
	\begin{equation*}
	\{ A: A \in \mathcal{C}_{{\mathcal{R}}} \} \subset \{A : J_A =1  \}.
	\end{equation*}
\end{lemma}


\begin{corollary}[\cite{ls15}, Corollary 2.4]\label{couplingcor}
	Let $\{J_{A, A'} : A, A' \subset V \}$ be a family of independent indicators satisfying
	\begin{equation}\label{jdef}
	\mathbb{P}(J_{A,A'}=1) = \Psi_{A} \Psi_{A'} ~~~ for ~any~ A, A' \subset V.
	\end{equation}
	The conditional distribution of $(V_{\mathcal{R}}, V_{{\mathcal{R}}'} )  $ given $\mathscr{H}_{{\mathcal{G}}}$ can be  coupled to $J_{A,A'} $'s such that
	\begin{equation}\label{stochdom}
	| V_{{\mathcal{R}}} \cap V_{{\mathcal{R}}'} | ~\preceq ~
	\sum_{A \cap A' \neq \emptyset} |A \cup A' | J_{A, A'}.
	\end{equation}
\end{corollary}

\quad We continue analysing (\ref{l2distwithmp}). Using $|A \cup A'|$  into $|A| + |A'|$, we get
\begin{equation}\label{maineq2}
\begin{split}
\sup_{\mathscr{H}_{{\mathcal{G}}}} \mathbb{E} \left. \left[q^{| V_{R} \cap V_{R'} | }  ~ \right\vert \mathscr{H}_{{\mathcal{G}}} \right]
&\leq \mathbb{E} \left[ \exp \left( 	\log q  \mspace{4mu} \sum_{A \cap A' \neq \emptyset}( |A|+ |A'|) \: J_{A,A'} \right)  \right] \\
&= \prod_{A \cap A' \neq \emptyset}  \mathbb{E} \left[ \exp \left( 	\log q  \mspace{4mu} (|A|+| A' |) \: J_{A,A'}  \right)  \right] ,
\end{split}
\end{equation}
with the equality due to the independence of $J_{A,A'}$'s. Note that we have eliminated the conditioning on $\mathscr{H}_{\mathcal{G}}$. By the definition of $J_{A,A'}$'s in (\ref{jdef}),
\begin{equation}\label{maineq3}
\begin{split}
\prod_{A \cap A' \neq \emptyset}  \mathbb{E} \left[ \exp \left( 	\log q  \mspace{4mu} (|A|+| A' |)  J_{A,A'}  \right)  \right]&\leq \prod_{v}\prod_{A \cap A' \neq \emptyset} \left[ \left( q^{|A|+|A'|} -1 \right) \Psi_A \Psi_{A'} +1 \right]\\
&\leq \exp \left(n \left(\sum_{A \ni v} q^{|A|} \Psi_A \right)^2 \right) ,
\end{split}
\end{equation}
where the last inequality follows from the symmetry of $G$.
Thus, by referring to (\ref{l2distwithmp}), we conclude that
\begin{equation}\label{maineq4}
\max_{x_0} \| \mathbb{P}_{x_0} ( X_{t} \in \cdot ) - \pi \|^2_{L^2(\pi)} \leq 2 \mspace{4mu} \exp \left(n \left(\sum_{A \ni v} q^{|A|} \Psi_A \right)^2 \right) -1 .
\end{equation}

\quad Implementing Lemma \ref{psilem} with $\theta = \log (4q ed) $ we obtain that
\begin{equation*}
\begin{split}
\sum_{A\ni v} q^{|A|} \Psi_A &\leq M (3edp)^{t -\frac{1}{2}} \sum_{k \geq 1} \sum_{\substack{A\ni v \\ \mathfrak{M}(A)=k }  } q^{k} e^{-\theta k} \\
&\leq  M (3edp)^{t -\frac{1}{2}} \sum_{k \geq 1} (2q d e^{-\theta+1})^k \leq M (3edp)^{t -\frac{1}{2}} ,
\end{split}
\end{equation*}
where we bounded the number $|\{A: A\ni v, \, \mathfrak{M}(A)=k \}| $ in the same way as (\ref{sumpsi}). Thus, plugging in $t = t_0$ which is defined in the statement of Proposition \ref{l2prop} gives us that
\begin{equation*}
2 \mspace{3mu} \exp \left(n \left(\sum_{A \ni v} q^{|A|} \Psi_A \right)^2 \right) -1 \leq 2 \mspace{3mu} \exp \left(\frac{M^2}{n (3edp)} \right) -1 \leq 2 ,
\end{equation*}
for all large enough $n$. Together with (\ref{maineq4}), we deduce (\ref{mainpropineq}). \qed

\subsection{Proof of Lemma \ref{psilem}}\label{secinfoperc3} Fix an arbitrary subset $A$ of $V$ and recall the definition of $\Psi_A$ in (\ref{psidef}). To estimate $\Psi_A$, we will first fix $\mathscr{H}_A^- $ to be equal to some history $\mathcal{X}$ and later take supremum over $\mathcal{X}$. For the validity of the definition of $\Psi_A$, pick any history diagram $\mathcal{X}$ that is disjoint with $ A \times \{t -\frac{1}{2} \}  $.

\quad For a given subset $S \subset V$, let $\mathcal{C}_{{\mathcal{R}}(S)}^* $ denote the collection of red clusters that arise when exposing the joint histories of $S $ (for instance, $\mathcal{C}_{{\mathcal{R}}} = \mathcal{C}_{{\mathcal{R}}(V)}^*  $). Let $V_{{\mathcal{R}}(S)}^*$ be the union of the members of $\mathcal{C}_{{\mathcal{R}}(S)}^* $ and define $\mathcal{C}_{{\mathcal{B}}(S)}^*, ~V_{{\mathcal{B}}(S)}^*$ analogously. Observe the following two identities given that $\mathscr{H}_A^- = \mathcal{X} $:
\begin{equation*}
\begin{split}
\{ A \in \mathcal{C}_{\mathcal{R}}  \} ~&=~ \{ A\in \mathcal{C}_{\mathcal{R}(A)}^*  \} \cap \{\mathscr{H}_A \cap \mathcal{X} = \emptyset  \};\\
\{ A \subset \mathcal{C}_{\mathcal{B}}  \} ~&=~ \{ A \subset \mathcal{C}_{{\mathcal{B}}(A)}^*  \} \cap \{\mathscr{H}_A \cap \mathcal{X} = \emptyset  \}
=~ \{ A \subset \mathcal{C}_{{\mathcal{B}}(A)}^*  \}.
\end{split}
\end{equation*}

\noindent Indeed, $A \in \mathcal{C}_{{\mathcal{R}}}$ clearly implies both $A\in \mathcal{C}_{{\mathcal{R}}(A)}^* $ and $\mathscr{H}_A \cap \mathcal{X} = \emptyset  $. On the other hand, if $A\in \mathcal{C}_{{\mathcal{R}}(A)}^* $, then $A$ should be a subset of a red cluster in $\mathcal{C}_{{\mathcal{R}}}$. Then $\mathscr{H}_A \cap \mathcal{X} = \emptyset$ imposes the history of $A$ to be disjoint with that of $V\setminus A $, and hence $A$ itself is a red cluster in $ \mathcal{C}_{{\mathcal{R}}}$.
Thus, we find out that  $\mathbb{P} ( A \in \mathcal{C}_{{\mathcal{R}}} \mid \mathscr{H}_A^- = \mathcal{X} , ~ \{A \in \mathcal{C}_{{\mathcal{R}}} \} \cup \{A \subset V_{{\mathcal{B}}} \} )$ is equal to
\begin{equation}\label{quotientrep}
\frac{\mathbb{P} ( A \in \mathcal{C}_{{\mathcal{R}}(A)}^* , ~\mathscr{H}_A \cap \mathcal{X} = \emptyset  \mid \mathscr{H}_A^- = \mathcal{X} )}{\mathbb{P}( \{A \in \mathcal{C}_{{\mathcal{R}}(A)}^* \} \cup \{A \subset V_{{\mathcal{B}}(A)}^* \} , ~\mathscr{H}_A \cap \mathcal{X} = \emptyset  \mid \mathscr{H}_A^- = \mathcal{X})}.
\end{equation}

\quad We start by estimating $\mathbb{P} ( A \subset V_{{\mathcal{B}}(A)}^*, ~ \mathscr{H}_A \cap \mathcal{X} = \emptyset \mid \mathscr{H}_A^- = \mathcal{X}) $, which is at most the denominator of (\ref{quotientrep}).
As discussed in Remark \ref{bluermk}, the only possible way for $\{A \subset V_{{\mathcal{B}}(A)}^*\}$ to happen is when every $v \in A$ is isolated in the graph $(V,\bar{\omega}_{t_\star -1})$, which is to close all the edges adjacent to $v$ in $\bar{\omega}_{t_\star -1} $. Since the underlying graph is transitive with degree $d$, we get
\begin{equation}\label{lowerbddenom}
\mathbb{P} ( A \subset V_{{\mathcal{B}}} \mid \mathscr{H}_A^- = \mathcal{X}) \geq (1-p)^{d |A|},
\end{equation}
where the bound being is independent of $\mathcal{X}$.

\quad Let us turn our attention to bounding $\mathbb{P} ( A \in \mathcal{C}_{{\mathcal{R}}(A)}^*, ~\mathscr{H}_A \cap \mathcal{X} = \emptyset \mid \mathscr{H}_A^- = \mathcal{X})$ from above. We start by obtaining
\begin{equation}\label{numtrans}
\begin{split}
&\mathbb{P} ( A \in \mathcal{C}_{{\mathcal{R}}(A)}^* , ~\mathscr{H}_A \cap \mathcal{X} = \emptyset  \mid \mathscr{H}_A^- = \mathcal{X} )\\
&~~~= \mathbb{P} ( A \in \mathcal{C}_{{\mathcal{R}}(A)}^* , ~\mathscr{H}_A \cap \mathcal{X} = \emptyset )~\leq~ \mathbb{P} ( A \in \mathcal{C}_{{\mathcal{R}}(A)}^*),
\end{split}
\end{equation}
where the first equation comes from the observation that  $\{A \in \mathcal{C}_{{\mathcal{R}}(A)}^*\} \cap \{ \mathscr{H}_A \cap \mathcal{X} = \emptyset\}$ is independent of  $ \{\mathscr{H}_A^- = \mathcal{X}\}$.

\quad Now we estimate the quantity $ \mathbb{P} ( A \in \mathcal{C}_{{\mathcal{R}}(A)}^*)$. Set $W_t = |\mathscr{H}_A(t_\star -t)| $ for each integer $t=0,1, \ldots, t_\star$ and $W_{t_\star +1}$=0. We also define
\begin{equation*}
T := \max \{0 \leq t \leq t_\star : ~ |\mathscr{H}_A (t) |=2 \},
\end{equation*}
and let $T = -1$ if $ |\mathscr{H}_A (t) | \geq 3$ for all $0\leq t \leq t_\star$. Note that on the event $ \{A \in \mathcal{C}_{\mathcal{R}(A)}^*\}$, $T$ is well-defined. Also, Recall that there cannot be a case with $W_t = 1$, since it can either be zero or at least $2$. Observe that if  $A \in \mathcal{C}_{\mathcal{R}(A)}^*$, then the following two events should also occur:

\begin{enumerate}
	\item [\textbf{1.}] The history starting from $\mathscr{H}_A (T)$ survives survive until $t=0$.
	
	\item [\textbf{2.}]The total number of spatial edges in $\{\mathscr{H}_A(s+\frac{1}{2}) :  s= T,  \ldots,\, t_\star -1 \} $ is at least $\mathfrak{M}(A)-1.$
\end{enumerate}

\quad The first event comes from the definition of red clusters. For the second one, note that if we project the histories $\{\mathscr{H}_A(s+\frac{1}{2}) : ~ s= T,  \ldots, t_\star -1 \}$ on the spatial space $G=(V,E)$, then we should have a connected subgraph of $G$ that contains $A$. Since the number of edges in such a graph is at least $\mathfrak{M}(A)-1$, the second event  should also take place. (Recall that $\mathfrak{M}$ denotes the size of the minimal connected subgraph containing $A$.)

\begin{claim}\label{claim:expdecay}
	For any $p>0$ satisfying $edp \leq \frac{1}{2}$, the probability of the first event conditioned on $T$ is at most $2(3edp)^T\wedge 1$.
\end{claim}

\begin{proof}
	Suppose that the history $\mathscr{H}_v$ from time $t$ to $s$ has been revealed and $(u,s) \in \mathscr{H}_v$. Then for any $k\geq1$, the probability that $(u,s)$ branches out to $k+1$ children at time $s-1$ is bounded by $(edp)^k$, as in Lemma \ref{gapofswd}. Therefore, we can stochastically dominate $W_t$ by  Galton-Watson branching process as follows. 
	
	\quad Let $\{\xi^{(i)} : i\in \mathbb{N} \}$ be the collection of i.i.d. random variables on integers with the distribution given by $\{r(k): k\in \mathbb{N} \}$ such that $r(k+1) = (edp)^k$ for $k\geq 2$, $r(1)=0$ and $r(0)=1-\sum_{k\geq 2} r(k)$. Then for each $t$, conditioned on $W_t$, we have
	\begin{equation}\label{gwrep}
	W_{t+1 } \preceq \xi^{(1)} + \ldots + \xi^{(W_t)}.
	\end{equation}
	Hence by Markov's inequality, the probability of the event \textbf{1} is bounded by $\mathbb{E} [W_{t_{\star}} | T] $. Thus, we conclude the proof by noticing
	\begin{equation}
	\mathbb{E}[W_{t_{\star}} |T] = 2\mathbb{E}[\xi^{(1)}]^T \leq 2(3edp)^T,
	\end{equation}
	where the last inequality holds for any positive $p$ with $edp\leq \frac{1}{2}$.
\end{proof}

\quad Note that the first event is independent of the histories from time $T+\frac{1}{2}$ to $t_\star - \frac{1}{2}$, whereas the second one is measurable with respect to these. Also, the event \textbf{2}  clearly implies 
\begin{equation}
\sum_{j=1}^{t_\star -T} W_j \geq \mathfrak{M}(A)+t_\star -T -1 \geq \mathfrak{M}(A),
\end{equation}
since the graph $\mathscr{H}_A (t_\star -t + \frac{1}{2})$ has at least $W_t-1$ edges. Therefore we obtain,
\begin{equation}\label{mainlemcomputation0}
\begin{split}
\mathbb{P} ( A \in \mathcal{C}_{{\mathcal{R}}(A)}^*) \leq \mathbb{E} \left[ \{2(3edp)^T\wedge 1 \}\; \mathbbm{1}_{\{W_1 + \ldots + W_{t_\star -T} \:\geq\: \mathfrak{M}(A) \} } \right].
\end{split}
\end{equation}
As we proceed, we get
\begin{equation}\label{mainlemcomputation1}
\begin{split}
\mathbb{E} &\left[ \{2(3edp)^T\wedge 1 \} \mspace{3mu} \mathbbm{1}_{\{W_1 + \ldots + W_{t_\star -T} \:\geq\: \mathfrak{M}(A)  \} } \right] \\
&\leq \sum_{k=0}^{t_\star} \mathbb{E} \left[ 2(3edp)^k \mspace{3mu} \mathbbm{1}_{\{W_1 + \ldots + W_{t_\star -T} \:\geq\: \mathfrak{M}(A)  \}} \mspace{3mu} \mathbbm{1}_{\{T=k \}}  \right]\\
&~~~~+\mspace{3mu}\mathbb{E} \left[ \mspace{3mu} \mathbbm{1}_{\{W_1 + \ldots + W_{t_\star } \:\geq\: \mathfrak{M}(A)  \}} \mspace{3mu} \mathbbm{1}_{\{T = -1 \}} \right].
\end{split}
\end{equation}

\noindent Note that the event $\{T =k \}$ implies $\{W_t \geq 3 \}$ for all $0 \leq t \leq t_\star - k -1$ and $\{W_{t_\star -k} \geq 2 \}$. We also introduce a constant $C$ satisfying $pC <1$ whose precise value will be specified later. Then, the r.h.s. of (\ref{mainlemcomputation1}) is at most
\begin{equation}\label{mainlemcomputation2}
\begin{split}
&\sum_{k=0}^{t_\star} \mathbb{E} \left[ 2(3edp)^k \mspace{3mu} \mathbbm{1}_{\left\{\sum_{j=1}^{t_\star-k} W_j \mspace{3mu}\geq\mspace{3mu} \mathfrak{M}(A) \right\}} \mspace{3mu} \mathbbm{1}_{\{W_1 , \mspace{1mu}\ldots\mspace{1mu}, \mspace{1mu} W_{t_\star -k -1} \geq 3 \}}  \mathbbm{1}_{\{ \mspace{1mu} W_{t_\star -k} \geq 2 \}} \right]\\
&~~+\mspace{3mu} \mathbb{E} \left[ \mspace{3mu} \mathbbm{1}_{\left\{\sum_{j=1}^{t_\star} W_j \mspace{3mu}\geq\mspace{3mu} \mathfrak{M}(A) \right\}}\mspace{3mu} \mathbbm{1}_{\{W_1 , \mspace{1mu}\ldots\mspace{1mu}, \mspace{1mu} W_{t_\star} \geq 3 \}} \right]\\
&\leq  \sum_{k=0}^{t_\star} 2(3edp)^k e^{-\lambda \mathfrak{M}(A) } (pC)^{3 t_\star -3k -1} \mathbb{E} \left[ \left(\frac{e^{\lambda}}{ pC} \right)^{ \sum_{j=1}^{t_\star-k} W_j   }  \right]\\
&~~~~~~~~+ e^{-\lambda \mathfrak{M}(A)} (pC)^{3t_\star} \mathbb{E} \left[ \left(\frac{e^{\lambda}}{ pC} \right)^{\sum_{j=1}^{t_\star} W_j   }  \right],
\end{split}
\end{equation}
where we used either  $\mathbbm{1} \{X \geq x \} \leq e^{\lambda (X-x)} $ or $\mathbbm{1} \{X \geq x \} \leq (pC)^{-(X-x)} $. Now we bound the terms $W_t$ as follows, based on the domination by the Galton-Watson process as in (\ref{gwrep}):
\begin{equation}\label{gwbound}
\begin{split}
\mathbb{E}& \left[ \left. \left(e^{2 \lambda} (pC)^{-1} \right)^{W_{t+1}} \right| W_t  \right] \leq \left[ 1 + \sum_{k \geq 1} \left(e^{2 \lambda} (pC)^{-1} \right)^{k+1} (edp)^k \right]^{W_t} \\
&~= \left(1+ \frac{e^{2\lambda}}{pC} \sum_{k \geq 1} \left( \frac{de^{2\lambda +1}}{C}\right)^k \right)^{W_t}
\leq \; \exp \left(\frac{2d e^{4\lambda +1}}{pC^2}  W_t \right) \leq \; e^{\lambda W_t},
\end{split}
\end{equation}
where we picked $C$ such that $\frac{2d e^{4\lambda +1}}{pC^2} = 1 \wedge \lambda$. Using (\ref{gwbound}), we deduce that
\begin{equation*}
\begin{split}
\mathbb{E}& \left[ \left(e^{\lambda} (pC)^{-1} \right)^{\sum_{j=1}^{t_\star-k} W_j  }  \right]~\leq \mathbb{E} \left[ \left(e^{\lambda} (pC)^{-1} \right)^{\sum_{j=1}^{t_\star-k-1} W_j  }  \cdot e^{\lambda W_{t_\star -k-1}} \right]\\
&~~~~~~~~~~~~~~~~=\mathbb{E} \left[ \left(e^{\lambda} (pC)^{-1} \right)^{\sum_{j=1}^{t_\star-k-2} W_j   }  \cdot \left(e^{2\lambda} (pC)^{-1} \right)^{\lambda W_{t_\star -k-1}} \right],
\end{split}
\end{equation*}
and iterating this inequality gives
\begin{equation}\label{mainlemcomputation3}
\begin{split}
\mathbb{E} \left[ \left(e^{\lambda} (pC)^{-1} \right)^{\sum_{j=1}^{t_\star-k} W_j  }  \right]
\leq \mathbb{E} \left[\left(e^{2 \lambda} (pC)^{-1} \right)^{W_1} \right] \leq \left(1 + \frac{2de^{4\lambda +1}}{pC^2} \right)^{W_0} \leq \;2^{|A|}.
\end{split}
\end{equation}
Finally,  (\ref{mainlemcomputation0})--(\ref{mainlemcomputation2}) combined with (\ref{mainlemcomputation3}) implies that
\begin{equation}\label{mainlemcomputaiton4}
\begin{split}
\mathbb{P} ( A \in \mathcal{C}_{{\mathcal{R}}(A)}^*) 
&\leq \left\{ \sum_{k=0}^{t_\star}  (pC)^{ -3k -1} (3edp)^k 2^{|A|}\; + \;  2^{|A|} \right\} e^{-\lambda \mathfrak{M}(A)} (pC)^{3t_\star} \\
&= \left\{  (pC)^{-1} 
\sum_{k=0}^{t_\star} \left(\frac{3ed}{p^2C^3} \right)^k 
+  1 \right\}e^{-(\lambda -1) \mathfrak{M}(A) }(pC)^{3t_\star}\\
&\leq 2 \, (pC)^{-1} (3edp)^{t_\star} e^{-(\lambda -1) \mathfrak{M}(A)} \leq \sqrt{6} e^{-2\lambda} (3edp)^{t_\star -\frac{1}{2}} e^{-(\lambda -1) \mathfrak{M}(A)},
\end{split}
\end{equation}
where we used $2^{|A|} \leq e^{\mathfrak{M}(A)} $ and $(ed/p^2C^3) >1$ by making $p$ smaller if needed. If we substitute $\lambda$ by $\widetilde{\lambda}+1$ in (\ref{mainlemcomputaiton4}) and set
\begin{equation*}
C(\widetilde{\lambda}, d,p)= (2d e^{4\widetilde{\lambda} +5}/p)^{1/2}; \;\;  M(\widetilde{\lambda}) = 3 e^{-2 \widetilde{\lambda}-2}; \;\;  p_0 (\widetilde{\lambda}, d) = \frac{1}{8d e^{12\widetilde{\lambda} + 13}},
\end{equation*}

\noindent then we deduce that for any $p<p_0$,
\begin{equation}\label{mainlemcomputation5}
\mathbb{P} ( A \in \mathcal{C}_{{\mathcal{R}}(A)}^*)
\leq M (3edp)^{t_\star -\frac{1}{2}} e^{\widetilde{\lambda}\mathfrak{M}(A) }.
\end{equation}

\noindent By combining (\ref{mainlemcomputation5}) and (\ref{lowerbddenom}), we have that
\begin{equation*}
\Psi_A ~ \leq ~M (3edp)^{t_\star -\frac{1}{2}} \left(e^{-\widetilde{\lambda}} (1- p)^{-d} \right)^{\mathfrak{M}(A)}.
\end{equation*}

\quad Therefore, by choosing $\widetilde{\lambda}$ that satisfies $e^{-\theta} = e^{-\widetilde{\lambda}} (1- p_0)^{-d} $, we obtain the conclusion that there are $M=M(\theta)$ and $p_0 = p_0 (\theta, d)$ such that
\begin{equation*}
\Psi_A ~ \leq ~ M (3edp)^{t_\star -\frac{1}{2}} e^{-\theta \, \mathfrak{M}(A)}
\end{equation*}
holds true for all $p<p_0$. \qed

\section{Reducing $L^1$ mixing to $L^2$-local mixing}\label{secl1l2}

We turn our attention to the case when the underlying graph has the lattice structure: having the lattice points as vertices and the nearest-neighbor connections as edges. In this section, we show that the total-variation distance of the Swendsen-Wang dynamics on $\mathbb{Z}_n^d$ from stationarity is essentially controlled by the $L^2$-distance of the dynamics on a smaller lattice $\mathbb{Z}_r^d$ from its stationarity.

\quad More precisely, consider the Swendsen-Wang dynamics $(X_t)$ for the $q$-Potts model on $\mathbb{Z}_n^d$, the $d$-dimensional lattice of side-length $n$ and with periodic boundary conditions, and let $\pi$ denote the Potts measure on $\mathbb{Z}_n^d$. Also, we assume that the parameter $p$ satisfies $p<p_0$, where $p_0$ is the constant given in Theorem \ref{l2thm}. Further, consider such a chain defined on a smaller lattice, namely ($X^{\dagger}_t$) on $\mathbb{Z}_r^d$ for $r= 3\log^5 n$, and let $\pi^{\dagger}$ denote its stationary distribution ($\mathbb{Z}_r^d$ is also endowed with periodic boundary conditions). Inside $\mathbb{Z}_r^d$, let $\Lambda \subset \mathbb{Z}_r^d$ be a subcube of side-length $2\log^5 n$, whose precise location is not of an interest due to the periodic boundary conditions. Define
\begin{equation}\label{mtdef}
\textbf{m}_t := \max_{x_0 \,\in \, \Omega_{\mathbb{Z}_r^d}} \left\| \mathbb{P}_{x_0} \left(X^{\dagger}_t (\Lambda) \in \, \cdot \, \right) - \pi^{\dagger}_{\Lambda} \right\|_{L^2 (\pi^{\dagger}_{\Lambda})},
\end{equation}

\noindent where $X_t^{\dagger}(\Lambda)$ and $\pi^{\dagger}_{\Lambda}$ denote the projection of $X_t^{\dagger}$ and $\pi^{\dagger}$ onto the cube $\Lambda$, respectively. Our goal is to prove Theorem \ref{l1l2thm}  which explains how the quantity $\textbf{m}_t$ governs the total-variation distance of $X_t$ from stationarity. Its statement is a variant of Theorem 3.1 of \cite{ls13}.

\begin{remark}\label{rrmk}
	\textnormal{We actually have a lot of freedom in the choice of $r$. Indeed, all the results in this section are identically applicable to any poly-logarithmic quantity which is at least $\log^{4+\delta} n$, for any positive constant $\delta$. This fact turns out to be useful later on in \S\ref{secfinal} when we specify the cutoff location as stated in Theorem \ref{thm1}. }
\end{remark}

\begin{theorem}\label{l1l2thm}
	Let $X_t$ be the Swendsen-Wang dynamics for the $q$-Potts model on $\mathbb{Z}_n^d$ with $p<p_0$, define $\textbf{m}_t$ as in (\ref{mtdef}), and set $p_{\star} = \frac{1}{\log (1/4edp)}$. Then the following holds true:
	
	\begin{enumerate}
		\item  Let $s=s(n)$ and $t=t(n)$ satisfy $(10d p_{\star}) \log \log n \leq s \leq \log^{4/3} n$ and $0<t \leq \log^{4/3} n$. For sufficiently large $n$,
		\begin{equation*}
		\max_{x_0 } \left\| \mathbb{P}_{x_0} \left( X_{t+s} \in \,\cdot \,\right) - \pi \right\|_{{ \textnormal{\tiny TV}}}
		\leq \frac{1}{2} \left[\exp \left(\left(n /\log^7 n \right)^d \textbf{m}_t^2 \right) -1 \right]^{\frac{1}{2}} + n^{-9d}.
		\end{equation*}
		In particular, if $\;(n/\log^7 n)^d \textbf{m}_t^2 \rightarrow 0\;$ as $n \rightarrow \infty$ for the above choice of $s$ and $t$, then
		\begin{equation*}
		\limsup_{n\rightarrow \infty} \max_{x_0 } \left\| \mathbb{P}_{x_0} \left( X_{t} \in \,\cdot \,\right) - \pi \right\|_{{ \textnormal{\tiny TV}}}=0.
		\end{equation*}

		\item If $\; (n/ 3\log^5 n)^d \textbf{m}_t^2 \rightarrow \infty\;$ for some $t \geq (20d p_{\star}) \log\log n$, then
		\begin{equation*}
		\liminf_{n\rightarrow \infty} \max_{x_0 } \left\| \mathbb{P}_{x_0} \left( X_{t} \in \,\cdot \,\right) - \pi \right\|_{{ \textnormal{\tiny TV}}}=1.
		\end{equation*}
	\end{enumerate}
\end{theorem}

\begin{remark}\label{l1l2rmk}
	The first statement of the theorem can be generalized to the case of $X_t'$ defined on a smaller lattice. To be specific, let $\log^5 n \leq m \leq n$ and let $X_t'$ denote the Swendsen-Wang dynamics on $\mathbb{Z}_m^d$. Then under the same assumptions on $s$ and $t$ as in Theorem \ref{l1l2thm},
	\begin{equation*}
	\max_{x_0 } \left\| \mathbb{P}_{x_0} \left( X_{t+s}' \in \,\cdot \,\right) - \pi \right\|_{{ \textnormal{\tiny TV}}}
	\leq \frac{1}{2} \left[\exp \left(m^d \textbf{m}_t^2 \right) -1 \right]^{\frac{1}{2}} + n^{-9d}.
	\end{equation*}
	This variant will be proven along with Theorem \ref{l1l2thm}, and turn out to be useful when establishing the explicit location of cutoff.
\end{remark}

\quad Theorem \ref{l1l2thm} is proven in \S\ref{secl1l22}. In \S\ref{secl1l21}, we describe a modified dynamics defined on smaller blocks and develop an argument to compare this with the original chain. To this end, we focus on explaining two major ingredients: 
\begin{itemize}
	\item  Eliminating the dependencies between distant sites;
	
	\item Restricting our attention on the ``update support'' which typically exhibits a nice ``sparse'' geometry as well as contains all the information of the dynamics.
\end{itemize}

\subsection{$L^1$ to $L^2$ reduction: ingredients}\label{secl1l21}

Throughout this section  we assume $p < p_0$, where $p_0$ is the constant given in Theorem \ref{l2thm}. We begin by observing two important properties of the high temperature Swendsen-Wang dynamics. The first fact we notice is that the information in the Swendsen-Wang dynamics cannot travel too fast.

\begin{lemma}\label{speedofpropa}
	Let $B\subset \mathbb{Z}_n^d$ be a lattice cube with side-length $\log^4 n$. Let
	\begin{equation*}
	B^+ := \{v \in \mathbb{Z}_n^d : \: \textnormal{dist}(v, B) \leq \log^3 n \}
	\end{equation*}
	where \textnormal{dist}$(\cdot,\cdot)$ denotes the $l^\infty$-distance in the lattice. Suppose that $(X_t), (X_t')$ are the two copies of the Swendsen-Wang dynamics coupled by the same update sequence with their initial configurations satisfying $X_t (B^+) = X_t' (B^+) $. Then for $t_{\textnormal{max}} := \log^{4/3}n$, $X_t (B) = X_t' (B)$ holds for all $t \leq t_{\textnormal{max}}$ except with an error probability $n^{-11d}$ for all sufficiently large $n$.
\end{lemma}

\begin{proof}
	Let $\{\Xi_{j} \}_{j=0}^{t_{\textnormal{max}}}$ be the collection of boxes constructed as follows:
	\begin{equation*}
	\Xi_{0} := B^+ , \:\:\:\:
	\Xi_{k+1} := \{v \in \Xi_{k} : \; \textnormal{dist}(v, \Xi_{k}^c) \geq \log^{3/2}n \}
	\:\:\:\: \textnormal{for all } 0 \leq k \leq t_{\textnormal{max}}.
	\end{equation*}
	
	\noindent Note that we have $\Xi_{t_{\textnormal{max}}} \supset B$ by definition. Also let the update sequence of $X_t$ and $X_t'$ be $\mathfrak{H}_{t_{\textnormal{max}}} = \left\lbrace (\bar{\omega}_t, (c_{v,t})_{v \in V} , \mathcal{A}_{\bar{\omega}_t} )  \right\rbrace_{t=0}^{t_{\textnormal{max}}-1}$ (recall Definition \ref{updseq}).
	
	\quad Let us estimate the probability of the event $\{X_{t+1} (\Xi_{t+1}) \neq {X}_{t+1}' (\Xi_{t+1})\}$ conditioned on being $X_t (\Xi_{t}) = {X}_t' (\Xi_{t})$. Suppose that a connected component $K$ of $(\mathbb{Z}_n^d,\:\bar{\omega}_t)$ satisfies that $K \cap \Xi_{t+1} \neq \emptyset$ and $K \subset \Xi_{t}$. In that case, $X_t(K)  =X_t'(K)$ implies that on $K$, they remain the same after an update, i.e., $X_{t+1} (\Xi_{t+1}) = {X}_{t+1}' (\Xi_{t+1})$.
	
	\quad Therefore, if the event $\{X_{t+1} (\Xi_{t+1}) \neq {X}_{t+1}' (\Xi_{t+1})\}$ occurs, we should have a connected component $K$ of $\bar{\omega}_t$ that satisfies $K \subsetneq \Xi_t$ as well as $K \cap \Xi_{t+1} \neq \emptyset$, which means that there is a percolation path in $\bar{\omega}_t$ that crosses between $\Xi_{t}^c$ and $\Xi_{t+1}$. Implementing the well-known fact that the crossing probability in a sub-critical percolation  decays exponentially with respect to the distance (see e.g., Theorem 6.75 of  \cite{grimmettperc}), we deduce that
	\begin{equation}\label{barrierineq1}
	\begin{split}
	\mathbb{P} &\left(  \left. X_{t+1} (\Xi_{t+1}) \neq {X}_{t+1}' (\Xi_{t+1}) \; \right|  X_t (\Xi_{t}) = {X}_t' (\Xi_{t})) \right)  \\
	&~~\leq
	\mathbb{P} \left(  \left.\Xi_{t}^c \overset{\bar{\omega}_t}{\longleftrightarrow} \Xi_{t+1}   \; \right|  X_t (\Xi_{t}) = {X}_t' (\Xi_t) \right) \leq
	\mathbb{P} \left(  \Xi_{t}^c \overset{\bar{\omega}_t}{\longleftrightarrow} \Xi_{t+1} \right)
	\leq ~ \exp \left(-c \log^{3/2} n  \right),
	\end{split}
	\end{equation}
	where $c >0$ is the constant depends on $d$ and $p$.
	
	\quad Therefore, by summing up the left hand side of (\ref{barrierineq1}) over $s$, we obtain
	\begin{equation*}
	\begin{split}
	\mathbb{P} \left(X_t (B) = X_t' (B) ~ \textnormal{for all } \, t\leq t_{\textnormal{max}} \right)
	&\geq
	\mathbb{P} \left(X_t (\Xi_{t}) = X_t' (\Xi_{t}) ~ \textnormal{for all } \, t \leq t_{\textnormal{max}} \right)\\
	&\geq 1- \left(\log^{4/3} n \right) \exp \left(-c_p \log^{3/2} n  \right).
	\end{split}
	\end{equation*}

	\noindent Finally, a crude union bound yields that
	\begin{equation}\label{speedofpropa1}
	\begin{split}
	\mathbb{P} \left( X_t = X'_t   ~ \textnormal{for all } \, t\leq t_{\textnormal{max}} \right)& \geq 1- n^d \left(\log^{4/3} n \right) \exp \left(-c_p \log^{3/2} n  \right)
	\\&\geq 1- n^{-11d},
	\end{split}
	\end{equation}
	where the last inequality holds for all $n$ sufficiently large.
\end{proof}

\quad This lemma tells us that until time $\log^{4/3}n$, we can possibly ignore the dependency between sites with distance at least $\log^3 n$.
Next, we observe that the dependency on the initial condition disappears quickly. 

\begin{lemma}\label{perfectsamp}
	Let $(X_t)$ be the Swendsen-Wang dynamics defined on $\mathbb{Z}_l^d$ and let $\mathscr{H}_{\mathbb{Z}_l^d}$ denote its  history diagram defined in Definition \ref{historydiag}. Then,
	\begin{equation*}
	\mathbb{P} \left(\mathscr{H}_{\mathbb{Z}_l^d} (0) = \emptyset \right) \geq 1- l^d (3edp)^t.
	\end{equation*}
	In particular, if $l= O(\log^5 n)$ and $t \geq 11 d p_\star \log\log n$, then we have
	\begin{equation}\label{perfectsampineq1}
	\mathbb{P} \left(\mathscr{H}_{\mathbb{Z}_l^d} (0) = \emptyset \right) \geq 1- (\log n)^{-5d}.
	\end{equation}
\end{lemma}

\vspace{1mm}

\begin{proof}
	By a union bound and the symmetry, we have
	\begin{equation}\label{perfectsampineq2}
	\begin{split}
	\mathbb{P} \left(\mathscr{H}_{\mathbb{Z}_l^d} (0) = \emptyset \right)
	\;\geq\; 1- l^d \mathbb{P} \left(\mathscr{H}_v (0) \neq \emptyset \right) .
	\end{split}
	\end{equation}
	
	\quad We then bound $\mathbb{P} \left(\mathscr{H}_v (0) \neq \emptyset \right)$ analogously as Claim \ref{claim:expdecay}, hence obtaining the estimate $$ \mathbb{P} \left(\mathscr{H}_v (0) \neq \emptyset \right) \leq (3edp)^t.$$
	For the final inequality, just note that the average offspring number of $\xi$ is smaller than $4edp$. Together with (\ref{perfectsampineq2}), we readily obtain our conclusion.
\end{proof}

\quad Based on the two fundamental properties of Lemmas \ref{speedofpropa} and \ref{perfectsamp}, we prove that the $L^1$-distance of the Swendsen-Wang dynamics from equilibrium at time $t+s$ can be bounded in terms of the $L^1$-distance at time $t$ projected onto subsets of \textit{sparse} geometry which  is defined as follows.


\begin{definition}[Sparse set]\label{sparsedef}
	Let $\log^5 n \leq m \leq n$. We say that the set $\Delta \subset \mathbb{Z}_m^d$ is \textbf{sparse} if for some $L \leq (n/\log^7 n)^d$ it can be partitioned into components $A_1,\ldots,A_L$ such that,

	\begin{enumerate}
		\item Each $A_i$ has diameter at most $\log^5 n$ in $\mathbb{Z}_m^d$.
		
		\item The $\| \,\cdot \, \|_{\infty}$-distance in $\mathbb{Z}_m^d$ between any distinct $A_i$ and $A_j$ is at least $2d \log^4 n$.
	\end{enumerate}
	We additionally define $\mathcal{S}(m) := \{\Delta \subset \mathbb{Z}_m^d :\; \Delta \textnormal{\; is sparse} \,\}$. 
\end{definition}

\quad This is a slightly modified version of the sparsity defined in Definition 3.3 of \cite{ls13}. Our goal of this subsection is to prove the following theorem:

\begin{theorem}\label{projthm}
	For $\log^5 n \leq m \leq n$, let $(X_t)$ be the Swendsen-Wang dynamcis for the Potts model on $\mathbb{Z}_m^d$ and $\pi$ be its stationary measure. Let $(11d p_{\star}) \log\log n \leq s \leq \log^{4/3}n$ and $t >0$. Then there exists some distribution $\rho$ on $\mathcal{S}(m)$ such that
	\begin{equation}\label{projeq}
	\begin{split}
	&\left\| \mathbb{P}_{x_0} \left( X_{t+s} \in \,\cdot \,\right) - \pi \right\|_{{ \textnormal{\tiny TV}}}\\
	&~~~\leq
	\int_{\mathcal{S}(m)} \left\| \mathbb{P}_{x_0} \left( X_{t}(\Delta) \in \,\cdot \,\right) - \pi_{\Delta} \right\|_{{ \textnormal{\tiny TV}}} d\rho(\Delta)
	+  3 n^{-10d}
	\end{split}
	\end{equation}
	holds true for all posible starting state $x_0$, where $p_\star := \frac{1}{\log(1/4edp)}$.
\end{theorem}

\quad To prove this, we first introduce the notion of \textit{induced update sequence}, which is necessary when coupling the two copies of the chain defined on different graphs. Using this we define the \textit{barrier-dynamics}, a variant of the chain that forcibly blocks the information coming from remote sites. It turns out that the barrier-dynamics resembles the original dynamics except for a negligible error, for which our focus turns to the investigation of the barrier-dynamics. This argument will be a variant of what is done in \cite{ls13}.


\begin{definition}[Induced update sequence]\label{inducedupd}
	Let $G=(V,E)$ and $G'=(V',E')$ be two graphs which are subgraphs of a same larger graph. Let $\mathfrak{H}_{t_{\star}} = \left\lbrace (\bar{\omega}_t, (c_{v,t})_{v \in V} , \mathcal{A}_{\bar{\omega}_t} )  \right\rbrace_{t=0}^{t_{\star}-1}$ denote the update sequence for the Swendsen dynamics on $G$ until time $t_\star$. Then the \textbf{induced update sequence} $\mathfrak{H}_{t_\star}'=\left\lbrace (\bar{\omega}_t', (c_{v,t}')_{v \in V} , \mathcal{A}_{\bar{\omega}_t'}' )  \right\rbrace_{t=0}^{t_{\star}-1}$ of $\mathfrak{H}_{t_\star}$ on $G'$ is defined as follows:
	
	\begin{enumerate}
		\item The percolation configuration $\bar{\omega}_t'$ is given by the following rule:
		\begin{equation*}
		\begin{split}
		\textnormal{for each } t, ~~
		\bar{\omega}_t' (e) =
		\begin{cases}
		\bar{\omega}_t (e)  &  \text{if } e \in E ;\\
		\textnormal{i.i.d. Ber} (p)  &  \text{if } e \in E' \setminus E;
		\end{cases}
		\end{split}
		\end{equation*}

		\item Let $K$ be an arbitrary connected component in $(V', \bar{\omega}_t')$.
		
		\begin{itemize}
			\item  If $K$ is also a connected component in $(V,\bar{\omega}_t)$, then
			$$c_{v,t}' = c_{v,t} \: \:\: \:\textnormal{and} \:\:\:\:\mathcal{A}_{\bar{\omega}_t'}' (v) = \mathcal{A}_{\bar{\omega}_t} (v), \:\:\:\: \textnormal{for all } v\in K.$$

			\item If $K$ is not a connected component in $(V,\bar{\omega}_t)$, then
			$$c_{v,t}' = \textnormal{i.i.d. Unif}\{1,\ldots,q\} ~\textnormal{ for all } v\in K,$$  which is independent from everything else, and
			
			$\left. \mathcal{A}_{\bar{\omega}_t'}' \right|_K : K \rightarrow \{1,\ldots, |K| \}$ $\:$ is given by an arbitrary ordering.
		\end{itemize}
	\end{enumerate}
\end{definition}

\vspace{2mm}

\begin{definition}[Barrier-dynamics]\label{barrierdef}
	Let $(X_t)_{0\leq t \leq t_\star}$ be the Swendsen-Wang dynamcis on $\mathbb{Z}_m^d$, for $m$ that satisfies  $\log^5 n \leq m \leq n$.  Also, let  the update sequence for $(X_t)_{0\leq t \leq t_\star}$ be  $\mathfrak{H}_{t_{\star}} = \left\lbrace (\bar{\omega}_t, (c_{v,t})_{v \in V} , \mathcal{A}_{\bar{\omega}_t} )  \right\rbrace_{t=0}^{t_\star-1}$.  Then we define the \textbf{barrier-dynamics} as the following coupled Markov chain.

	\begin{enumerate}
		\item Partition the lattice into disjoint $d$-dimensional (rectangular) boxes, where each of them has side-length either $\log^4 n$ or $\log^4 n -1$. We will call these boxes as ``blocks".

		\item For each block $B_i$, let $B_i^+$ be the $d$-dimensional box of side-lengths $\log^4 n + 2 \log^3 n$ centered at $B_i$,  e.g., if $B_i$ has side-lengths $\log^4 n$, then
		\begin{equation*}
		B_i^+ = \bigcup_{u\in B_i} \left\{v: \; \|u-v\|_{\infty} \leq \log^3 n \right\}.
		\end{equation*}
		For each $i$, let $\Phi_{i}$ be a graph isomorphism mapping $B_i^+$ onto some block $C_i^+$ and $B_i$ onto $C_i \subset C_i^+$, where the $C_i^+$ blocks are pairwise disjoint.

		\item Impose periodic boundary conditions on each block $C_i^+$ and run the Swendsen-Wang dynamics where the update sequence 
		$\mathfrak{H}_{t_\star}' (C_i^+)$ is given by the induced update sequence of $\mathfrak{H}_{t_\star} \circ \Phi_i^{-1}$ on $C_i^+$.

		\item The barrier-dynamics is the Swendsen-Wang dynamics on $\cup_i C_i^+$ with the update sequence given by $\mathfrak{H}_{t_\star}^b := \{\mathfrak{H}_{t_\star}'(C_i^+)  \}_i.$
		
	\end{enumerate}

\end{definition}

\begin{remark}
	Note that in the third step of the definition, the graph of $C_i^+$ is not a subgraph of $\mathbb{Z}_m^d$. As we impose a periodic boundary condition on $C_i^+$, there are newly added edges on the boundary of $C_i^+$. Therefore, in the percolation configuration of the induced update sequence, these new edges are endowed with i.i.d. Bernoulli random variables.
\end{remark}

\quad Based on the barrier-dynamics defined as above, we can construct a randomized operator $\mathcal{G}_t$ on $\mathbb{Z}_m^d$ as follows. Starting from a given initial configuration $x_0$ on $\mathbb{Z}_m^d$, we transform $x_0$ into a configuration on the block $C_i^+$ in the obvious manner: $x_0^{(i)} = x_0(B_i^+) \circ \Phi_i^{-1}$. Then we run the barrier-dynamics until time $t$ with initial state $x_0^{(i)}$ for each block. The output of $\mathcal{G}_t$ is then obtained by projecting the result of the barrier-dynamics onto $B_i$ for each $i$, i.e., we define the color at vertex $v$  to be the color at $\Phi_i (v)$ of the barrier-dynamics, where $B_i$ is the unique block that contains $v$. In other words, we pull-back the configuration from $C_i$ onto $B_i$ for each $i$. We denote this output as $\mathcal{G}_t (x_0)$ and call $\mathcal{G}_t$  the \textbf{barrier-dynamics operator}.

\quad If the time period is not too long, then the original dynamics $X_t$ can be coupled with $\mathcal{G}_t(X_0)$ except for a tiny error. The basic reason for this is that each block $B_i$
is far enough from $\partial B_i^+$.

\begin{lemma}\label{barrierdlem}
	Suppose that $\log^5 n \leq  m \leq n.$ Set $t_{\textnormal{max}} = \log^{4/3}n$. The barrier-dynamics and the original dynamics on $\mathbb{Z}_m^d$ are coupled up to time $t_{\textnormal{max}}$ except with probability $n^{-10d}$. That is,  we have $X_t = \mathcal{G}_t (X_0)$ for all $t \leq t_{\textnormal{max}}$ with probability at least $1- n^{-10d}$, for any starting configuration $X_0$ and for any sufficiently large $n$ .
\end{lemma}

\textit{Proof}.\quad Let $({X}_t^b)$ be the barrier-dynamics on $\cup_i C_i^+$ as defined in Definition \ref{barrierdef}. Since both the update sequence and the starting configuration of $(X_t)$ and $({X}_t^b)$ coincide on $B_i^+$ (where we identify $B_i^+$ and $C_i^+$ in the obvious way), we can apply Lemma \ref{speedofpropa} to these two processes. Therefore, we obtain
\begin{equation*}
\mathbb{P} \left(X_t (B_i) = {X}_t^b (B_i) ~ \textnormal{for all } \, t\leq t_{\textnormal{max}} \right)
\geq
1 - n^{-11d},
\end{equation*}
which holds for all $n$ sufficiently large. By a union bound over $B_i$, we get
\begin{equation*}
\pushQED{\qed}
\mathbb{P} \left( X_t = \mathcal{G}_t(X_0)   ~ \textnormal{for all } \, t\leq t_{\textnormal{max}} \right)
\geq 1- n^{-10d}. \qedhere
\popQED
\end{equation*}

\quad Thanks to Lemma \ref{barrierdlem},  we may focus on the barrier-dynamics rather than the original one when proving Theorem \ref{projthm}.  To this end, we introduce the notion of \textit{update support} and study its geometry as it is done in \cite{ls13, ls14}.

\quad As a first step towards the definition of update support, observe that the randomized operator $\mathcal{G}_t (x)$ can be regarded as a deterministic function of the starting configuration $x$ and the random update sequence $\mathfrak{H}_{t}^b$. Let us rewrite this as $\mathcal{G}_t (x) = g(x, \mathfrak{H}_{t}^b)$. Note that $g( \, \cdot \, ,\mathfrak{H}_{t}^b ) : \Omega_{\mathbb{Z}_m^d} \rightarrow \Omega_{\mathbb{Z}_m^d}$ is a deterministic function for each update sequence $\mathfrak{H}_{t}$. 

\begin{definition}[Update support]\label{updsupdef}
	Let $\mathfrak{H}_{t}^b$ be a realisation of an update sequence for the barrier-dynamics between times $[0,t]$. The \textbf{update support} of $\mathfrak{H}_{t}^b$ is the smallest subset $\Delta_{\mathfrak{H}_{t}^b} \subset \mathbb{Z}_m^d$ such that $\mathcal{G}_t (x)$ is a function of $x(\Delta_{\mathfrak{H}_{t}^b})$ for any $x$, i.e., there exists a function $f_{\mathfrak{H}_{t}^b}: \Omega_{\Delta_{\mathfrak{H}_{t}^b}} \rightarrow \Omega_{\Delta_{\mathfrak{H}_{t}^b}} $ such that
	\begin{equation*}
	g(x, \mathfrak{H}_{t}^b) = f_{\mathfrak{H}_{t}^b} ( \: x(\Delta_{\mathfrak{H}_{t}^b})) \:\:\: \textnormal{for all } \: x \in \Omega_{\mathbb{Z}_m^d}.
	\end{equation*}
	In other words, $v \notin \Delta_{\mathfrak{H}_{t}^b}$ if and only if for every initial configuration $x$, any color change at site $v$ does not affect the configuration $g(x, \mathfrak{H}_{t}^b) $ . This definition uniquely defines the update support of $\mathfrak{H}_{t}^b$.
\end{definition}

\quad Keeping Lemma \ref{perfectsamp} in mind, we may predict that  the update support shrinks considerably in a relatively short period of time, hence resulting in having sparse geometry. We prove that this indeed happens typically for the barrier-dynamics, following the approach of \cite{ls13}, Lemma 3.9.

\begin{lemma}\label{sparselem}
	Let $\mathcal{G}_s$ be the barrier-dynamics operator on $\mathbb{Z}_m^d$, let $\mathfrak{H}_s^b$ denote the update sequence of $\mathcal{G}$ up to time $s$ for some $s \geq (11d p_{\star}) \log \log n$, and $\mathcal{S}(m)$ be the family of sparse sets of $\mathbb{Z}_m^d$. Then $\mathbb{P} (\Delta_{\mathfrak{H}_s^b} \in \mathcal{S}(m)) \leq n^{-10d}$ for any sufficiently large $n$.
\end{lemma}

\begin{proof}
	Let $({X_t^b})$ be the barrier dynamics on $\mathbb{Z}_m^d$ and $\mathscr{H}_{\mathbb{Z}_m^d}^b$. For a block $B$ which appears in the first step of Definition \ref{barrierdef}, define $E_B$ to be the event that $\Delta_{\mathfrak{H}_s^b} \cap B \neq \emptyset$ for a random update sequence $\mathfrak{H}_s^b$. By the definition of the barrier-dynamics, ${X}_s^b (B^+)$ is not affected by $X_0(u) \:$ if  $ u \notin \cup_{\bar{B} \in N(B)} \bar{B}^+$, where $N(B)$ denotes the collection of the block $B$ and its neighboring blocks. Therefore, $\Delta_{\mathfrak{H}_s^b} \cap B = \emptyset \:$ if $\mathscr{H}_{\bar{B}^+}^b(0)= \emptyset  $ for every $\bar{B} \in N(B) $. Hence, we can apply Lemma \ref{perfectsamp} to deduce that
	\begin{equation}\label{ebbound}
	\begin{split}
	\mathbb{P}(E_B )
	\leq \mathbb{P} \left( \underset{\bar{B}\in N(B)}{\bigcup} \left\{ \mathscr{H}_{\bar{B}^+}^b(0)\neq \emptyset \right\} \right)
	\leq
	\;3^d  (\log n)^{-5d}
	\leq (\log n )^{-4d}.
	\end{split}
	\end{equation}
	
	In the following, we define two events $E^\sharp$ and $E^\flat$ whose union dominates the target event $\{\Delta_{\mathfrak{H}_s^b} \notin \mathcal{S}(m) \}$.
	
	\begin{itemize}
		\item $E^\sharp$ : there exists a collection $\mathcal{B}$ of $(n / \log^7 n)^d$ blocks such that $E_B$ holds for every $B\in \mathcal{B}$, and the pairwise distances between any two distinct blocks in $\mathcal{B}$ are at least 4 in blocks.

		\item $E^\flat$ : there exists a sequence of blocks $(B_{i_0}, B_{i_1}, \; \ldots \; B_{i_l})$ with $l \geq l_0 := \frac{1}{3d}\log n $ such that for all $k \leq l$, $B_{i_k} \cap \Delta_{\mathfrak{H}_s^b} \neq \emptyset$ and the distance between $B_{i_k}$ and $B_{i_{k-1}}$ is at most $3d$ in blocks.
	\end{itemize}
	We first show that $\{\Delta_{\mathfrak{H}_s^b} \notin \mathcal{S}(m) \} \subset E^\sharp \cup E^\flat$. Suppose that we have $\Delta_{\mathfrak{H}_s^b} \notin \mathcal{S}(m)$ but not $E^\sharp$. Then, we partition   $\Delta_{\mathfrak{H}_s^b}$ according to the following rule:
	
	\begin{center}
		$u \in B \cap \Delta_{\mathfrak{H}_s^b}$ and $u' \in B' \cap \Delta_{\mathfrak{H}_s^b}$ belong to the same component \\
		$\Longrightarrow$  ~the distance between $B$ and $B'$ is at most $3d$ in blocks.
	\end{center}
	In other words, 	$u \in B \cap \Delta_{\mathfrak{H}_s^b}$ and $u' \in B' \cap \Delta_{\mathfrak{H}_s^b}$ are in the same component if and only if there exists a sequence of blocks $(B=B_0, B_1, \ldots, B_k=B')$ such that for each $i$ the distance between $B_i $ and $B_{i+1}$ is at most $3d$ in blocks.
	
	\quad Under this partitioning, the number of components is less than $(n/ \log^7 n)^d$, since we are not in $E^\sharp$. Therefore, we can find a component of diameter greater than $\log^5 n$ because $\Delta_{\mathfrak{H}_s^b}$ is not sparse. Therefore, in that particular component, it is possible to find a sequence of $l_0$ blocks whose distance in blocks is at most $3d$ between each step, and hence $\Delta_{\mathfrak{H}_s^b} \in E^\flat$.

	\quad Next, we verify that both $\mathbb{P}(E^\sharp)$ and $\mathbb{P}(E^\flat)$ are small. Note that the events $E_B$ and $E_{B'}$ are independent if the distance between $B $ and $B'$ is at least 4 in blocks. This is because $E_B$ and $E_{B'}$ are  determined by the update sequence on $\cup_{\bar{B} \in N(B)} \bar{B}^+$ and $\cup_{\bar{B'} \in N(B')} \bar{B'}^+$, respectively. Utilizing (\ref{ebbound}), we can bound $\mathbb{P}(E^\sharp)$ by
	\begin{equation*}
	\mathbb{P}(E^\sharp) \leq {(2n /\log^4 n)^d \choose (n/ \log^7n)^d } (\log n)^{-4d (n/\log^7 n)^d}
	< \; n^{-(n/\log^8 n)^d} < n^{-11d}.
	\end{equation*}
	
	\quad We can bound $\mathbb{P}(E^\flat)$ similarly, using the independence of $E_B$ among distant boxes. Once it is done appropriately, we obtain
	$\mathbb{P}(E^\flat) \leq n^{-11d}$, where a complete calculation can be found in Lemma 3.9 of \cite{ls13}. Therefore, we conclude the proof by summing up two estimates on $\mathbb{P}(E^\sharp)$ and $\mathbb{P}(E^\flat)$.
\end{proof}

\quad We conclude this subsection by proving Theorem \ref{projthm}. To this end, we first show that the total-variation distance from stationarity at time $t+s$ can be bounded by its projection onto the update support at time $t$. The following lemma formalizes this approach, while its proof is omitted due to similarity to Lemma 3.8 of \cite{ls13}. In what follows, we also use the abbreviated form $\mathbb{P}(\mathfrak{H}_{t}^b)$ to denote the probability of having the specific update sequence $\mathfrak{H}_{t}^b$ between times $[0,t]$.

\begin{lemma}\label{projlem}
	Let $\log^5 n \leq m \leq  n$, and let $(X_t)$ be the Swendsen-Wang dynamics on $\mathbb{Z}_m^d$. Then for any $x_0$, $t>0$ and $0\leq s \leq \log^{4/3} n$,
	\begin{equation}\label{projeq1}
	\| \mathbb{P}_{x_0} (X_{t+s} \in \, \cdot \, ) -\pi \|_{\textnormal{\tiny{TV}}}
	\leq
	\int 	\| \mathbb{P}_{x_0} (X_{t}(\Delta_{\mathfrak{H}_s^b})  \in \, \cdot \, ) -\pi_{\Delta_{\mathfrak{H}_s^b}} \|_{\tiny{\textnormal{\tiny{TV}}}} \:d\mathbb{P} (\mathfrak{H}_s^b) \:+ \: 2 n^{-10d} ,
	\end{equation}
	where $\mathfrak{H}_s^b$ denotes the update sequence of the barrier-dynamics over the time period $[t,t+s]$ and $\Delta_{\mathfrak{H}_s^b}$ is its update support.
\end{lemma}

\textit{Proof of Theorem \ref{projthm}.}~ This is obtained as a direct consequence of Lemmas \ref{sparselem} and \ref{projlem}. Since we have $s \geq 10d p_\star \log\log n$, $ \Delta_{\mathfrak{H}_s^b}$ is a sparse set except with an error at most $n^{-10d}$. Keeping this in mind, let $\rho$ be the probability measure on $\mathcal{S}(m)$ defined as
\begin{equation*}
\rho(S) := \mathbb{P}( \Delta_{\mathfrak{H}_s^b} \in S \mid  \Delta_{\mathfrak{H}_s^b} \in \mathcal{S}(m) ).
\end{equation*}
\quad By plugging $\rho $ instead of $d\mathbb{P}( \mathfrak{H}_s^b)$ into (\ref{projeq1}) along with compensating the error $n^{-10d}$, we deduce (\ref{projeq})
as a conclusion. \qed

\subsection{Proof of Theorem \ref{l1l2thm}}\label{secl1l22}

Our approach is very similar to Theorem 3.1 of \cite{ls13}. In order to cover some necessary changes, we present a proof for part 1 of the theorem. However, for the second part, we refer  to the literature instead of reproducing it due to its similarity.

\quad To begin with, we introduce an elementary lemma on the $L^2$-distance between the product measures. This lemma will be useful when dealing with product chains which will occasionally appear later on.

\begin{lemma}\label{prodmsrlem}
	Let $(\mu_i)_{i=1}^k$ and $(\nu_i)_{i=1}^k$ be two collections of probability measures on a discrete state space, and let $\mu = \otimes_{i=1}^k \mu_i$ and $\nu = \otimes_{i=1}^k \nu_i $. Then the following inequality holds true:
	\begin{equation*}
	\|\mu - \nu \|^2_{L^2(\nu)} \leq \exp \left(\sum_{i=1}^k \|\mu_i - \nu_i \|^2_{L^2(\nu_i)} \right) -1.
	\end{equation*}
\end{lemma}

\textit{Proof}.\quad Proof is done by elementary calculations on $\|\mu - \nu \|_{L^2(\nu)}$.
\begin{align*}
\|\mu - \nu \|^2_{L^2(\nu)} &= \sum_x \frac{\mu(x)^2}{\nu(x)} -1
=\prod_{i=1}^k \left[\sum_{x_i} \frac{\mu_i(x_i)^2}{\nu_i(x_i)} \right] \:-1 \\
&=\prod_{i=1}^k
\{ \|\mu_i - \nu_i \|^2_{L^2(\nu_i)} +1 \} \:-1\leq
\exp \left(\sum_{i=1}^k \|\mu_i - \nu_i \|^2_{L^2(\nu_i)} \right) -1. \rlap{$~~~\:~~~\:~\square$}
\end{align*}

\textit{Proof of Theorem \ref{l1l2thm}, Part 1}.
~Let $m$ be a fixed integer such that $\log^5 n \leq m \leq n.$  Let $\Delta \subset \mathbb{Z}_m^d$ be a sparse set, and let $\cup_{i=1}^L A_i$ denote its partition according to Definition \ref{sparsedef}, where we have $L \leq m^d \wedge (n/\log^7 n)^d$ by definition. For each component $A_i$, define
\begin{equation*}
A_i^+ := \{v: \textnormal{dist}(v, A_i) \leq \log^3 n \}.
\end{equation*}
For each $i$, let $\psi_i$ be the graph isomorphism mapping $A_i^+$ onto $B_i^+$, where each $B_i^+$ is contained in distinct tori $\mathbb{Z}_r^d$ with $r = 3\log^5 n$. Let $\Gamma = \cup_{i=1}^L B_i.$ and let $\psi$ denote the combined information of $\psi_i$'s, i.e., $\psi|_{A_i^+}  = \psi_i$

\quad Let us define $(X_t^*)$ to be the product chain of the Swendsen-Wang dynamics on $(\mathbb{Z}_r^d)^L$, and let $\pi^*$ denote its stationary distribution. We couple $X_t^*$ and $X_t$ in a natural way as follows. For the initial configuration $x_0$ of $(X_t)$, define $x_0^*$  to be $x_0^* ( B_i^+) = x_0(A_i^+) \circ \psi_i^{-1}$ for each $i$, and endow arbitrary colors for the rest of the sites. Also, the update sequence of $(X_t^*)$ is given by the induced update sequence of $(X_t)$ as defined in Definition \ref{inducedupd}.  Then the triangle inequality implies that
\begin{equation}\label{tvineq1}
\begin{split}
\left\| \mathbb{P}_{x_0} \left( X_{t}(\Delta) \in \cdot \,\right) - \pi_{\Delta} \right\|_{{ \textnormal{\tiny TV}}}
&\leq
\left\| \mathbb{P}_{x_0} \left( X_{t}(\Delta) \in \cdot \,\right) - \mathbb{P}_{x_0^*} \left( X_{t}^*(\Gamma) \in \cdot \,\right) \right\|_{{ \textnormal{\tiny TV}}} \\
&~~+
\left\| \mathbb{P}_{x_0^*} \left( X_{t}^*(\Gamma) \in \cdot \,\right) - \pi_{\Gamma}^* \right\|_{{ \textnormal{\tiny TV}}} 
+
\left\| \pi_{\Delta} - \pi_{\Gamma}^* \right\|_{{ \textnormal{\tiny TV}}},
\end{split}
\end{equation}
where we omit the expression such as $\psi^{-1}$ since the correspondence between $\Delta$ and $\Gamma$ is clear in the current context.

\quad Note that the distance between $A_i$ and $\partial A_i^+$ is at least $\log^3 n$. By identifying $A_i$ and $B_i$ in an obvious way, Lemma \ref{speedofpropa} implies that $X_t^*$ and $X_t$ are coupled on $\cup_{i=1}^L A_i$ until time $t_{\textnormal{max}}=\log^{4/3} n$, with an error probability at most $n^{-10d}$ (this property can be proven analogously as in Lemma \ref{barrierdlem}).
This shows that
\begin{equation}\label{tvineq2}
\left\| \mathbb{P}_{x_0} \left( X_{t}(\Delta) \in \,\cdot \,\right) - \mathbb{P}_{x_0^*} \left( X_{t}^*(\Gamma) \in \,\cdot \,\right) \right\|_{{ \textnormal{\tiny TV}}} \leq n^{-10d}.
\end{equation}

\quad The third term in the r.h.s. of (\ref{tvineq1}) is split into three parts as follows.
\begin{equation}\label{tvineq3}
\begin{split}
\left\| \pi_{\Delta} - \pi_{\Gamma}^* \right\|_{{ \textnormal{\tiny TV}}}
&\leq
\left\| \mathbb{P}_{x_0^*} \left( X_{t_{\textnormal{max}}}^*(\Gamma) \in \,\cdot \,\right) - \pi_{\Gamma}^* \right\|_{{ \textnormal{\tiny TV}}} \\
&+
\left\| \mathbb{P}_{x_0^*} \left( X_{t_{\textnormal{max}}}^*(\Gamma) \in \,\cdot \,\right) -
\mathbb{P}_{x_0} \left( X_{t_{\textnormal{max}}}(\Delta) \in \,\cdot \,\right)   \right\|_{{ \textnormal{\tiny TV}}} \\	
&+
\left\| \mathbb{P}_{x_0} \left( X_{t_{\textnormal{max}}}(\Delta) \in \,\cdot \,\right) -
\pi_{\Delta} \right\|_{{ \textnormal{\tiny TV}}}.
\end{split}
\end{equation}

\noindent An analogous method as (\ref{tvineq2}) can be used to bound the second term in the r.h.s. of (\ref{tvineq3}). For the third term, we apply Lemma \ref{perfectsamp} to obtain that
\begin{equation}\label{tvineq4}
\begin{split}
\left\| \mathbb{P}_{x_0} \left( X_{t_{\textnormal{max}}}(\Delta) \in \cdot \,\right) -
\pi_{\Delta} \right\|_{{ \textnormal{\tiny TV}}}
&\leq
\left\| \mathbb{P}_{x_0} \left( X_{t_{\textnormal{max}}} \in \cdot \,\right) - \mathbb{P}_{\pi} \left( X_{t_{\textnormal{max}}} \in \cdot \,\right) \right\|_{{ \textnormal{\tiny TV}}}\\
&\leq \mathbb{P} (X_{t_{\textnormal{max}}} \textnormal{ is dependent on } X_0)  \leq n^{-10d},	
\end{split}
\end{equation}

\noindent where the last inequality is obtained by putting $l=m$ and $t=t_{\textnormal{max}}=\log^{4/3}n$ into (\ref{perfectsampineq1}). For the first term of (\ref{tvineq3}), we apply Lemma \ref{prodmsrlem} to deal with the product chain:
\begin{equation}\label{tvineq5}
\begin{split}
\| \mathbb{P}_{x_0^*} &\left( X_{t_{\textnormal{max}}}^*(\Gamma) \in \,\cdot \,\right) - \pi_{\Gamma}^* \|_{{ \textnormal{\tiny TV}}}
\leq
\frac{1}{2}\left\| \mathbb{P}_{x_0^*} \left( X_{t_{\textnormal{max}}}^*(\Gamma) \in \,\cdot \,\right) - \pi_{\Gamma}^* \right\|_{L^2(\pi^*)}\\
&~~~~~\leq
\frac{1}{2}\left[ \exp \left\{\sum_{i=1}^L  \| \mathbb{P}_{x_0^*} ( X_{t_{\textnormal{max}}}^*(B_i) \in\; \cdot \;) - \pi_{i}^* \|^2_{L^2(\pi_{i}^*)}   \right\} -1 \right]^{1/2}\\
&~~~~~\leq
\frac{1}{2}\left[ \exp \left\{L \; \| \mathbb{P}_{x_0^*} ( X_{t_{\textnormal{max}}}^*(\mathbb{Z}_r^d) \in \,\cdot \,) - \pi_r^* \|^2_{L^2(\pi_r^*)}   \right\} -1 \right]^{1/2},
\end{split}
\end{equation}

\noindent where $\pi_i^*$ and $\pi_r^*$ are shorthand notations for $\pi_{B_i}^*$ and $\pi_{\mathbb{Z}_r^d}^*$, respectively. In the  last inequality we used  the fact that projection can only decrease the $L^2$-distance while the first line is due to Cauchy-Schwarz. Now by plugging $t=t_{\textnormal{max}}=\log^{4/3}n$ into (\ref{l2ineq}) of Theorem \ref{l2thm}, we get
\begin{equation}\label{l2ineqappl}
\| \mathbb{P}_{x_0^*} ( X_{t_{\textnormal{max}}}^*(\mathbb{Z}_r^d) \in \,\cdot \,) - \pi_r^* \|_{L^2(\pi_r^*)} \leq n^{-11d}.
\end{equation}

\quad Using this combined with (\ref{tvineq5}) gives
\begin{equation}\label{tvineq6}
\left\| \mathbb{P}_{x_0^*} \left( X_{t_{\textnormal{max}}}^*(\Gamma) \in \,\cdot \,\right) - \pi_{\Gamma}^* \right\|_{{ \textnormal{\tiny TV}}}
\leq \frac{1}{2} \left[\exp \{L n^{-22d} \} -1\right]^{1/2} \leq n^{-10d},
\end{equation}

\noindent which holds for all sufficiently large $n$. Thus, we can rewrite (\ref{tvineq3}) using (\ref{tvineq4}, \ref{tvineq6}) as
\begin{equation}\label{tvineq7}
\left\| \pi_{\Delta} - \pi_{\Gamma}^* \right\|_{{ \textnormal{\tiny TV}}} \leq 3n^{-10d}.
\end{equation}

\noindent Hence, by combining (\ref{tvineq1}), (\ref{tvineq2}) and (\ref{tvineq6}), we get
\begin{equation}\label{tvineq8}
\left\| \mathbb{P}_{x_0} \left( X_{t}(\Delta) \in \,\cdot \,\right) - \pi_{\Delta} \right\|_{{ \textnormal{\tiny TV}}}
\leq
\left\| \mathbb{P}_{x_0^*} \left( X_{t}^*(\Gamma) \in \,\cdot \,\right) - \pi_{\Gamma}^* \right\|_{{ \textnormal{\tiny TV}}} + 4 n^{-10d}.
\end{equation}

\quad We derive an upper bound on the r.h.s. of (\ref{tvineq8}) in terms of $\textbf{m}_t$ similarly as what is done in (\ref{tvineq5}). Note that the diameter of $B_i^+$ is smaller than $\frac{2}{3}r = 2\log^5 n$, and hence we have
\begin{equation*}
\| \mathbb{P}_{x_0^*} ( X_{t}^*(B_i) \in\; \cdot \;) - \pi_{i}^* \|_{L^2(\pi_{i}^*)}
\leq
\;\textbf{m}_t.
\end{equation*}

\noindent Therefore, plugging this  into (\ref{tvineq5}) with replacing $t_{\textnormal{max}}$ by $t$ gives that
\begin{equation}\label{tvineq9}
\| \mathbb{P}_{x_0^*} \left( X_{t}^*(\Gamma) \in \,\cdot \,\right) - \pi_{\Gamma}^* \|_{{ \textnormal{\tiny TV}}} \leq \frac{1}{2} \left( \exp\left(L \textbf{m}^2_t \right) -1 \right)^{1/2},
\end{equation}

\noindent and this holds regardless of the initial configuration $x_0^*$. Altogether, (\ref{tvineq8}), (\ref{tvineq9}) and Theorem \ref{projthm} imply that
\begin{equation}\label{mtineqfinal}
\max_{x_0} \left\| \mathbb{P}_{x_0} \left( X_{t+s} \in \,\cdot \,\right) - \pi \right\|_{{ \textnormal{\tiny TV}}}
\leq
\frac{1}{2} \left( \exp\left(L \textbf{m}^2_t \right) -1 \right)^{1/2} +7n^{-10d}.
\end{equation}

\noindent Finally, recalling that $L \leq m^d \wedge (n/\log^7 n)^d$ and replacing $7n^{-10d}$ by $n^{-9d}$ establishes the first part of Theorem \ref{l1l2thm}. This verifies the variant version in Remark \ref{l1l2rmk} as well. \qed

\vspace{3mm}	
\textit{Proof of Theorem \ref{l1l2thm}, Part 2.}
~Since the proof is identical to that of Theorem 3.1 in \cite{ls13}, we refer to the literature rather than rewriting it in the current paper. However, we explain two minor changes that should be made for our case.

\quad Firstly, we divide the underlying lattice $\mathbb{Z}_n^d$ into blocks of side-length $3\log^5 n$, in contrast to the side-length $3\log^3 n$ blocks
in \cite{ls13}.	 Also, whenever the log-Sobolev-type inequality (Theorem 2.1 of \cite{ls13}) is used in the reference, we implement Theorem \ref{l2thm} or Lemma \ref{perfectsamp} as an alternative. It is applied when bounding the terms such as $	\left\| \mathbb{P}_{x_0} \left( X_{t_{\textnormal{max}}} \in \,\cdot \,\right) -
\pi \right\|_{{ \textnormal{\tiny TV}}} $, which can be done as (\ref{tvineq4}) and (\ref{l2ineqappl}) in our case. \qed

\section{Cutoff for the Swendsen-Wang Dynamics}\label{secfinal}

In this section we prove Theorems \ref{thm1} and  \ref{thm2}. In \S\ref{secfinal1}, we establish the existence and the location of cutoff. However, the cutoff location will be written in terms of a finite-volume spectral gap. In \S\ref{secfinal2}, we prove that the spectral gap of the finite-volume dynamics indeed converges to the infinite-volume gap, which verifies both Theorems \ref{thm1} and \ref{thm2}.

\begin{remark}\label{gaplimitrmk}
	\textnormal{In \S\ref{secglobalcoup}, we showed that the spectral gap is bounded strictly away from 0 uniformly in $n$ when $p<p_0$. On the other hand, one can also verify that the spectral gap is strictly away from 1 if $p$ is sufficiently small, whose proof is defered to Proposition \ref{specgapprop} in the appendix. Thus, if we can demonstrate that the spectral gap converges as the lattice size tends to infinity, we consequently have that the limit is strictly between $0$ and $1$.}
\end{remark}

\quad Throughout this section, let $0< p_0 < \frac{1}{4e^2d}$ denote a small constant that not only satisfies the condition given in  Theorem \ref{l2thm}, but also lies in the regime where there exists a constant $c>0$ such that $\gamma(r)$, the spectral gap of the Swendsen-Wang dynamics on $\mathbb{Z}_r^d$, lies in $[c, 1-c]$ uniformly in $r$, as discussed in the above remark. Moreover, for a given constant $0<p<p_0$, the following notations are introduced for convenience:
\begin{equation}\label{gammastar}
\gamma_\star (r) : = \log \left(\frac{1}{1-\gamma(r)} \right), ~~~~~p_\star := \left[\log \left( \frac{1}{4edp}\right) \right]^{-1}.
\end{equation}

\noindent Note that  $p_0 \leq \frac{1}{4e^2d}$ implies $p_\star \leq 1$. Then Proposition \ref{prop1} implies that $\gamma_\star(r) \geq p_\star^{-1} \geq 1.$

\subsection{Existence of Cutoff}\label{secfinal1} Our starting point is to sum up the results in the previous sections and derive a sharp bound on $\textbf{m}_t$ defined in (\ref{mtdef}), which will then naturally imply the existence of cutoff.

\begin{lemma}\label{mtestim}
	Set $r=3\log^5 n$. For every $0<p<p_0$, $18d \log\log n \leq t \leq \log^{4/3} n$ and $n$ sufficiently large,
	\begin{equation*}
	e^{-\gamma_\star(r) t \;- \;15d \gamma_\star(r)\log\log n } - n^{-9d}
	\leq
	\textbf{m}_t
	\leq
	e^{-\gamma_\star(r)t \;+\; 12d \gamma_\star(r) \log\log n }.
	\end{equation*}
\end{lemma}

\begin{proof}
	Let $X_t^\dagger$ denote the Swendsen-Wang dynamics on $\mathbb{Z}_r^d$ with periodic boundary conditions, and let $\pi^\dagger$ be its stationary distribution. Then the r.h.s. of the desired inequality comes directly from Theorem \ref{l2thm} and Remark \ref{windowrep}.
	\begin{equation}\label{mtineq1}
	\begin{split}
	\textbf{m}_t
	\leq
	\max_{x_0} \| \mathbb{P}_{x_0} ( X_t^\dagger \in \cdot) - \pi^\dagger \|_{L^2(\pi^\dagger)}
	&\leq 2 e^{ - \gamma_\star(r) \left( t- 11p_\star \log\log n \right)  }\\
	&\leq
	e^{- \gamma_\star(r)  t \;+\; 12d \gamma_\star(r)  \log\log n}  .
	\end{split}
	\end{equation}
	
	\quad Further, note that $r^d \textbf{m}_t = o(1)$ due to the condition $t \geq 18d \log \log n$. Therefore, combining Theorem \ref{l1l2thm} with $m=r$ (see Remark \ref{l1l2rmk}) and Proposition \ref{specgap} implies that
	\begin{equation}\label{mtineq11}
	\begin{split}
	e^{-\gamma_\star(r) (t+s)}
	&\leq
	2\| \mathbb{P} (X_{t+s}^\dagger \in \; \cdot \;) - \pi^\dagger \|_{\textnormal{\tiny{TV}}} \\
	&\leq
	\left(\exp \left( r^d \textbf{m}_t^2 \right) -1\right)^{1/2} + n^{-9d}
	\leq~
	2 r^{d/2} \textbf{m}_t + n^{-9d},
	\end{split}
	\end{equation}
	where $s= 11d p_\star \log\log n$, and the last inequality is achieved by the elementary inequality $e^x -1 \leq 4x$ which holds for $x \in [0,1]$. Using the fact that $\gamma_\star(r) \geq p_\star^{-1} \geq 1 $, we deduce that
	\begin{equation}\label{mtineq2}
	\textbf{m}_t \geq e^{-\gamma_\star(r) t \;- \;15d \gamma_\star(r)\log\log n } - n^{-9d}.
	\end{equation}
	Combining the inequalities (\ref{mtineq1}) and (\ref{mtineq2}) concludes the proof.
\end{proof}

\quad We can now prove the existence of cutoff in the following theorem, establishing the cutoff location in terms of $\gamma_\star(r)$ and the $O(\log\log n)$-window.

\begin{theorem}\label{cutoffthm}
	Let $(X_t)$ be the Swendsen-Wang dynamics defined on $\mathbb{Z}_n^d$. Set $r=3\log^5 n$, $0<p<p_0$, and let  $t_\star$, $t_n^-$ and $t_n^+$ be defined as follows:
	\begin{equation*}
	t_\star = t_\star(n) := \frac{d}{2\gamma_\star(r)} \log n, ~~~ t_n^- := t_\star - 18 d \log \log n , ~~~
	t_n^+ := t_\star + 20d \log\log n.
	\end{equation*}
	Then we have the following which establishes cutoff of $(X_t)$.
	\begin{equation}\label{cutoffeq}
	\begin{split}
	&\lim_{n\rightarrow \infty} \; \max_{x_0} \| \mathbb{P}_{x_0} (X_{t_n^-} \in \;\cdot\; ) - \pi \|_{\textnormal{\tiny{TV}}} = 1;\\
	&\lim_{n\rightarrow \infty} \; \max_{x_0} \| \mathbb{P}_{x_0} (X_{t_n^+} \in \;\cdot\; ) - \pi \|_{\textnormal{\tiny{TV}}} = 0.
	\end{split}
	\end{equation}
\end{theorem}

\begin{proof}
	The proof is a straightforward application of Theorem \ref{l1l2thm} and Lemma \ref{mtestim}. The latter one combined with the fact $\gamma_\star(r) \geq 1$ implies that
	\begin{equation*}
	\begin{split}
	(n/3\log^5 n)^d \textbf{m}_{t_n^-}^2 &\geq \frac{1}{3^d} \log^d n ~\longrightarrow \infty ; \\
	(n/\log^7 n)^d \textbf{m}_{t_n^+ -s}^2 &\leq \log^{-d} n ~~~\mspace{1mu}\longrightarrow 0,
	\end{split}
	\end{equation*}
	as $n$ tends to infinity, where $s:= 11d\log\log n$. Therefore, Theorem \ref{l1l2thm} shows that the two equations in (\ref{cutoffeq}) are true.
\end{proof}

\subsection{Limit of spectral gaps}\label{secfinal2}

In this final subsection, we verify Theorem \ref{thm2} and conclude the proof of Theorem \ref{thm1}. To this end, we apply Theorem \ref{cutoffthm} to varying values of $r$ to prove the convergence of $\{\gamma_\star(r) \}$.

\quad Although all our argument has been formulated in terms of $r=3\log^5 n$, it can be extended naturally to $r = \log^{4+\delta} n$ for any constant $\delta>0$, maintaining the window of size $O(\log \log n)$. (see Remark \ref{rrmk})  We can state this as follows:

\begin{corollary}\label{cutoffcor}
	Let $(X_t)$ be the Swendsen-Wang dynamics on $\mathbb{Z}_n^d$ and let $\delta>0$ be any small constant. Set $r_1 := \log^{4+\delta} n$, $0<p<p_0$ and let $\gamma_\star$ be defined as (\ref{gammastar}). Then there exists $C=C(d, \delta)>0$ such that the following holds for all $r_1 \leq r \leq r_1^2$ : For the parameters $t_\star $, $t_n^-$ and $t_n^+$ given by
	\begin{equation*}
	t_\star = t_\star(r) := \frac{d}{2\gamma_\star(r)} \log n, ~~ t_n^- := t_\star - C \log \log n , ~~
	t_n^+ := t_\star + C\log\log n,
	\end{equation*}
	we have
	\begin{equation}\label{cutoffeq1}
	\begin{split}
	&\lim_{n\rightarrow \infty} \; \max_{x_0} \| \mathbb{P}_{x_0} (X_{t_n^-} \in \;\cdot\; ) - \pi \|_{\textnormal{\tiny{TV}}} = 1;\\
	&\lim_{n\rightarrow \infty} \; \max_{x_0} \| \mathbb{P}_{x_0} (X_{t_n^+} \in \;\cdot\; ) - \pi \|_{\textnormal{\tiny{TV}}} = 0.
	\end{split}
	\end{equation}
\end{corollary}

\noindent Implementing this generalization, we can now prove the following proposition which is the first step towards establishing Theorem \ref{thm2}.

\begin{proposition}\label{convprop}
	Let $(X_t)$ be the Swendsen-Wang dynamics defined on $\mathbb{Z}_n^d$, set $0<p<p_0$, and let $\gamma_\star$ be defined as (\ref{gammastar}).  Then  there exists a constant $\hat{\gamma_\star} \in (0,1)$ such that
	\begin{equation*}
	|\gamma_\star(r) - \hat{\gamma_\star} | \leq  2 \; r^{-1/4 +\delta},
	\end{equation*}
	which holds for any constant $\delta > 0.$
\end{proposition}

\begin{proof}
	Our proof uses the approach of \cite{ls13}, Lemma 4.3. Let $r_1 = 3\log^{4+\delta} n$ and pick $r_2$ such that $r_1 \leq r_2 \leq r_1^2$. Then, Corollary \ref{cutoffcor} implies that
	\begin{equation*}
	\frac{d}{2\gamma_\star(r_1)} \log n - C \log\log n
	\leq
	\frac{d}{2\gamma_\star(r_2)} \log n + C\log\log n.
	\end{equation*}
	
	\noindent By rearranging the terms, the uniform boundedness of $\gamma_\star (r)$ (Proposition \ref{specgapprop}) gives that
	\begin{equation*}
	\gamma_\star(r_2) -\gamma_\star(r_1) \leq
	\frac{4\gamma_\star(r_1)\gamma_\star(r_2)}{d} \frac{C\log \log n}{\log n}
	\leq \: r_1^{-1/4 +\delta },
	\end{equation*}
	
	\noindent where the last inequality holds for all sufficiently large $n$. Since the role of $r_1$ and $r_2$ can clearly be reversed, we deduce that for any large $n$,
	\begin{equation*}
	\max_{r_1 \leq r \leq r_1^2}|\gamma_\star(r_1) -\gamma_\star(r) |
	\leq r_1^{-1/4 + \delta}.
	\end{equation*}
	Therefore, by iterating this inequality, we obtain
	\begin{equation*}
	\sum_{i\geq 0} |\gamma_\star(r_1^{2^i}) -\gamma_\star(r_1^{2^{i+1}}) |
	\:\leq\:
	\sum_{i\geq 0} r_1^{-2^{i-2} (1-4\delta)} \:\leq\: 2 r_1^{-1/4 +\delta} \: <\infty,
	\end{equation*}
	which implies the existence of the limit $\hat{\gamma_\star} : = \lim_{n \rightarrow \infty} \gamma_\star(r)$ as well as
	\begin{equation*}
	|\gamma_\star(r) - \hat{\gamma_\star} | \leq  2 \;r^{-1/4 + \delta}.
	\end{equation*}
	
	\noindent The property $0< \hat{\gamma_\star}<1 $ follows by Propositions  \ref{prop1} and \ref{specgapprop}, whose statements combined tell us that there are two constants $c_1 , c_2 >0$ depending on $d, p, q$ such that
	\begin{equation*}
	0< c_1 < \gamma_\star(r) < c_2 < \infty,
	\end{equation*}
	uniformly in $r$.
\end{proof}

\quad Our next goal is showing that $\hat{\gamma}:= 1 - e^{-\hat{\gamma}_\star}$ is equal to the infinite-volume spectral gap $\gamma$, which leads us to concluding the proof of Theorem \ref{thm2}.

\vspace{3mm}
\textit{Proof of Theorem \ref{thm2}.}  ~ Thanks to Proposition \ref{convprop}, it suffices to verify that $\gamma = \hat{\gamma}.$

\vspace{2mm}
\textbf{STEP 1.} $\gamma \leq \hat{\gamma}$.

\quad In order to prove $\gamma \leq \hat{\gamma}$, we need a good control on the following quantity which is just the $L^1$-variant of $\textbf{m}_t$ defined in (\ref{mtdef}):
\begin{equation}\label{mtstardef}
\textbf{m}_t^* := \max_{x_0 \,\in \, \Omega_{\mathbb{Z}_r^d}} \left\| \mathbb{P}_{x_0} \left(X^{\dagger}_t (\Lambda) \in \, \cdot \, \right) - \pi^{\dagger}_{\Lambda} \right\|_{\textnormal{\tiny{TV}}},
\end{equation}
where $r:= 3\log^5 n$, $\Lambda \subset \mathbb{Z}_r^d$ is a sub-cube of side-length $2\log^5 n$, and $X_t^\dagger$ is the Swendsen-Wang dynamics on $\mathbb{Z}_r^d$ with the stationary distribution $\pi^\dagger$, as defined in Lemma \ref{mtestim}. It turns out that $\textbf{m}_t^*$ has a lower bound which resembles that of $\textbf{m}_t$ in Lemma \ref{mtestim}.
This is shown by the following lemma whose proof is presented in Appendix \ref{secmtstarbd}.

\begin{lemma}\label{mtstarbd}
	For every $t>0$, $\textbf{m}_t^*$ satisfies the following inequality:
	\begin{equation}\label{mtstarineq2}
	\textbf{m}_t^* \geq e^{-\gamma_\star(r) t \;- \;15d \gamma_\star(r)\log\log n } - n^{-9d}.
	\end{equation}
\end{lemma}

\quad Let us pick $x_0^\dagger \in \Omega_\Lambda$ and $A \subset \Omega_\Lambda$ which achieve the maximum $\textbf{m}_t^*$, i.e.,
\begin{equation*}
\mathbb{P}_{x_0^\dagger} \left( X_t^\dagger (\Lambda) \in A \right) - \pi_\Lambda^\dagger (A) = \textbf{m}_t^*.
\end{equation*}

\noindent Let $Z_t$ be the Swendsen-Wang dynamics on the infinite-volume lattice $\mathbb{Z}^d$, and let $\pi^\infty$ denote its stationary distribution. Define $\Lambda^+ := \{v: \textnormal{dist}(v,\Lambda) \leq \frac{1}{3}\log^5 n  \} \subset \mathbb{Z}_r^d$, and let $\psi : \Lambda^+ \rightarrow \mathbb{Z}^d$ be a graph homomorphism that maps $\Lambda^+$ onto its isomorphic copy in $\mathbb{Z}^d$ whose center  is located at  the origin.

\quad Let $Z_0' \sim \pi^\infty$ denote a random configuration on $\mathbb{Z}^d$ distributed according to $\pi^\infty$, and let $Z_0\in \Omega_{\mathbb{Z}^d}$ be defined as follows:
\begin{equation}\label{z0def}
Z_0(v) =
\begin{cases}
x_0^\dagger (\psi^{-1}(v)), & ~~\textnormal{if } ~v\in \psi(\Lambda^+);\\
Z_0'(v),  & ~~\textnormal{otherwise}.
\end{cases}
\end{equation}

\noindent Set $s_0= 7d\log\log n$. In order to deduce the desired conclusion $\gamma
\leq \hat{\gamma}$, we control the $L^2$-distance between $\pi^\infty$ and the law of $Z_t$ with starting configuration $Z_0$.  We begin with the following inequality (cf. (\ref{scl2ineq})).
\begin{equation}\label{ztineq1}
e^{-\gamma_\star t}  \left\| \mathbb{P}_{Z_0} \left(Z_{s_0} \in \cdot\: \right) - \pi^\infty  \right\|_{L^2 (\pi^\infty)}
\geq
\left\| \mathbb{P}_{Z_0} \left(Z_{t+s_0} \in \cdot\: \right) - \pi^\infty  \right\|_{L^2 (\pi^\infty)},
\end{equation}
where $\gamma_\star $ satisfies $1-\gamma = e^{-\gamma_\star}$. By Cauchy-Schwarz, we have
\begin{equation}\label{ztineq2}
\begin{split}
\left\| \mathbb{P}_{Z_0} \left(Z_{t+s_0} \in \:\cdot\: \right) - \pi^\infty  \right\|_{L^2 (\pi^\infty)}
&\geq
\left\| \mathbb{P}_{Z_0} \left(Z_{t+s_0} \in \:\cdot\: \right) - \pi^\infty  \right\|_{\textnormal{\tiny{TV}}} \\
&\geq
\mathbb{P}_{Z_0} \left(Z_{t+s_0}(\Lambda) \in A \right) - \pi^\infty_\Lambda (A)  ,
\end{split}
\end{equation}
where we wrote $\Lambda$ (resp. $A$) instead of $\psi(\Lambda)$ (resp. $A \circ \psi^{-1}$) for convenience.


\quad Consider the coupling between $Z_t$ and $X_t^\dagger$ such that  the update sequence of $Z_t$ is given by the induced update sequence of $X_t^\dagger$ on $\Lambda^+$ translated by $\psi$. In particular, the update sequences of $Z_t$ and $X_t^+$ coincide on $\Lambda^+$ modulo $\psi$.
Under such coupling, an analogous argument as Lemma \ref{speedofpropa}  implies that
\begin{equation}\label{ztineq3}
\mathbb{P} \left(Z_{t+s_0} (\Lambda) \neq X_{t+s_0}^\dagger(\Lambda) \right) \leq n^{-9d} ~~~\textnormal{for all }~ t<\log^{3} n,
\end{equation}
since the two chains starts with the same initial configuration on $\Lambda^+$. Moreover, the weak spatial mixing property of the Potts measure at high enough temperature (see, e.g., \cite{m}) gives that
\begin{equation}\label{ztineq4}
\left\| \pi^\infty |_\Lambda - \pi^\dagger |_\Lambda  \right\|_{\textnormal{\tiny{TV}}} \leq
e^{-c \log^{3} n} \leq n^{-9d},
\end{equation}
where  $c$ is a positive constant depending on $d,\, p$.
By combining the four inequalities (\ref{ztineq1}--\ref{ztineq4}) we deduce that
\begin{equation}\label{ztineq5}
\begin{split}
e^{-\gamma_\star t}  \left\| \mathbb{P}_{Z_0} \left(Z_{s_0} \in \:\cdot\: \right) - \pi^\infty  \right\|_{L^2 (\pi^\infty)}
&\geq
\mathbb{P}_{x_0^\dagger} \left(X_{t+s_0}^\dagger(\Lambda) \in A \right) - \pi^\dagger_\Lambda (A) - 2n^{-9d}\\
&=
\textbf{m}_{t+s_0}^* - 2n^{-9d}\\
&\geq e^{-\gamma_\star(r) (t+s_0) \;- \;15d \gamma_\star(r)\log\log n } - 3n^{-9d},
\end{split}
\end{equation}
where the last inequality is due to Lemma \ref{mtstarbd}. We now upper bound the l.h.s. of (\ref{ztineq5}) by the following theorem which can be understood as an infinite-volume analogue of Theorem \ref{l2thm}.

\begin{theorem}\label{l2thminfty}
	Let $Z_0' \sim \pi^\infty$, let $Z_0$ be defined as (\ref{z0def}), and let $(Z_t)$ be the Swendsen-Wang dyamics on $\mathbb{Z}^d$ with initial configuration $Z_0$. Then there exists $p_0'=p_0'(d)>0$ such that for any $0<p<p_0'$ and $s_0:= 7d\log\log n$, we have
	\begin{equation*}
	\left\| \mathbb{P}_{Z_0} \left(Z_{s_0} \in \:\cdot\: \right) - \pi^\infty  \right\|_{L^2 (\pi^\infty)}
	\leq
	2,
	\end{equation*}
	where $\pi^\infty$ denotes the infinite-volume Potts measure, i.e., the stationary distribution for $(Z_t)$.
\end{theorem}

\quad One can prove Theorem \ref{l2thminfty} by implementing the information percolation framework in the infinite-volume domain $\mathbb{Z}^d$. This is done similarly as in Theorem \ref{l2thm}, while some difficulties arise due to the infinite nature of the domain $\mathbb{Z}^d$. We discuss the details in Appendix \ref{secl2thminfty}.

\quad Implementing Theorem \ref{l2thminfty}, the equation (\ref{ztineq5}) implies that
\begin{equation}\label{ztineq6}
3^{1/t}e^{-\gamma_\star }
\geq \left[e^{-\gamma_\star(r) (t+s_0) \;- \;15d \gamma_\star(r)\log\log n } - 3n^{-9d}\right]^{1/t} ,
\end{equation}

\noindent which holds for all $t<\log^{3} n $. Then, substituting $t=\log^{1/2}n$ and letting $n\rightarrow \infty$ gives that
\begin{equation*}
e^{-\gamma_\star} \geq e^{-\hat{\gamma_\star}},
\end{equation*}
since $ e^{-\gamma_\star(r) (\log^{1/2} n +12d\log\log n)} \gg 3n^{-9d}$. Thus we deduce the desired conclusion $\gamma \leq \hat{\gamma}$.

\vspace{3mm}
\textbf{STEP 2.} $\gamma \geq \hat{\gamma}$.

\quad The second part is shown by utilizing the variational characterization (\ref{specgapvar}) of the spectral gap. To this end, we first note the fact that the transition matrix of the Swendsen-Wang dynamics on any finite graph is non-negative definite (e.g., Remark 4.4 of \cite{ullrichphd}). This naturally extends to the  infinite-volume dynamics, implying that the transition kernel is non-negative definite.

\quad Using the variational characterization, write $\gamma$ as
\begin{equation*}
\gamma = \inf_{\substack{f \in L^2(\pi^\infty) \\ f \neq 0 }} \frac{\mathcal{E}_\infty (f,f)}{\textnormal{Var}_\infty (f)},
\end{equation*}

\noindent where $\mathcal{E}_\infty$ and $\textnormal{Var}_\infty$ denote the Dirichlet form and the variance in terms of $\pi^\infty$, respectively. For any $f\in L^2 (\pi^\infty)$, we can pick a sequence $\{f_n\}$ of finitely supported (i.e., the value of $f_n$ depends only on spins at finitely many sites) $L^2(\pi^\infty)$-functions such that $f_n \rightarrow f$ in $L^2(\pi^\infty)$.   In this case, we also have that
\begin{equation*}
\mathcal{E}_\infty (f_n  , f_n) \rightarrow \mathcal{E}_\infty (f,f) ~~\textnormal{and}~~ \textnormal{Var}_\infty (f_n) \rightarrow \textnormal{Var}_\infty (f),
\end{equation*}
as $n\rightarrow \infty$. Therefore, for any $\epsilon >0$, we can pick a finitely supported $g\in L^2(\pi^\infty)$ such that
\begin{equation*}
\gamma + \frac{\epsilon}{2} \geq
\frac{\mathcal{E}_\infty (g,g)}{\textnormal{Var}_\infty (g)}.
\end{equation*}

\noindent Then, due to the convergence of Gibbs measures as the underlying volume tends to infinity, there exists $M>0$ such that for all $m \geq M$,
\begin{equation}\label{gammapluse}
\gamma + \epsilon \geq \frac{\mathcal{E}_m (g,g)}{\textnormal{Var}_m (g)},
\end{equation}
where  $\mathcal{E}_m$ and $\textnormal{Var}_m$ denote the Dirichlet form and the variance in terms of the stationary distribution on $\mathbb{Z}_m^d$, respectively. Since the r.h.s. of (\ref{gammapluse}) is greater than or equal to $\gamma(m)$, we obtain that
\begin{equation*}
\gamma + \epsilon \geq \hat{\gamma} ,
\end{equation*}
as $m$ tends to infinity. This  holds for all $\epsilon >0$, so we deduce that $\gamma \geq \hat{\gamma}$. \qed

\begin{remark}\label{rmkspecgapgen}
	In the proof of Theorem \ref{thm2}, we did not use any property specific to the Swendsen-Wang dynamics, and hence the theorem can be generalized to other types of Markov chains on spin systems. For instance, our method yields the same result for the Potts Glauber dynamics, implying that the cutoff location of the Potts Glauber dynamics in Theorem 3 of \cite{ls14} can be written in terms of the infinite-volume spectral gap. 
	
	In general, one can show that by following the proof of Theorem \ref{thm2}, the finite-volume spectral gap converges to the infinite-volume gap if the Markov chain has the following properties:
	\begin{itemize}
		\item Information does not spread too fast. (Analogue of Lemma \ref{speedofpropa})
		
		\item Dependence on the initial condition wears off quickly enough to deduce an exponential decay of the $L^2$-distance from stationarity. (Analogue of Theorems \ref{l2thm} and \ref{l2thminfty})
	\end{itemize}
\end{remark}

\quad We finally conclude the proof of Theorem \ref{thm1}, which comes as a direct consequence of Theorem \ref{cutoffthm} and Proposition \ref{convprop}. \vspace{3mm}

\textit{Proof of Theorem \ref{thm1}.}\quad By Theorem \ref{cutoffthm}, the Swendsen-Wang dynamics on $\mathbb{Z}_n^d$ has cutoff at
\begin{equation*}
t_\star = \frac{d}{2\gamma_\star(r)} \log n,
\end{equation*}
where $r = 3\log^5 n$ and with $O(\log\log n)$-window. Note that by Proposition \ref{convprop}, we have
\begin{equation*}
\begin{split}
\left| \frac{d}{2\gamma_\star(r)} \log n \: -
\:\frac{d}{2\gamma_\star} \log n
\right|
&\leq
\frac{d}{2 \gamma_\star(r) \; \gamma_\star} |\gamma_\star(r) - \gamma_\star| \log n \\
&\leq
\frac{2d}{ \gamma_\star(r) \; \gamma_\star}  \log^{-1/4 + \delta} n ~ =~ o(1),
\end{split}
\end{equation*}
as one sets $\delta$ to satisfy $\delta < 1/4$, where $\gamma_\star $ is given by $1-\gamma = e^{-\gamma_\star}$.  Therefore, the cutoff locations stated in Theorems \ref{thm1} and \ref{cutoffthm} coincide with the same $O(\log \log n)$-window, and this concludes the proof of Theorem \ref{thm1}.
\qed

\section{Appendix}

\subsection{Upper bound on the Spectral Gap}\label{secupperbd}

We establish an upper bound on the spectral gap of the Swendsen-Wang dynamics, which has been assumed in proving the main theorems. Our approach is to investigate the edge Swendsen-Wang dynamics (see \S2.2) instead of the original one.  As mentioned at the end of \S2.1, we impose that the parameters $p$ and $\beta$ to satisfy the equation $p=1- e^{-\beta}$.

\quad We begin with a lemma that sheds light on the relationship between the edge Swendsen-Wang dynamics and the original chain.

\begin{lemma}\label{ullrichlem}
	Let $G=(V,E)$ be a graph. Let $\gamma$ (resp. $\widetilde{\gamma}$) be the spectral gap  the Swendsen-Wang dynamics (resp. edge SW dynamics) defined on G. Then we have
	\begin{equation*}
	\gamma ~ =  ~ \widetilde{\gamma}.
	\end{equation*}
\end{lemma}

\quad The main idea behind this lemma is the similarity between the transition matrices of the two Markov chains. For a proof, see e.g., Lemma 2.6 of \cite{ullrichphd}.

\quad Let $(\omega_t^1)$ and $(\omega_t^0)$ be the two copies of the edge Swendsen-Wang dynamics on $(\mathbb{Z}/n\mathbb{Z})^d$ with initial configurations $\omega_0^1 \equiv 1$ and $\omega_0^0 \equiv 0$, respectively. It is well-known that there exists a coupling between the two that satisfies
\begin{equation}\label{optimalcoupeq}
\| \mathbb{P}(\omega_t^1 \in \; \cdot \;) - \mathbb{P}(\omega_t^0 \in \; \cdot \;) \|_{\textnormal{\tiny{TV}}}
=
\mathbb{P} ( \omega_t^1 \neq \omega_t^0).
\end{equation}

\quad We take such an optimal coupling $(\omega_t^1, \omega_t^0)$. (For more explanation on coupling inequality and optimal coupling, see e.g., \cite{lpwmcmt}.) Investigating this pair, we show that the spectral gap of the dynamics lies strictly away from $1$ uniformly in $n$.

\begin{proposition}\label{specgapprop}
	Consider the Swendsen-Wang dynamics on the $d$-dimensional torus $(\mathbb{Z}/n\mathbb{Z})^d$. Then  for every $0<p<(2d)^{-5/2}$, the spectral gap ${\gamma}(n)$ of the process satisfies
	\begin{equation*}
	{\gamma}(n)~ \leq ~1 - p\left(1 - \frac{1}{q} -\frac{2dp^2}{q} \right) ~<~1~~~~\textnormal{for all}~n.
	\end{equation*}
\end{proposition}

\begin{proof}
	Consider an optimal coupling $(\omega_t^1, \omega_t^0)$ of the edge Swendsen-Wang dynamics on $\mathbb{Z}_n^d$ that satisfies (\ref{optimalcoupeq}), each starting from the all-open and all-closed configuration, respectively. Pick any vertex $u \in \mathbb{Z}_n^d$ and one of its neighbor $v$, and let $\{ u \overset{\omega_t^0}{\longleftrightarrow} v \}$ denote the event that there exists an open path in $\omega_t^0$ that connects $u$ and $v$. Then for each $t$, our goal is to derive a lower bound of the following probability:
	\begin{equation}\label{goalprob}
	\mathbb{P} \left( \left. \omega_{t+1}^1(e) = 1 , \:  u \overset{\omega_{t+1}^0}{\centernot\longleftrightarrow} v \:\; \right|\: \omega_{t}^1(e)=1, \: u \overset{\omega_{t}^0}{\centernot\longleftrightarrow} v  \right),
	\end{equation}
	where $e$ denotes the edge $(uv)$.

	\quad We start with an observation which is clear by the definition of our chain.
	\begin{equation}\label{goalprob1}
	\mathbb{P} \left( \omega_{t+1}^1 (e)=1 \; \mid \omega_t^1 (e)=1 \right) = p.
	\end{equation}
	
	\noindent On the other hand, conditioned on $u \overset{\omega_{t}^0}{\centernot\longleftrightarrow} v $, we have $u \overset{\omega_{t+1}^0}{\longleftrightarrow} v $ if and only if the open clusters of $u$ and $v$ are assigned with the same color and the percolation configuration at time $t+1$ connects the pair of vertices. The probability of the latter event can be bounded by a rough estimate as follows: for $\omega \sim \textnormal{Perc}(\mathbb{Z}_n^d,p)$,
	\begin{equation*}
	\mathbb{P} \left(u \overset{\omega}{\longleftrightarrow} v \right)
	\leq p + (2d-2)p^3 + \sum_{l\geq 5} (2dp)^l \leq p+ 2dp^3.
	\end{equation*}
	For the first inequality, we used the fact that the number of length-one and length-three paths between $u$ and $v$ are $1$ and $2d-2$, respectively, and the number of length-$l$ paths is bounded by $(2d)^l$. The second inequality holds for all $p$ such that $p<(2d)^{-5/2}$. Since $\omega \sim \textnormal{Perc}(\mathbb{Z}_n^d,p)$ stochastically dominates $\omega_{t+1}^0 $ in a natural way, we have
	\begin{equation}\label{goalprob2}
	\mathbb{P} \left( \left.  u \overset{\omega_{t+1}^0}{\longleftrightarrow} v \:\; \right|\:u \overset{\omega_{t}^0}{\centernot\longleftrightarrow} v  \right)
	\leq \frac{1}{q} (p+ 2dp^3).
	\end{equation}
	
	\quad Now we derive a lower bound on (\ref{goalprob}) using (\ref{goalprob1}) and (\ref{goalprob2}).
	\begin{equation*}
	\begin{split}
	\mathbb{P} &\left( \left. \omega_{t+1}^1(e) = 1 , \:  u \overset{\omega_{t+1}^0}{\centernot\longleftrightarrow} v \:\; \right|\: \omega_{t}^1(e)=1, \: u \overset{\omega_{t}^0}{\centernot\longleftrightarrow} v  \right)\\
	&~~~	\geq
	~	\mathbb{P} \left( \left. \omega_{t+1}^1(e) = 1 \: \right|\: \omega_{t}^1(e)=1 \right)
	~ -~
	\mathbb{P} \left( \left.   u \overset{\omega_{t+1}^0}{\longleftrightarrow} v \:\; \right|\:  u \overset{\omega_{t}^0}{\centernot\longleftrightarrow} v  \right)\\
	&~~~
	\geq~ p - \frac{1}{q} (p+2dp^3) = p\left( 1-\frac{1}{q} -\frac{2dp^2}{q} \right).
	\end{split}
	\end{equation*}
	Then since the starting configurations satisfy $\omega_{0}^1(e)=1$ and $ \: u \overset{\omega_{0}^0}{\centernot\longleftrightarrow} v $, 
	\begin{equation*}
	\mathbb{P}( \omega_t^1 \neq  \omega_t^0) ~\geq ~
	\mathbb{P} \left( \omega_{t}^1(e)=1, \: u \overset{\omega_{t}^0}{\centernot\longleftrightarrow} v \right) ~ \geq ~
	\left\{ p\left( 1-\frac{1}{q} -\frac{2dp^2}{q} \right) \right\}^t,
	\end{equation*}
	and this also gives the bound on the total-variation distance between the law of the two copies as the pair is an optimal coupling. Therefore, by utilizing the second inequality of Proposition \ref{specgap} and by Lemma \ref{ullrichlem},  we obtain that
	\begin{equation*}
	{\gamma}~ \leq ~1 - p\left(1 - \frac{1}{q} -\frac{2dp^2}{q} \right),
	\end{equation*}
	as one tends $t$ to infinity. Morever, it is an estimate that holds uniformly in $n$, hence concluding the proof.
\end{proof}

\quad Although not used in this paper, we can derive an upper bound of the spectral gap for low temperature Swendsen-Wang dynamics using an analogous method as in Proposition \ref{specgapprop}. The following corollary illustrates how it is generalized.

\begin{corollary}
	Consider the Swendsen-Wang dynamics on the $d$-dimensional torus $(\mathbb{Z}/n\mathbb{Z})^d$. If $p>0$ satisfies $p >\frac{1}{q}$, the spectral gap ${\gamma}$ of the given process satisfies
	\begin{equation*}
	{\gamma}~ \leq ~1 - p + \frac{1}{q}.
	\end{equation*}
	In particular, ${\gamma}$ is strictly smaller than $1$, uniformly in $n$.
\end{corollary}

\textit{Proof}. ~The proof is identical as it is done in Proposition \ref{specgapprop}, except for the estimate in (\ref{goalprob2}). In the current case, we use a more obvious estimate:
\begin{equation*}
\pushQED{\qed}
\mathbb{P} \left( \left.  u \overset{\omega_{t+1}^0}{\longleftrightarrow} v \:\; \right|\:u \overset{\omega_{t}^0}{\centernot\longleftrightarrow} v  \right)
~\leq~ \frac{1}{q} . \qedhere
\popQED
\end{equation*}

\subsection{Proof of Lemma \ref{mtstarbd}}\label{secmtstarbd}
We first introduce the following simple property of product measures.
\begin{lemma}\label{prodmsrlem2}
	Let $\{\mu_i \}_{i=1}^k$, $\{\nu_i \}_{i=1}^k$ be collections of probability measures on a  state space,  let $\mu = \otimes_{i=1}^k \mu_i$ and $\nu = \otimes_{i=1}^k \nu_i $. Then the following inequality holds true:
	\begin{equation*}
	\|\mu - \nu \|_{\textnormal{\tiny{TV}}}~ \leq~ \sum_{i=1}^k \|\mu_i - \nu_i \|_{\textnormal{\tiny{TV}}} .
	\end{equation*}
\end{lemma}

\begin{proof}
	We can prove this lemma by the following simple observation:
	\begin{equation*}
	\begin{split}
	&\sum_{x_1,\ldots,x_k} \left|\prod_{i=1}^k \mu_i (x_i) - \prod_{i=1}^k \nu_i (x_i) \right|\\
	&\leq
	\sum_{x_1,\ldots,x_k} |\mu_1 (x_1) - \nu_1 (x_1)| \prod_{i=2}^k \mu_i (x_i) +\sum_{x_1,\ldots,x_k} \left| \nu_1 (x_1) \prod_{i=2}^k \mu_i (x_i) - \prod_{i=1}^k \nu_i (x_i) \right|\\
	&= \sum_{x_1} | \mu_1(x_1) - \nu_1(x_1) | ~+  \sum_{x_2,\ldots,x_k} \left| \prod_{i=2}^k \mu_i (x_i) - \prod_{i=2}^k \nu_i (x_i) \right|
	\end{split}
	\end{equation*}
	Iterating this for $k$ times draws the conclusion.
\end{proof}

\quad Therefore, by following the proof of Theorem \ref{l1l2thm} along with the application of the previous lemma to (\ref{tvineq9}), we obtain that
\begin{equation*}
\max_{x_0} \left\| \mathbb{P}_{x_0} \left( X_{t+s} \in \,\cdot \,\right) - \pi \right\|_{{ \textnormal{\tiny TV}}}
\leq
\left( m^d \wedge (n/\log^7 n)^d \right) \textbf{m}^*_t  +n^{-9d},
\end{equation*}
where $X_t$ is the Swendsen-Wang dynamics on $\mathbb{Z}_m^d$ with $\log^5 n \leq m \leq n$ and $s =11dp_\star \log\log n $. Then, as what we did in (\ref{mtineq11}), pick $m=r:= 3\log^5 n$ and deduce that
\begin{equation}\label{mtstarineq1}
e^{-\gamma_\star(r) (t+s)}
\leq
2\| \mathbb{P} (X_{t+s}^\dagger \in \; \cdot \;) - \pi^\dagger \|_{\textnormal{\tiny{TV}}}
\leq
2r^d \textbf{m}_t^* +2n^{-9d},
\end{equation}
and hence
\begin{equation*}
\pushQED{\qed}
\textbf{m}_t^* \geq e^{-\gamma_\star(r) t \;- \;15d \gamma_\star(r)\log\log n } - n^{-9d}. \qedhere
\popQED
\end{equation*}

\subsection{Proof of Theorem \ref{l2thminfty}}\label{secl2thminfty}

\subsubsection{Information percolation on $\mathbb{Z}^d$}\label{secl2thminfty1} Let $(Z_t)_{t=0}^{t_\star}$ denote the Swendsen-Wang dynamics on $\mathbb{Z}^d$ with initial configuration $Z_0$ defined in (\ref{z0def}), and let $\pi^\infty$ denote its stationary distribution. Also, let $\Lambda$, $\Lambda^+$ and $\psi$ be defined as in (\ref{mtstardef}, \ref{z0def}). Note that the boxes $\Lambda$, $\Lambda^+$ are centered at the origin.
Throughout this section, we write $\Lambda$ (resp. $\Lambda^+$) instead of $\psi(\Lambda)$ (resp. $\psi(\Lambda^+)$) if there is no ambituity.

\quad We choose an analogous approach as  in Theorem \ref{l2thm}, while adjusting the argument to the infinite-volume setting. 
We draw the history diagram on  the space-time slab $\mathbb{Z}^d \times [0, t_\star]$, but use a modified definition of classifying the information percolation clusters.

\quad Let $\widetilde{\mathfrak{H}}_{t_\star}$ denote the update sequence of $(Z_t)_{t=0}^{t_\star}$ and let $\widetilde{\mathscr{H}}_v$ be the update history of $v\in \mathbb{Z}^d$ defined as in Definition \ref{historydiag}. We adopt the same graph structure $ u \sim_i v$ on $\mathbb{Z}^d$ as in Definition \ref{graphstruc}: $u \sim_i v$ if and only if $\widetilde{\mathscr{H}}_u(t+ \frac{1}{2}) \cap  \widetilde{\mathscr{H}}_v(t+\frac{1}{2})  \neq \emptyset$ for some integer $t$.

\begin{definition}[Information percolation clusters for $Z_t$]
	Let $C \subset \mathbb{Z}^d$ be a connected component in the graph $(\mathbb{Z}^d , \sim_i)$. Then,
	\begin{enumerate}
		\item $C$ is marked \textsc{Red} if $\widetilde{\mathscr{H}}_C (0) \cap \Lambda^+ \neq \emptyset$;
		
		\item $C$ is marked \textsc{Blue} if $\widetilde{\mathscr{H}}_C (0) \cap \Lambda^+ = \emptyset$ and $|C|=1$;
		
		\item $C$ is marked \textsc{Green} if otherwise, i.e., if $\widetilde{\mathscr{H}}_C (0) \cap \Lambda^+ = \emptyset$ and $|C| \geq 2$.
	\end{enumerate}
	
\end{definition}

\quad Intuition for this modified definition is straightforward: if we consider the stationary chain $Z_t'$  with the initial configuration $Z_0'$, then $Z_t (C) = Z_t' (C)$ as long as $\widetilde{\mathscr{H}}_C (0) \cap \Lambda^+ = \emptyset$, since the starting configurations $Z_0$ and $Z_0'$ can possibly differ only on $\Lambda^+$. Therefore, even if the history survives until $t=0$, the two chains are still coupled if it does not intersect with $\Lambda^+$ at $t=0$.

\quad Set $\widetilde{V}:= \mathbb{Z}^d$, and let $\widetilde{\mathcal{C}}_\mathcal{R}$ denote the collection of red clusters. Let  $\widetilde{V}_\mathcal{R}:= \cup_{C\in \widetilde{\mathcal{C}}_\mathcal{R}} C $ be the union of red clusters, and define $\widetilde{\mathcal{C}}_\mathcal{B}$, $\widetilde{V}_\mathcal{B}$, $\widetilde{\mathcal{C}}_\mathcal{G}$ and $\widetilde{V}_\mathcal{G}$ analogously. 
Moreover, let $\widetilde{\mathscr{H}}_{\mathcal{R}} := \widetilde{\mathscr{H}}_{\widetilde{V}_{\mathcal{R}}}$, and define $\widetilde{\mathscr{H}}_{\mathcal{G}}$, $\widetilde{\mathscr{H}}_{\mathcal{B}}$ similarly.
To introduce the analog of $\Psi_A$ in (\ref{psidef}), let $\widetilde{\mathscr{H}}_A^- := \bigcup \{\widetilde{\mathscr{H}}_v : v\in \widetilde{V} \setminus A \}~ = \widetilde{\mathscr{H}}_{\widetilde{V}\setminus A}  $, and set
\begin{equation}\label{psitildedef}
\widetilde{\Psi}_A := \sup_{\widetilde{\mathscr{H}}_A^-: \: \widetilde{\mathscr{H}}_A^- \in \widetilde{\mathscr{H}}_{com}(A) } \mathbb{P} \left( A \in \widetilde{\mathcal{C}}_{\mathcal{R}} \mid \widetilde{\mathscr{H}}_A^-, \mspace{5mu} \{A \in \widetilde{\mathcal{C}}_{\mathcal{R}  }\} \cup \{A \subset \widetilde{V}_{\mathcal{B}} \} \right),
\end{equation}
where $\widetilde{\mathscr{H}}_A^- \in \widetilde{\mathscr{H}}_{com}(A)$ is the shorthand notation meaning $\widetilde{\mathscr{H}}_A^- \cap \,A \times \{t_\star-\frac{1}{2} \} =\emptyset  $, which imposes a compatibility condition on $\widetilde{\mathscr{H}}_A^-$. Then we have the following infinite-volume analog of Lemma \ref{psilem}, whose proof is presented in the final subsection.

\begin{lemma}\label{psilem2}
	Set $r_0 := \frac{5}{2} \log^5 n$ and let $A \neq \emptyset$ be an arbitrary subset of $\widetilde{V}$. For any $\theta >0$, there exist constants $M=M(\theta)$ and $p_0 = p_0 (\theta, d)$ such that for any $p < p_0$,
	\begin{equation*}
	\widetilde{\Psi}_A \leq M e^{-\theta \mspace{3mu}\mathfrak{M}(A)}  \left[(3edp)^{t_\star -\frac{1}{2}} \wedge e^{-\theta (|| A || - r_0) } \right],
	\end{equation*}
	where $\mathfrak{M}(A)$ is the size of the smallest connected subgraph containing $A$, and $||A||$ is defined by	 	$||A|| := \max_{x\in A} ||x||_\infty$.
\end{lemma}

\quad We introduce one more lemma that restricts our attention to a finite domain. Let $r : = 3 \log^5 n$ and define $R$ to be the following random variable:
\begin{equation}\label{Rdefn}
R := \max \{||v||_\infty : \: v  \in \widetilde{V}_\mathcal{R} \} \vee r.
\end{equation}

\noindent The following lemma indicates that it is unlikely to have $R>r$ at time $s \asymp \log\log n$. Its proof is postponed to \S \ref{subsecRlem}.

\begin{lemma}\label{Rlem}
	Let $(Z_t)_{t=0}^{s}$ be the Swendsen-Wang dynamics on $\widetilde{V}$ defined as above, where $s =\Theta(\log \log n)$. Then at time $s$, the random variable $R$ defined in (\ref{Rdefn}) satisfies
	\begin{equation}
	\mathbb{P}(R>r) \leq n^{-10d}.
	\end{equation} 
\end{lemma}

\quad Keeping Lemmas \ref{psilem2} and \ref{Rlem} in mind, we continue by comparing the $L^2$ distance of the dynamics at time $s \asymp \log\log n$ from its stationarity. Due to Lemma \ref{jensen}, we get
\begin{equation}\label{zscomp1}
\begin{split}
&\| \mathbb{P}_{Z_0}( Z_{s} \in \cdot) - \mathbb{P}_{Z_0'}(Z_{s}' \in \cdot) \|_{L^2(\pi^\infty)}^2 \leq
\mathbb{E} \left[	\left\| \mathbb{P}_{Z_0}( Z_{s} \in \cdot \mspace{3mu} \vert \widetilde{\mathscr{H}}_{\mathcal{G}}) - \mathbb{P}_{Z_0'}(Z_s \in \cdot \mspace{3mu} \vert \widetilde{\mathscr{H}}_{{\mathcal{G}}}) \right\|_{L^2(\widehat{\pi}^\infty)}^2 \right] \\
&~~~~~~~~~~~~~~~=
\lim_{R_0 \rightarrow \infty}\mathbb{E} \left[	\left\| \mathbb{P}_{Z_0}( Z_{s} \in \cdot \mspace{3mu} \vert \widetilde{\mathscr{H}}_{\mathcal{G}}) - \mathbb{P}_{Z_0'}(Z_s \in \cdot \mspace{3mu} \vert \widetilde{\mathscr{H}}_{{\mathcal{G}}}) \right\|_{L^2(\widehat{\pi}^\infty)}^2 \: \mathbbm{1}_{\{R\leq R_0\}}  \right] \\
&~~~~~~~~~~~~~~~\leq
\limsup_{R_0 \rightarrow \infty}\,
\mathbb{E}_R \left[ \sup_{\widetilde{\mathscr{H}}_{\mathcal{G}}} 	\left\| \widehat{\mu}-\widehat{\pi}_0^\infty \right\|_{L^2(\widehat{\pi}_0^\infty)}^2\right] ,
\end{split}
\end{equation}

\noindent where $\widehat{\pi}^\infty$  is the shorthand notation for $\pi^\infty(\: \cdot \: \vert \widetilde{\mathscr{H}}_{\mathcal{G}})$, and the measures $\widehat{\mu}$ and $\widehat{\pi}_0^\infty$ are defined by
\begin{equation*}
\begin{split}
\widehat{\mu}(\cdot)& := \mathbb{P}_{Z_0} \left. \left( Z_s \in \cdot \mspace{3mu} \right\vert \widetilde{\mathscr{H}}_{\mathcal{G}}, \,R\leq R_0 \right);\\
\widehat{\pi}_0^\infty(\cdot)& :=  \mathbb{P}_{Z_0'}\left. \left(Z_s \in \cdot \mspace{3mu} \right\vert \widetilde{\mathscr{H}}_{\mathcal{G}}, \, R \leq R_0\right).
\end{split}
\end{equation*}
Also, $\mathbb{E}_R$ denotes the expectation over the randomness of $R$. The  equation in the third line is due to monotone convergence theorem.

\quad Let $R_0 \geq r$ be a fixed number. Conditioned on both $\widetilde{\mathscr{H}}_\mathcal{G} $ and the event $\{R \leq R_0 \}$, $Z_s$ and $Z_s'$ are coupled on $B_{R_0}^c \cup \widetilde{V}_\mathcal{G}$, where $B_l := \{ v: ||v||_\infty \leq l \}$. Therefore, the integrand inside the r.h.s. of (\ref{zscomp1}) can be written as
\begin{equation}\label{zscomp2}
\begin{split}
\sup_{\widetilde{\mathscr{H}}_{\mathcal{G}}}	\left\|  \widehat{\mu} - \widehat{\pi}_0^\infty \right\|_{L^2(\widehat{\pi}^\infty)}^2 
= \sup_{\widetilde{\mathscr{H}}_{\mathcal{G}}} \left\| \widetilde{\mu} - \widetilde{\pi} \right\|_{L^2(\widetilde{\pi})}^2,
\end{split}
\end{equation}
where $\widetilde{\mu}$ and $\widetilde{\pi}$ are defined as 
\begin{equation}\label{tilmupidef}
\begin{split}
\widetilde{\mu}(\cdot)& := \mathbb{P}_{Z_0} \left. \left( Z_s(B_{R_0}\setminus \widetilde{V}_{\mathcal{G}}) \in \cdot \mspace{3mu} \right\vert \widetilde{\mathscr{H}}_{\mathcal{G}}, \,R\leq R_0 \right);\\
\widetilde{\pi}(\cdot)& :=  \mathbb{P}_{Z_0'}\left. \left(Z_s(B_{R_0}\setminus \widetilde{V}_{\mathcal{G}}) \in \cdot \mspace{3mu} \right\vert \widetilde{\mathscr{H}}_{\mathcal{G}}, \, R \leq R_0\right).
\end{split}
\end{equation}

\noindent We now state an analog of Lemma \ref{uniflem}  that enables us to work with the uniform distribution instead of the complicated measure $\widetilde{\pi}$. Its proof turns out to be similar to that of Lemma \ref{uniflem}, which is postponed to \S\ref{secuniflem2}.
\begin{lemma}\label{uniflem2}
	Let $R_0 \geq r$. For every subset $S \subset B_{R_0}$, define $\nu_S$ to be the uniform distribution on $\{1,\ldots,q\}^S$. Then there exists $p_0' = p_0' (d)$ such that for any $0<p<p_0'$ and $s\geq (6d)\log\log n$, we have the following inequality: conditioned on $\widetilde{\mathscr{H}}_\mathcal{G}$ and $\{R \leq R_0 \}$,
	\begin{equation*}
	\| \widetilde{\mu} - \widetilde{\pi} \|^2_{L^2(\widetilde{\pi})}
	~	\leq ~2\| \widetilde{\mu} - \nu \|^2_{L^2(\nu )} +1,
	\end{equation*}
	where $\nu$ is the shorthand notation for $\nu_{B_{R_0}\setminus \widetilde{V}_\mathcal{G}} $.
\end{lemma}

\noindent Applying this lemma to (\ref{zscomp1}, \ref{zscomp2}), we focus on bounding the following:
\begin{equation*}
\mathbb{E}_R \left[ \sup_{\widetilde{\mathscr{H}}_{{\mathcal{G}}}} \left\| \widetilde{\mu} - \nu  \right\|_{L^2(\nu)}^2  \right].
\end{equation*}

\quad Similarly to \S\ref{secinfoperc2}, we proceed by implementing Lemmas \ref{mplem}, \ref{couplinglem} and Corollary \ref{couplingcor}. Therefore, we obtain first by Lemma \ref{mplem} that
\begin{equation}\label{mpeqinfty2}
\begin{split}
\mathbb{E}_R  \left[ \sup_{\widetilde{\mathscr{H}}_{{\mathcal{G}}}} \left\| \widetilde{\mu} - \nu  \right\|_{L^2(\nu)}^2\right]&\leq
\mathbb{E}_R  \left[ \;\sup_{\widetilde{\mathscr{H}}_\mathcal{G} } \mathbb{E} \left. \left[ q^{|\widetilde{V}_\mathcal{R} \cap \widetilde{V}_{\mathcal{R}'}| }\: \right\vert \widetilde{\mathscr{H}}_\mathcal{G}, \{R\leq R_0 \} \right] \right]-1\\
&\leq
(1+n^{-9}) \:\sup_{\widetilde{\mathscr{H}}_{\mathcal{G}} }  \mathbb{E} \left. \left[ q^{|\widetilde{V}_\mathcal{R} \cap \widetilde{V}_{\mathcal{R}'}| }\: \right\vert \widetilde{\mathscr{H}}_\mathcal{G} \right] -1,
\end{split}
\end{equation}
where the expectation $\mathbb{E}$ is taken over the randomness of $\widetilde{V}_\mathcal{R}$ and $\widetilde{V}_{\mathcal{R}'}$, the i.i.d. copies of red vertices. The second inequality comes from Lemma \ref{Rlem}. Then, Corollary \ref{couplingcor} implies that under the same conditioning,
\begin{equation*}
|\widetilde{V}_\mathcal{R} \cap \widetilde{V}_{\mathcal{R}'}|
\preceq
\sum_{\substack{A \cap A' \neq \emptyset \\ A , A' \subset B_{R_0}\setminus \widetilde{V}_\mathcal{G}}} |A\cup A'| \:\widetilde{J}_{A,A'},
\end{equation*}
where $\{\widetilde{J}_{A,A'}\}$ are the independent indicators such that $\mathbb{P}(\widetilde{J}_{A,A'} =1) = \widetilde{\Psi}_A \widetilde{\Psi}_{A'}$. Thus, the same series of calculations as (\ref{maineq2}, \ref{maineq3}) give  that
\begin{equation}\label{zscomp5}
\begin{split}
\sup_{\widetilde{\mathscr{H}}_\mathcal{G} } \mathbb{E} \left. \left[ q^{|\widetilde{V}_\mathcal{R} \cap \widetilde{V}_{\mathcal{R}'}| }\: \right\vert \widetilde{\mathscr{H}}_\mathcal{G}\right]
&\leq
\exp \left\{ \sum_{v\in B_{R_0}} \left(\sum_{v\in A \subset B_{R_0} } q^{|A|} \widetilde{\Psi}_A  \right)^2  \right\}\\
&\leq
\exp \left\{ \sum_{v\in \widetilde{V}} \left(\sum_{A \ni v } q^{|A|} \widetilde{\Psi}_A  \right)^2  \right\}.
\end{split}
\end{equation}

\noindent We split the summation over $v\in \widetilde{V}$ in the exponent to two parts, $v\in B_{r}$ and $v \in \widetilde{V} \setminus B_{r}$ (Recall that $r= 3\log^5 n$). For the first part, we implement Lemma \ref{psilem2} to deduce that
\begin{equation}\label{zscomp3}
\begin{split}
\sum_{v\in B_{r}} \Big( \sum_{A \ni v} q^{|A|} \widetilde{\Psi}_A \Big)^2
&\leq
M^2 (2r +1)^d \Big( \sum_{k \geq 1}  \sum_{\substack{A \ni v \\ {\mathfrak{M}(A) = k}}} q^k e^{-\theta k} (4edp)^{s-2} \Big)^2\\
&\leq
M^2 (2r +1)^d \Big( \sum_{k \geq 1} (4ed)^k q^k e^{-\theta k} (4edp)^{s-2} \Big)^2  \\
&\leq M^2 (\log n)^{11d}(4edp)^{2s-4} ~<~ \frac{1}{6},
\end{split}
\end{equation}
where we choose large enough $\theta$ that makes the summation over $k$ smaller than $1$, and $p$ is  accordingly small in order to satisfy the conditions of Lemma \ref{psilem2}. Note that in the second inequality, we use the same bound as (\ref{sumpsi}) on the number of $A\ni v$ such that $\mathfrak{M}(A)=k$. The last inequality is obtained if we set $s\geq 6d\log\log n$ with $n$ being large enough.

\quad The second part is derived similarly but utilizing the alternative bound on $\widetilde{\Psi}_A$. Let $\partial B_l$ denote the boundary points of $B_l$, i.e., the points having $l^\infty$-distance exactly $l$ from the origin. Then we obtain that
\begin{equation}\label{zscomp4}
\begin{split}
\sum_{v\in B_{r}^c} &\Big( \sum_{A \ni v} q^{|A| } \widetilde{\Psi}_A  \Big)^2
\leq
M^2 \sum_{l\geq r} \sum_{v\in \partial B_l} \bigg( \sum_{k \geq 1} \sum_{\substack{A \ni v \\ \mathfrak{M}(A) =k}} q^k e^{-\theta k}   \bigg)^2 e^{-2\theta(l-r_0) } \\
&~\leq
M^2 \sum_{l\geq r} 2d (2l+1)^{d-1} e^{-\theta l /3} \bigg( \sum_{k \geq 1} (4e^{-\theta +1}dq)^k 
\bigg)^2 \leq
\sum_{l\geq r} e^{-\theta l/6}
~<~ \frac{1}{6},
\end{split}
\end{equation}
where $\theta, \:p$ are chosen to be large and small respectively as in (\ref{zscomp3}), and the inequalities in the last line holds true for all $n$ sufficiently large.

\quad Thus, combining the equations (\ref{mpeqinfty2}--\ref{zscomp4}) give us that
\begin{equation*}
\mathbb{E}_R  \left[ \sup_{\widetilde{\mathscr{H}}_{{\mathcal{G}}}} \left\| \widetilde{\mu} - \nu  \right\|_{L^2(\nu)}^2 \right]
\leq
(1+n^{-9}) \sup_{\widetilde{\mathscr{H}}_\mathcal{G} } \mathbb{E} \left. \left[ q^{|\widetilde{V}_\mathcal{R} \cap \widetilde{V}_{\mathcal{R}'}| } \right\vert \widetilde{\mathscr{H}}_\mathcal{G} \right]
-1\leq
e^{1/3} -1.
\end{equation*}

\noindent Therefore, the equation (\ref{zscomp1}) and Lemma \ref{uniflem2} imply that
\begin{equation*}
\begin{split}
&\| \mathbb{P}_{Z_0}( Z_{s} \in \cdot) - \mathbb{P}_{Z_0'}(Z_{s}' \in \cdot) \|_{L^2(\pi^\infty)}^2\leq
\lim_{R_0 \rightarrow \infty } \mathbb{E}_R \left[
\sup_{\widetilde{\mathscr{H}}_{{\mathcal{G}}}} \left\|
\widetilde{\mu} - \widetilde{\pi}  \right\|^2_{L^2(\widetilde{\pi})}  \; \mathbbm{1}_{\{R\leq R_0\} } \right]
\leq ~ 2,
\end{split}
\end{equation*}
concluding the proof of Theorem \ref{l2thminfty}. \qed

\subsubsection{Proof of Lemma \ref{uniflem2}}\label{secuniflem2}

Recall the definitions of the measures $\widetilde{\mu}, \widetilde{\pi}$ in (\ref{tilmupidef}). Note that the $L^2$-distance between $\widetilde{\mu}$ and $\widetilde{\nu}$ can be written as
\begin{equation*}
\| \widetilde{\mu} - \widetilde{\pi}\|^2_{L^2 (\widetilde{\pi})} =
\sum_{x} \frac{\widetilde{\mu}(x)^2}{\widetilde{\pi}(x)} ~-1	
=\sum_{x} \frac{\widetilde{\mu}(x)^2}{{\nu}(x)} \frac{\nu(x)}{\widetilde{\pi}(x)} -1,	
\end{equation*}
where the sum is taken over
$x \in \Omega_{B_{R_0} \setminus \widetilde{V}_\mathcal{G} }$. Thus, it suffices to show that
\begin{equation*}
\nu(x) \leq 2\widetilde{\pi}(x)
\end{equation*}
for all $x$. Analogously as in the proof of Lemma \ref{uniflem}, we can interpret $\widetilde{\pi}$ as follows: we first sample a subset $S \subset B_{R_0} \setminus \widetilde{V}_\mathcal{G} $ of red vertices (we denote this probability as $\widetilde{\eta}(S)$), generate the configuration on $S$ via some law $\widetilde{\varphi}_S $, and sample the configuration on $B_{R_0} \setminus (\widetilde{V}_\mathcal{G} \cup S)$ according to the uniform distribution. In other words, we have
\begin{equation*}
\widetilde{\pi}(x) = \sum_{S \subset B_{R_0} \setminus \widetilde{V}_\mathcal{G}} \widetilde{\eta} (S) \widetilde{\varphi}_S (x_S) q^{-|S^-|},
\end{equation*}
where $S^- := B_{R_0} \setminus (\widetilde{V}_\mathcal{G} \cup S)$. In particular, we get
\begin{equation*}
\widetilde{\pi}(x) \geq \widetilde{\eta}(\emptyset) \nu(x),
\end{equation*}
and hence it suffices to verify that $\widetilde{\eta}(\emptyset) \geq 1/2$.

\quad Note that we can obtain the following similarly as Claim \ref{boundeta}:
\begin{equation}\label{etatilde1}
\begin{split}
\widetilde{\eta}(\{\emptyset\}^c) 
&\leq
\sum_{A \neq \emptyset} \mathbb{P}(A =  \widetilde{V}_\mathcal{R} \,\vert\, \widetilde{\mathscr{H}}_\mathcal{G}, \, R\leq R_0 )\\
&\leq
(1+n^{-9d})\sum_{A \neq \emptyset} \mathbb{P}(A\in \widetilde{\mathcal{C}}_\mathcal{R} \,\vert\, \widetilde{\mathscr{H}}_\mathcal{G} )\leq
(1+n^{-9d}) \sum_{v\in B_{R_0}} \sum_{A \ni v} \widetilde{\Psi}_A,
\end{split}
\end{equation}
where the second inequality is due to Lemma \ref{Rlem}. We split the above sum into $v \in B_r$ and $v\in \widetilde{V} \setminus B_r$. Following the same series of calculations  as (\ref{zscomp3}, \ref{zscomp4}), we deduce that 
\begin{equation}\label{etatilde2}
\begin{split}
\sum_{v\in B_{R_0}} \sum_{A \ni v} \widetilde{\Psi}_A 
~\leq~
M^2 (\log n)^{11d} (4edp)^{2s-4} + \frac{1}{6}.
\end{split}
\end{equation}
Thus, (\ref{etatilde1}, \ref{etatilde2}) with $s\geq (6d) \log\log n $ imply that
$\widetilde{\eta} (\{\emptyset|\}^c) \leq 1/2.$
\qed

\subsubsection{Proof of Lemma \ref{Rlem}}\label{subsecRlem}

Our proof goes similar to Lemma \ref{speedofpropa}, using the subcriticality of percolation.  Recall that $r := 3 \log^5 n$, $||\Lambda^+ || \leq r_0 =: \frac{5}{2}\log^5 n$ and $s =\Theta(\log\log n)$.

\quad If $R> r$, then we have a vertex $v$ such that $||v||=l>r$, and that $\widetilde{\mathscr{H}}_v(0) \cap \Lambda^+ \neq \emptyset$, implying that there exist percolation paths over the time period from $0$ to $s_\star$ that connect $v $  to $\Lambda^+$. Thus, there exists integers $s' \in [0,s)$ and $r' \in (r_0,\, l]$ such that there is an open path in $\bar{\omega}_{s'}$ that connects $B_{r'}$ to $(B_{r''})^c$, where $r'' : = r' + (l -r_0 )/s$. This implies that
\begin{equation*}
\mathbb{P}(\widetilde{\mathscr{H}}_v (0) \cap \Lambda^+ \neq \emptyset) \leq s\, l\, \exp \left(- c \:\frac{l-r_0}{s
} \right),
\end{equation*}
where $c >0$ is a positive constant depending on $p,d$.  Thus, we obtain that
\begin{equation*}
\begin{split}
\mathbb{P} &( R > r) \leq \sum_{v \in \widetilde{V}\setminus B_{r}} \mathbb{P} (\widetilde{\mathscr{H}}_v (0) \cap \Lambda^+ \neq \emptyset)
~	\leq~
\sum_{l> r} \sum_{v \in \partial B_l} s\, l\, e^{- \frac{c(l-r_0)}{s} } \\
&~~~~~~~~~~\leq \sum_{l> r} e^{- \frac{cl}{8s} }
~\leq
~\frac{e^{- \frac{cr}{8s} } }{1-e^{-c/8s}}
~\leq~
\frac{16s}{c} e^{- \frac{cr}{8s} }
~\leq~ n^{-10d},
\end{split}
\end{equation*}
where the inequalities in the last line hold for all sufficiently large $n$, and for the second one from the end we used the fact that $1-e^{-x} \geq x/2$ for $x\in [0,1]$.  \qed

\subsubsection{Proof of Lemma \ref{psilem2}}\label{secl2thminfty3}
Due to Lemma \ref{psilem}, it suffices to show that there exist $M, p_0$ such that for all $0<p<p_0$, we have
\begin{equation}\label{psitildebd}
\widetilde{\Psi}_A \leq M e^{-\theta \mspace{3mu}(\mathfrak{M}(A) + \mspace{3mu}|| A ||-r_0)}.
\end{equation}
Proceeding analogously as in the proof of Lemma \ref{psilem}, we focus on bounding
\begin{equation*}
\mathbb{P} \left(A \in \widetilde{\mathcal{C}}_{\mathcal{R}(A)}^* \right),
\end{equation*}
where $ \widetilde{\mathcal{C}}_{\mathcal{R}(S)}^*$ denotes the collection of red clusters that arise when exposing the joint histories of $S$. Set $W_t = |\widetilde{\mathscr{H}}_A (t_\star - t)|$ for each integer $t\geq0$. By the same argument from the proof of Lemma \ref{psilem}, the number of spatial edges in $\widetilde{\mathscr{H}}_A(t+\frac{1}{2})$ is at least $W_{t_\star - t}$. If  $ A \in \widetilde{\mathcal{C}}_{\mathcal{R}(A)}^*$,  we have
\begin{equation*}
W_1 + W_2 +\, \ldots \,+W_{t_\star} \geq \mathfrak{M}(A) +t_\star -1 \geq \mathfrak{M}(A),
\end{equation*}
since the history of $A$ should spatially connect until $t=0$. Moreover, there should be a space-time path in $\widetilde{\mathscr{H}}_A$ that connect each point of $A$ at time $t_\star$ to a point in $\Lambda^+ $ at time zero. This implies that
\begin{equation*}
W_1 + W_2 +\, \ldots\, W_{t_\star} \geq ||A||-||\Lambda^+|| \geq ||A|| - r_0,
\end{equation*}
where $r_0:= \frac{5}{2}\log^5 n.$ Therefore, we obtain that
\begin{equation}\label{psitildecomp1}
\begin{split}
\mathbb{P}\left(A \in \widetilde{\mathcal{C}}_{\mathcal{R}(A)}^* \right)
&\leq
\mathbb{E} \left[\mathbbm{1}_{ \{W_1 + \ldots + W_{t_\star} \: \geq\: \frac{1}{2} (\mathfrak{M}(A) + ||A|| -r_0) \}}  \right]\\
&\leq
e^{-\lambda (\mathfrak{M}(A) + ||A|| -r_0 )} \mathbb{E} \left[\exp (2\lambda(W_1 + \ldots + W_{t_\star})) \right],
\end{split}
\end{equation}
which holds for all $\lambda >0$, where we used $\mathbbm{1} \{X\geq x\} \leq e^{2\lambda (X-x)}$ to deduce the last line. Proceeding similarly as (\ref{gwbound}) using the Galton-Watson branching process representation (\ref{gwrep}) for $W_t$, we get
\begin{equation*}
\mathbb{E} \left[ \left. e^{3 \lambda W_{t+1}} \right| W_t  \right]
\leq
\bigg[ 1 + \sum_{k \geq 1} e^{3(k+1) \lambda} (edp)^k \bigg]^{W_t}
\leq  e^{2dp e^{6\lambda +1} W_t} 
\leq e^{\lambda W_t},
\end{equation*}
where we picked small enough $p$ such that $2dp \leq \lambda e^{-6\lambda-1}$. Implementing the same argument as (\ref{gwbound}), we deduce that
\begin{equation*}
\begin{split}
\mathbb{E} \left[ e^{2\lambda (W_1 + \ldots +  W_{t_\star })  }  \right]
&\leq
\mathbb{E} \left[ e^{2\lambda (W_1 + \ldots +  W_{t_\star -1})}  e^{\lambda W_{t_\star-1} }  \right] \\
&\leq
\mathbb{E} \left[ e^{2\lambda (W_1 + \ldots +  W_{t_\star -2})}  e^{3\lambda W_{t_\star-1} }  \right],
\end{split}
\end{equation*}
and iterating this inequality gives

\begin{equation}\label{psitildecomp2}
\mathbb{E} \left[ e^{2\lambda (W_1 + \ldots +  W_{t_\star })  }  \right]
\leq
\mathbb{E} \left[e^{3\lambda W_1} \right] \leq
\left(1+ 4dpe^{3\lambda +1} \right)^{W_0} \leq 2^{|A|}.
\end{equation}

\noindent By combining (\ref{psitildecomp1}) and (\ref{psitildecomp2}) , we get
\begin{equation*}
\mathbb{P}\left(A \in \widetilde{\mathcal{C}}_{\mathcal{R}(A)}^* \right)
\leq
e^{-\lambda (||A||-r_0)} e^{-(\lambda -1) \mathfrak{M}(A)}.
\end{equation*}

\noindent Therefore, the equations (\ref{quotientrep}--\ref{numtrans}) from Lemma \ref{psilem} followed by some adjustment of constants imply (\ref{psitildebd}), and hence the desired conclusion.
\qed

\section*{Acknowledgements}
We are grateful to Evita Nestoridi and Insuk Seo for fruitful discussions. 

\bibliographystyle{plain}
\bibliography{swdcutoff}

\vspace{15mm}
\end{document}